\documentclass[10pt]{amsart}
\usepackage{amsmath,amssymb,amsthm}
\usepackage{float}
\usepackage[all,cmtip]{xy}
\usepackage{tikz}
\usetikzlibrary{matrix}
\usetikzlibrary{cd}
\usepackage{subfig}
\usepackage[colorlinks,hyperindex,linkcolor=blue]{hyperref}
\hypersetup{
 pdfauthor = {Francisco Javier Martinez Aguinaga, Alvaro del Pino Gomez},
 pdftitle= {Classification of tangent and transverse knots in bracket-generating distributions},
 pdfsubject = {Differential Geometry, distributions, h-principle, horizontal curves, knots}}

\newcommand{\D}{{\mathbb{D}}}
\newcommand{\R}{{\mathbb{R}}}
\newcommand{\N}{{\mathbb{N}}}

\newcommand{\SD}{{\mathcal{D}}}

\newcommand{\T}{{\mathcal{T}}}

\newcommand{\SU}{{\mathcal{U}}}

\newcommand{\SB}{{\mathcal{B}}}

\newcommand{\FX}{{\mathfrak{X}}}

\newcommand{\ST}{{\mathcal{T}}}

\newcommand{\SL}{{\mathcal{L}}}
\newcommand{\SC}{{\mathcal{C}}}

\newcommand{\NS}{{\mathbb{S}}}

\newcommand{\lift}{{{\mathfrak{lift}}}}
\newcommand{\len}{{{\operatorname{len}}}}

\newcommand{\ev}{{{\operatorname{ev}}}}

\newcommand{\rank}{{{\operatorname{rank}}}}
\newcommand{\Op}{{\mathcal{O}p}}

\newcommand{\Id}{{\operatorname{Id}}}

\newcommand{\std}{{\operatorname{std}}}
\newcommand{\regu}{{\operatorname{r}}}
\newcommand{\epoint}{{\mathfrak{ep}}}
\newcommand{\Var}{{\operatorname{Var}}}

\newcommand{\GL}{{\operatorname{GL}}}
\newcommand{\Diff}{{\operatorname{Diff}}}
\newcommand{\Imm}{{\mathfrak{Imm}}}
\newcommand{\Emb}{{\mathfrak{Emb}}}
\newcommand{\Maps}{{\mathfrak{Maps}}}
\newcommand{\Mon}{{\operatorname{Mon}}}
\newcommand{\Immt}{{\mathfrak{Imm_{\mathcal{T}}}}}
\newcommand{\Embt}{{\mathfrak{Emb_{\mathcal{T}}}}}
\newcommand{\Immtf}{{\mathfrak{Imm}_{\mathcal{T}}^f}}
\newcommand{\Embtf}{{\mathfrak{Emb}_{\mathcal{T}}^f}}
\newcommand{\EmbAT}{{\mathfrak{Emb}_{\mathcal{AT}}}}

\newtheorem{proposition}{Proposition}[section]
\newtheorem{theorem}[proposition]{Theorem}
\newtheorem{definition}[proposition]{Definition}
\newtheorem{lemma}[proposition]{Lemma}

\newtheorem{corollary}[proposition]{Corollary}
\newtheorem{remark}[proposition]{Remark}

\newtheorem{assumption}[proposition]{Assumption}

\makeatletter
\newcommand{\superimpose}[2]{%
  {\ooalign{$#1\@firstoftwo#2$\cr\hfil$#1\@secondoftwo#2$\hfil\cr}}}
\makeatother

\title{Classification of tangent and transverse knots in bracket-generating distributions}
\subjclass[2010]{Primary: 53D10. Secondary: 53D15, 57R17.}

\author{Javier Mart\'inez-Aguinaga}
\address{Universidad Complutense de Madrid, Facultad de Matem\'aticas, Plaza de Ciencias 3, 28040 Madrid; Instituto de Ciencias Matem'aticas CSIC-UAM-UC3M-UCM, C. Nicol'as Cabrera, 13-15, 28049 Madrid, Spain}
\email{frmart02@ucm.es}

\author{\'Alvaro del Pino}
\address{Utrecht University, Department of Mathematics, Budapestlaan 6, 3584CD Utrecht, The Netherlands}
\email{a.delpinogomez@uu.nl}

\begin{document}
\begin{abstract}
Consider a manifold, of dimension greater than $3$, equipped with a bracket-generating distribution. In this article we prove complete $h$-principles for embedded regular horizontal curves and for embedded transverse curves. These results contrast with the 3-dimensional contact case, where the full $h$-principle for transverse/legendrian knots is known not to hold.

We also prove analogous statements for immersions, with no assumptions on the ambient dimension.


\end{abstract}
\maketitle
\setcounter{tocdepth}{1}
\tableofcontents

\section{Introduction}

\subsection{Setup}

A \emph{$q$-distribution} on a smooth manifold $M$ is a smooth section $\SD$ of the Grassmann bundle of $q$-planes. There is a control-theoretic motivation for considering such objects: we may think of $M$ as configuration space and of $\SD$ as the \emph{admissible directions of motion}. Then, a natural question is whether any two points in $M$ can be connected by a \emph{horizontal path}, i.e. a path whose velocity vectors take values in $\SD$. A sufficient condition is given by a classic theorem of Chow \cite{Chow}: any two points in $M$ can be connected if $\SD$ is \emph{bracket-generating}. Being bracket-generating means that any vector in $TM$ can be written as a linear combination of Lie brackets involving sections of $\SD$. That is, Chow's theorem is an \emph{infinitesimal} to \emph{global} statement.

Even though classic proofs of Chow's theorem produce horizontal paths that are piecewise smooth, $C^\infty$-paths can be constructed by suitable smoothing, see \cite[Subsection 1.2.B]{Gro96}. It follows that every homotopy class of loops on $M$ can be represented by a smooth horizontal loop. That is, the inclusion
\[ \iota: \SL(M,\SD) \quad\longrightarrow\quad \SL(M) \]
is a $\pi_0$-surjection. Here $\SL(M)$ is the (unbased) loop space of $M$ (endowed with the $C^\infty$-topology) and $\SL(M,\SD)$ is the subspace of horizontal loops. More recently, Z. Ge \cite{Ge} proved that the analogous inclusion for $H^1$-loops is a weak homotopy equivalence; see also \cite{BL}.

In this paper we consider a variation on this theme, proving classification statements for spaces of horizontal embeddings. Our theorems relate these spaces to their \emph{formal} counterparts (roughly speaking, spaces of smooth embeddings plus some additional homotopical data). Taking care of the embedding condition is rather delicate (as is often the case for $h$-principles of this type) and much of the paper is dedicated to handling it. Analogous classification statements hold for horizontal immersions, with simpler proofs. We also deduce that the map $\iota$ above is a weak homotopy equivalence (i.e. the smooth analogue of Ge's theorem). Lastly, our techniques translate to the setting of embedded/immersed transverse curves, yielding similar classification results.

We now state our theorems. We work under the following assumption:
\begin{assumption} \label{assumption:constantGrowth}
All the bracket-generating distributions we consider in this paper are of constant growth (i.e. the growth vector does not depend on the point). See Subsection \ref{sssec:growthVector}.
\end{assumption}

\subsection{Immersed horizontal curves}

Let us write $\Imm(M,\SD) \subset \SL(M,\SD)$ for the subspace of immersed horizontal loops. In order to study it, we introduce the so-called \emph{scanning map}:
\[ \Imm(M,\SD) \quad\longrightarrow\quad \Imm^f(M,\SD), \]
taking values in the space of formal horizontal immersions
\[ \Imm^f(M,\SD) := \{ (\gamma,F) \, |\, \gamma \in \SL(M), \quad F \in \Mon(T\NS^1,\gamma^*\SD) \}. \]
The question to be addressed is whether the scanning map is a weak homotopy equivalence. The answer is positive if $\SD$ is a contact structure \cite[Section 14.1]{EM} but, for other distributions, the answer may be negative due to the presence of \emph{rigid curves} \cite{BH}.

A horizontal curve is rigid if possesses no $C^\infty$-deformations relative to its endpoints, up to reparametrisation. These curves are isolated and conform exceptional components within the space of all horizontal maps with given boundary conditions. \emph{Rigid loops} also exist. Because of this, the inclusion $\Imm(M,\SD) \to \Imm^f(M,\SD)$ can fail to be bijective at the level of connected components; see \cite[Remark 23]{PP}. Being rigid is the most extreme case of being \emph{singular}. This means that the \emph{endpoint map} of the curve is not submersive, so the curve has fewer deformations than expected; see Subsection \ref{ssec:regularity}.

The subspaces of rigid and singular curves have a \emph{geometric} and not a \emph{topological} nature. By this, we mean that small perturbations of $\SD$ can radically change their homotopy type; see \cite{Mon} or \cite[Theorem 27]{PP}. This motivates us to discard singular curves and focus on $\Imm^\regu(M,\SD)$, the subspace of regular horizontal immersions. In doing so, the subspace that we discard is not too large: Germs of singular horizontal curves were shown to form a subset of infinite codimension among all horizontal germs, first in the analytic case \cite{PS} and then in general \cite{Bho}. Earlier, it had already been observed \cite[Corollary 7]{Hsu} that \emph{regular} (i.e. non-singular) germs are $C^\infty$--generic.

Our first result reads: 
\begin{theorem} \label{thm:immersions}
Let $(M,\SD)$ be a manifold endowed with a bracket-generating distribution. Then, the following inclusion is a weak homotopy equivalence:
\[ \Imm^\regu(M,\SD) \quad\longrightarrow\quad \Imm^f(M,\SD). \]
\end{theorem}
Apart from the aforementioned contact case, in which there are no singular curves, this was already known in the Engel case \cite{PP}.

\begin{corollary}
Let $\SD_0$ and $\SD_1$ be bracket-generating distributions on a manifold $M$, homotopic as subbundles of $TM$. Then, the spaces $\Imm^\regu(M,\SD_0)$ and $\Imm^\regu(M,\SD_1)$ are weakly homotopy equivalent.
\end{corollary}
This follows immediately from Theorem \ref{thm:immersions} and the analogous fact about $\Imm^f(M,\SD_0)$ and $\Imm^f(M,\SD_1)$. It follows that all the data about $\SD$ encoded in $\Imm^\regu(M,\SD)$ is purely formal.

\subsection{Embedded horizontal curves}

We now consider the subspace of embedded horizontal loops $\Emb(M,\SD) \subset \Imm(M,\SD)$, together with its scanning map
\[  \Emb(M,\SD) \quad\longrightarrow\quad \Emb^f(M,\SD), \]
into the space of \emph{formal horizontal embeddings}:
\begin{align*}
\Emb^f(M,\SD) := \left\{\left(\gamma,(F_s)_{s\in[0,1]}\right):\right. & \quad\gamma \in \Emb(M), \quad F_s \in \Mon_{\NS^1}(T\NS^1,\gamma^*TM), \\
                                                                     & \quad\left. F_0 = \gamma', \quad F_1 \in \gamma^*\SD \right\},
\end{align*}
i.e. the homotopy pullback of $\Emb(M)$ and $\Imm^f(M,\SD)$ mapping into $\Imm^f(M)$. Reasoning as above leads us to introduce $\Emb^\regu(M,\SD)$, the subspace of regular horizontal embeddings. Our second (and main) result reads:
\begin{theorem} \label{thm:embeddings}
Let $(M,\SD)$ be a bracket-generating distribution with $\dim(M) \geq 4$. Then, the following inclusion is a weak homotopy equivalence:
\[ \Emb^\regu(M,\SD) \quad\longrightarrow\quad \Emb^f(M\SD). \]
\end{theorem}
Note that the dimensional assumption is sharp, since the result is known to be false in $3$-dimensional Contact Topology \cite{Ben}.

Theorem \ref{thm:embeddings} was already known in the Engel case \cite{CP} and in the higher-dimensional contact setting \cite[p. 128]{EM}. Our arguments differ considerably from both. The proof in \cite{EM} is contact-theoretical in nature, relying on isocontact immersions. The one in \cite{CP} uses the so-called \emph{Geiges projection}, which is particular to the Engel case. The methods in the present paper use instead local charts in which the distribution can be understood as a connection; see Subsection \ref{sec:graphicalModels}. This is reminiscent of the Lagrangian projection in Contact Topology and closely related to methods used in the Geometric Control Theory \cite{Gro96,Mon} (with the added difficulty of tracking the embedding condition).

Much like earlier:
\begin{corollary}
Fix a manifold $M$ with $\dim(M) \geq 4$. Let $\SD_0$ and $\SD_1$ be bracket-generating distributions on $M$, homotopic as subbundles of $TM$. Then, the spaces $\Emb^\regu(M,\SD_0)$ and $\Emb^\regu(M,\SD_1)$ are weakly homotopy equivalent.
\end{corollary}

\subsection{Horizontal loops}

Now we go back to the problem we started with:
\begin{theorem} \label{thm:Ge}
Let $(M,\SD)$ be a manifold endowed with a bracket-generating distribution. Then, the following inclusion is a weak homotopy equivalence:
\[ \SL(M,\SD) \quad\longrightarrow\quad \SL(M). \]
\end{theorem}
This also holds for the (based) loop space $\Omega_p(M)$ and its subspace of horizontal loops $\Omega_p(M,\SD)$, for all $p \in M$. Observe that the statement uses no regularity assumptions. The reason is that singularity issues can be bypassed thanks to what we call the \emph{stopping-trick} (namely, one can slow the parametrisation of a horizontal curve down to zero locally in order to guarantee that enough compactly-supported variations exist). See Subsection \ref{ssec:hPrincipleHorizontalImmersions}.

\subsection{Immersed transverse curves}

The other geometrically interesting notion for curves in bracket--generating distributions is that of transversality. We define $\Immt(M,\SD)$ to be the space of immersed loops that are everywhere transverse to $\SD$. Like in the horizontal setting, one can introduce formal transverse immersions
\begin{align*} 
\Immtf(M,\SD) = \left\{(\gamma,F):\right. & \quad\gamma \in \SL(M), \quad F \in \Mon_{\NS^1}(T\NS^1,\gamma^*(TM/SD)) \},
\end{align*}
and see that there is a scanning map 
\[ \Immt(M,\SD) \quad\longrightarrow\quad \Immtf(M,\SD). \]
Being transverse is an open condition and therefore rigidity/singularity is not a phenomenon we encounter. We prove:
\begin{theorem} \label{thm:TransverseImmersions}
Let $(M,\SD)$ be a manifold endowed with a bracket-generating distribution. Then the inclusion
\[ \Immt(M,\SD) \quad\longrightarrow\quad \Immtf(M,\SD) \]
is a weak homotopy equivalence. 
\end{theorem}
This result is not new. The $h$-principle for smooth immersions (of any dimension!) transverse to analytic bracket-generating distributions was proven in \cite{PS}. The analyticity assumption was later dropped by A. Bhowmick in \cite{Bho}, using Nash-Moser methods. Both articles rely on an argument due to Gromov relating the flexibility of transverse maps to the microflexibility of (micro)regular horizontal curves. The approach in this paper is independent.

Once again, a corollary is that the weak homotopy type of $\Immt(M,\SD)$ depends on $\SD$ only formally.

\subsection{Embedded transverse curves}

Lastly, we address embedded transverse loops $\Embt(M,\SD)$ and their scanning map into the analogous formal space:
\begin{align*}
\Embtf(M,\SD) = \left\{\left(\gamma,(F_s)_{s\in[0,1]}\right):\right. & \quad\gamma \in \Emb(M), \quad F_s \in \Mon_{\NS^1}(T\NS^1,\gamma^*TM), \\
                                                                     & \quad\left. F_0 = \gamma', \quad F_1: T\NS^1\to \gamma^*TM \to \gamma^*(TM/\SD)\text{ is injective }\right\}.
\end{align*}

Our fourth result reads:
\begin{theorem} \label{thm:TransverseEmbeddings}
Let $(M,\SD)$ be a bracket-generating distribution with $\dim(M) \geq 4$. Then the inclusion
\[ \Embt(M,\SD) \quad\longrightarrow\quad \Embt^f(M,\SD) \]
is a weak homotopy equivalence. In particular, $\Embt(M,\SD)$ depends only on the formal class of $\SD$.
\end{theorem}
The dimension condition is sharp, since transverse embeddings into 3-dimensional contact manifolds do not satisfy a complete $h$-principle. Indeed, there are examples of transverse knots that have the same formal invariants but are not transversely isotopic \cite{BM}. Furthermore, Theorem \ref{thm:TransverseEmbeddings} is only interesting in corank $1$. Indeed, it is a classic result \cite[4.6.2]{EM} that closed $n$-dimensional submanifolds transverse to corank $k$ distributions abide by all forms of the $h-$principle if $k>n$.

\subsection{Structure of the paper}

In Section \ref{sec:preliminaries} we recall some standard definitions from the theory of tangent distributions. Basics of $h$-principle and some preliminary results, using the theory of $\varepsilon$-horizontal embeddings, are presented in Section \ref{sec:epsilon}.

In Section \ref{sec:graphicalModels} we introduce the notion of \emph{graphical model}. These are local descriptions in which the distribution is seen as a connection. Many of our arguments take place in such a local setting. Section \ref{sec:microflexibility} contains a series of technical lemmas (that roughly speaking correspond to the ``reduction step'' in our $h$-principles) about manipulating families of curves.

Sections \ref{sec:tangles} and \ref{sec:controllers} contain the main technical ingredients behind the proof, the notions of \emph{tangle} and \emph{controller}. These are models for horizontal curves (or rather, models for their projections to the base of a graphical model) meant to be used to produce a displacement \emph{transverse} to the distribution. They play a role analogous to the \emph{stabilisation} in Contact Topology, except for the fact that they can be introduced through homotopies of embedded horizontal curves. The existence of such a homotopy uses strongly the fact that the ambient dimension is at least 4 (and it is still rather technical to implement).

The $h$-principles for horizontal curves are proven in Section \ref{sec:hPrincipleHorizontal}. The $h$-principles for transverse curves in Section \ref{sec:hPrincipioTransverse}. Along the way we state and prove the appropriate relative versions. We will put all our emphasis on the embedding cases; the other statements (immersions and simply smooth curves) follow from the same arguments with considerable simplifications.

Appendix \ref{Appendix} contains various technical results on commutators of vector fields. These are used often throughout the paper.

\textbf{Acknowledgments:} The authors are thankful to E. Fern\'andez, M. Crainic, F. Presas, and L. Toussaint for their interest in this project. During the development of this work the first author was supported by the ``Programa Predoctoral de Formación de Personal Investigador No Doctor'' scheme funded by the Basque department of education (``Departamento de Educación del Gobierno Vasco''). The second author was funded by the NWO grant 016.Veni.192.013; this grant also funded the visits of the first author to Utrecht.

\section{Preliminaries on distributions} \label{sec:preliminaries}

In this section we recall some of the basic theory of distributions, including the notions of singularity and rigidity for horizontal curves (Subsection \ref{ssec:regularity}). For further details we refer the reader to \cite{GV,Mon}.

\subsection{Differential systems} \label{ssec:diffSystems}

The following definition generalises the notion of distribution:
\begin{definition}
Let $M$ be a smooth manifold. A \textbf{differential system} $\SD$ is a $C^\infty$-submodule of the space of smooth vector fields.
\end{definition}
Given a smooth distribution on $M$, we can construct a differential system by taking its smooth sections. Conversely, a differential system $\SD$ arises from a distribution if the dimension of its pointwise span $\SD(p) \subset T_pM$ is independent of $p \in M$. In this manner, we think of differential systems as singular distributions; we will often abuse notation and use $\SD$ to denote both the distribution and its sections.

\begin{remark}
When $M$ is not compact, it is convenient to impose that $\SD$ satisfies the sheaf condition. The reason is that there may be differential systems that only differ from one another due to their behaviour at infinity; imposing the sheaf condition removes this redundancy. These subtleties will not be relevant for us.
\end{remark}

\subsubsection{Lie flag} \label{sssec:LieFlag}

Let us introduce some terminology. We say that the string $a$, depending on the variable $a$, is a \textbf{formal bracket expression} of length $1$. Similarly, we say that the string $[a_1,a_2]$, depending on the variables $a_1$ and $a_2$, is a formal bracket expression of length $2$. Inductively, we define a formal bracket expression of length $n$ to be a string of the form $[A(a_1,\cdots,a_j),B(a_{j+1},a_n)]$ with $0 < j < n$ and $A$ and $B$ formal bracket expressions of lengths $j$ and $n-j$, respectively.

Given a differential system $\SD$, we define its \textbf{Lie flag} as the sequence of differential systems
\[ \SD_1 \subset \SD_2 \subset \SD_3 \subset \cdots \]
in which $\SD_i$ is the $C^\infty$-span of vector fields of the form $A(v_1,\cdots,v_j)$, $j \leq i$, where the $v_k$ are vector fields in $\SD$ and $A$ is a formal bracket expression of length $j$. As such, $\SD_1 = \SD$.

\subsubsection{Growth vector} \label{sssec:growthVector}

Given a point $p \in M$, one can use the Lie flag to produce a flag of vector spaces:
\[ \SD_1(p) \subset \SD_2(p) \subset \SD_3(p) \subset \cdots  \]
Here $\SD_i(p)$ denotes the span of $\SD_i$ at $p$. This yields a non-decreasing sequence of integers
\[ (\dim(\SD_1(p)), \dim(\SD_2(p)), \dim(\SD_3(p)), \cdots) \]
which in general depends on $p$. This sequence is called the \textbf{growth vector} of $\SD$ at $p$.

If the growth vector does not depend on the point, we will say that the differential system $\SD$ is of \textbf{constant growth}. If this is the case, all the differential systems in the Lie flag arise as spaces of sections of distributions. Some examples of distributions of constant growth are (regular) foliations, contact structures, and Engel structures.

The following notion is central to us:
\begin{definition}
A differential system $(M,\SD)$ is \textbf{bracket-generating} if, for every $p \in M$ and every $v \in T_pM$, there is an integer $m$ such that $v \in \SD_m(p)$. This integer is called the \textbf{step}.
\end{definition}
As stated in Assumption \ref{assumption:constantGrowth}: we henceforth focus on bracket-generating distributions of constant growth.

\subsubsection{Curvature}

Distributions with the same growth vector can have very different local behaviours. We now define another pointwise invariant called the \emph{curvature}.

Fix a point $p \in M$ and a vector $u \in \SD(p)$. A locally-defined vector field $\tilde{u} \in \FX(\Op(p))$ is a \emph{local extension} of $u$ (with respect to $\SD$) if $\tilde{u}(p) = u$ and $\tilde{u}(q) \in \SD(q)$ for every $q\in\Op(p)$. Then:
\begin{definition} \label{def:curvature}
The \textbf{curvature} of $\SD$ is the bundle morphism:
\begin{align*}
\Omega(\SD):  \SD\times\SD & \to T_pM/\SD\\
(u,v) &\mapsto [\tilde{u}, \tilde{v}](p) + \SD_p
\end{align*}
where $\tilde{u}$ and $\tilde{v}$ are local extensions of $u$ and $v$, respectively.
\end{definition}
From now on we will abuse notation and write $[u,v](p)$ for $[\tilde{u},\tilde{v}](p) + \SD_p$ in $T_pM/\SD_p$. The rank of the curvature measures how far the distribution is from being integrable. Indeed, according to Frobenius' theorem, $\SD$ is integrable if and only if the rank is zero. We can then consider further Lie brackets, yielding a collection of bundle morphisms:
\begin{align*}
\Omega_{i,j}: \SD_i \times \SD_j &\to \SD_{i+j}/\SD_{i+j-1}\\
 (u, v) &\mapsto [u,v],
\end{align*}
which we call the higher curvatures.

\subsection{Regularity of horizontal curves} \label{ssec:regularity}

We now recall how the phenomenon of singularity for horizontal curves shows up.

Given a $(M,\SD)$ distribution and a point $p \in M$, we write $\Maps_p([0,1];M,\SD)$ for the space of horizontal maps of the interval $[0,1]$ into $(M,\SD)$ with initial point $\gamma(0) = p$; we endow it with the $C^\infty$-topology.
\begin{definition}
The \textbf{endpoint map} is defined as the evaluation map at $1 \in [0,1]$:
\begin{align*}
\epoint:  \Maps_p([0,1];M,\SD) \quad&\longrightarrow\quad M\\
 \gamma \quad&\mapsto\quad \gamma(1)
\end{align*}
\end{definition}
This map is smooth. If it were submersive, its fibres would be smooth Frechet manifolds consisting of horizontal paths with given endpoints. The issue is that this is not always the case, leading to the conclusion that the fibres may develop singularities in which the tangent space is not well-defined. These singularities are thus horizontal curves that present issues in order to be deformed.

\begin{definition} \label{def:regular}
A curve $\gamma \in \Maps_p([0,1];M,\SD)$ is \textbf{regular} if the endpoint map $\epoint$ is submersive at $\gamma$. Otherwise, a curve is said to be \textbf{singular}.
\end{definition}
Equivalently, regularity means that, given any vector $v \in T_{\gamma(1)}M$, there exists a variation $(\gamma_s)_{s \in (-\varepsilon,\varepsilon)}$ such that 
\[ d_{\gamma} \epoint\left(\dfrac{d}{ds}\gamma_s\right) = \dfrac{d}{ds}(\epoint(\gamma_s)) = v. \]
We denote by $\Var_\gamma$ the space of infinitesimal variations of $\gamma$, endowed with the $C^\infty$-topology. In control theoretic terms, infinitesimal variations are simply sections of the bundle of controls over $\gamma$. A more down-to-earth description, when $\SD$ is a connection, is that $\Var_\gamma$ corresponds to the space of infinitesimal variations of the projection of $\gamma$ (to the base of the bundle). This approach will be used repeatedly in upcoming sections.

For the purposes of this paper, we are interested both in horizontal paths and horizontal loops. Then:
\begin{definition}
A curve $\gamma \in \SL(M,\SD)$ is \textbf{regular} if it is regular as a path (using the quotient map $[0,1] \to \NS^1$ given by a choice of basepoint).
\end{definition}
It is not difficult to see that being regular does not depend on the auxiliary choice of basepoint.

\section{$\varepsilon$-horizontality and $\varepsilon$-transversality} \label{sec:epsilon}

In this section we introduce $\varepsilon$-horizontal curves. These are curves that form an angle of at most $\varepsilon$ with the distribution and thus serve as approximations of horizontal curves. They provide a convenient starting point for the $h$-principle arguments that will appear later in the paper.

In Subsection \ref{ssec:horizontalCurves} we introduce some additional notation regarding horizontal curves. $\varepsilon$-horizontal curves appear in Subsection \ref{ssec:epsilonHorizontalCurves}. The main result is Proposition \ref{prop:hPrincipleEpsilon}: the space of $\varepsilon$-horizontal curves is weakly equivalent to the space of formal horizontal curves. We then introduce analogues of this idea in the transverse setting. This is done in Subsections \ref{ssec:transverseCurves} and \ref{ssec:epsilonTransverseCurves}.

We assume that the reader is familiar with the $h$-principle language. The standard references on the topic are \cite{EM,Gro86}.

\subsection{Horizontal curves} \label{ssec:horizontalCurves}

Fix a manifold and a distribution $(M,\SD)$. We already introduced the spaces of immersed horizontal loops $\Imm(M,\SD)$ and embedded horizontal loops $\Emb(M,\SD)$. The phenomenon of rigidity forced us to look instead into $\Imm^\regu(M,\SD)$ and $\Emb^\regu(M,\SD)$, the subspaces of regular curves. We want to compare these to the formal analogues $\Imm^f(M,\SD)$ and $\Emb^f(M,\SD)$. This comparison relates geometrically-defined spaces to spaces that are topological\footnote{Formal immersions are simply monomorphisms with image in $\SD$ and thus tractable using homotopy theoretical tools. The case of formal horizontal embeddings is more subtle, due to the difficulty of studying smooth embeddings themselves. Nonetheless, for the case considered in this article (submanifolds of codimension at least $3$), manifold calculus \cite{Wei,GW} may be used to describe embeddings in purely homotopy-theoretical terms.} in nature. 

Proofs in $h$-principle are local in nature. That is to say, in order to prove our theorems, we will reduce them to analogous statements for horizontal paths, relative boundary. This motivates us to introduce the following notation. Given a $1$-dimensional manifold $I$, we write
\[ \Imm^\regu(I;M,\SD) \longrightarrow \Imm(I;M,\SD) \longrightarrow \Imm^f(I;M,\SD) \]
for the spaces of regular horizontal immersions, horizontal immersions, and formal horizontal immersions of $I$ into $(M,\SD)$. Similarly, we write
\[ \Emb^\regu(I;M,\SD) \longrightarrow \Emb(I;M,\SD) \longrightarrow \Emb^f(I;M,\SD) \]
in the case of embeddings. All spaces are endowed with the weak Whitney topology.

\subsection{$\varepsilon$--horizontal curves} \label{ssec:epsilonHorizontalCurves}

Being horizontal is a closed differential relation. These are typically more difficult to handle than open relations; dealing with them often requires some input from PDE theory or the use of a trick that transforms the problem into one involving an open relation. In this paper we follow the second route, manipulating horizontal curves through their projections to the space of controls (Section \ref{sec:graphicalModels}).

We now introduce $\varepsilon$-horizontality. $\varepsilon$-horizontal curves can also be manipulated using the their projections, with the added advantage of being described by an open relation. Fix a riemannian metric $g$ in $M$. We can measure the (unsigned) angle $\angle$, in terms of the metric $g$, between any two linear subspaces at a given $T_pM$.

\begin{definition}
Fix a constant $0 < \varepsilon < \pi/2$. The space of \textbf{$\varepsilon$-horizontal embeddings} is defined as:
\[ \Emb^\varepsilon(M,\SD) := \left\{ \gamma \in \Emb(M) \,\mid\, \angle(\gamma',\SD) < \varepsilon \right\}. \]
\end{definition}

Its formal analogue, the space of \textbf{formal $\varepsilon$-horizontal embeddings}, reads:
\begin{align*}
\Emb^{f,\varepsilon}(M,\SD) := \left\{\left(\gamma,(F_s)_{s\in[0,1]}\right): \right. & \quad\gamma \in \Emb(M), \quad F_s \in \Mon_{\NS^1}(T\NS^1,\gamma^*TM), \\
																															        &	\quad\left. F_0 = \gamma', \quad \angle(F_1,\gamma^*\SD) < \varepsilon \right\}.
\end{align*}

\subsubsection{Some flexibility statements}

It is a classic result due to M. Gromov that the $h$-principle holds in the $\varepsilon$-horizontal setting:
\begin{lemma}\label{lem:hPrincipioEpsilonHorizontal}
Consider $(M,\SD)$ with $\rank(\SD) \geq 2$. Then, the inclusion $\Emb^\varepsilon(M,\SD) \to \Emb^{f,\varepsilon}(M,\SD)$ is a weak homotopy equivalence.
\end{lemma}
\begin{proof}
This follows from convex integration for open and ample relations \cite[Theorem 18.4.1]{EM}. The relation is clearly open. Ampleness follows from the fact that principal subspaces are in correspondence with tangent fibres $T_pM$, and the relation in each is an open conical set (as depicted in Figure \ref{EpsilonHorizontal}) that is path-connected and ample, because $\SD$ has at least rank 2.
\end{proof}

\begin{figure}[h] 
	\includegraphics[scale=0.55]{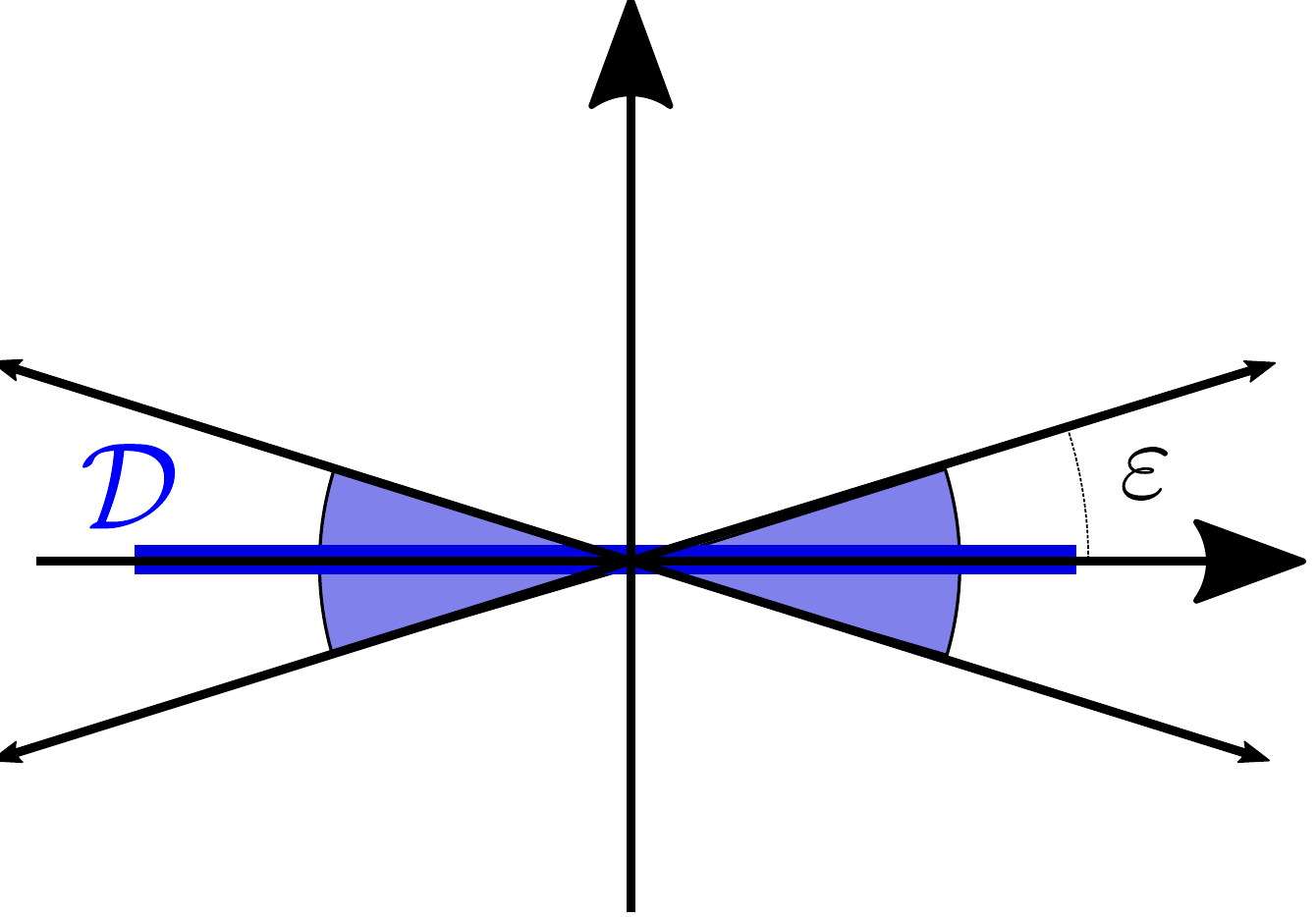}
	\centering
	\caption{Schematic depiction of the principal subspaces associated to $\varepsilon$-horizontality. The figure shows the rank $1$ case, in which the relation is conical and open but has two components. When the rank is at least 2, the relation is path-connected and thus ample.}\label{EpsilonHorizontal}
\end{figure}

Furthermore:
\begin{lemma}
The inclusion $\Emb^f(M,\SD) \to \Emb^{f,\varepsilon}(M,\SD)$ is a homotopy equivalence. In particular, $\Emb^\varepsilon(M,\SD)$ and $\Emb^f(M,\SD)$ are weakly homotopy equivalent.
\end{lemma}
\begin{proof}
Just note that the fiberwise orthogonal riemannian projection of $F_1$ onto $\SD$ provides a homotopy inverse.
\end{proof}

These results that we have just stated are also relative in the parameter, relative in the domain, and satisfy $C^0$-closeness. More precisely:
\begin{proposition} \label{prop:hPrincipleEpsilon}
Let $K$ be a compact manifold. Let $(M,\SD)$ be a manifold endowed with a distribution of rank greater or equal to $2$. Suppose we are given a map $(\gamma,F_s): K \to \Emb^f([0,1],M,\SD)$ satisfying the boundary conditions:
\begin{itemize}
\item $(\gamma,F_s)(k)|_{\Op(\{0,1\})}$ is a $\varepsilon$-horizontal embedding for all $k \in K$,
\item $(\gamma,F_s)(k) \in \Emb^\varepsilon([0,1],M,\SD)$ for $k \in \Op(\partial K)$.
\end{itemize}

Then, $(\gamma,F_s)$ extends to a homotopy $(\widetilde\gamma,\widetilde{F_s}): K \times [0,1] \to \Emb^f([0,1],M,\SD)$ that:
\begin{itemize}
\item restricts to $(\gamma,F_s)$ at time $s=0$,
\item maps into $\Emb^\varepsilon([0,1],M,\SD)$ at time $s=1$
\item is relative in the parameter (i.e. relative to $k \in \Op(\partial K)$),
\item is relative in the domain of the curves (i.e relative to $t \in \Op(\{0,1\})$),
\item has underlying curves $\widetilde\gamma(k,s)$ that are $C^0$-close to $\gamma(k)$ for all $k$ and $s$.
\end{itemize}
\end{proposition}
The statement still holds even if $\SD$ is allowed to vary parametrically with $k \in K$; this is not needed for our purposes.

\subsubsection{The punchline}

We can summarise the previous statements using the following commutative diagram:
\begin{center}
\begin{tikzcd}
\Emb^\regu(M,\SD) \subset \Emb(M,\SD) \arrow[r] \arrow[d] & \Emb^f(M,\SD) \arrow[d, "\cong"] \\
\Emb^\varepsilon(M,\SD) \arrow[r, "\cong"] & \Emb^{f,\varepsilon}(M,\SD)
\end{tikzcd}
\end{center}
It follows that, in order to prove our main Theorem \ref{thm:embeddings}, it is sufficient to understand the inclusion $\Emb^\regu(M,\SD) \hookrightarrow \Emb^\varepsilon(M,\SD)$. This simplification (passing from formal to $\varepsilon$) is commonly used in the $h$-principle literature, see for instance \cite{Mur,CP}.

\subsubsection{The case of immersions}

One can define, analogously, the space of immersed $\varepsilon$-horizontal loops:
\[ \Imm^\varepsilon(M,\SD) := \left\{\gamma \in \Imm(M):\angle(\gamma',\SD) < \varepsilon \right\}. \]
From the arguments above it follows that:
\begin{lemma} \label{lem:eHorizontalImm}
Let $\SD$ be a distribution with $\rank(\SD) \geq 2$. The map $\Imm^\varepsilon(M,\SD) \longrightarrow \Imm^f(M,\SD)$ is a weak homotopy equivalence.
\end{lemma}
This $h$-principle is also relative in the parameter, relative in the domain, and $C^0$-close. We leave it to the reader to spell out the analogue of Proposition \ref{prop:hPrincipleEpsilon}. Theorem \ref{thm:immersions} reduces then to the study of the inclusion $\Imm^\regu(M,\SD) \hookrightarrow \Imm^\varepsilon(M,\SD)$.

\subsection{Transverse curves} \label{ssec:transverseCurves}

We have already introduced the spaces of transverse immersed loops $\Immt(M,\SD)$, transverse embedded loops $\Embt(M,\SD)$, formally transverse immersions $\Immt^f(M,\SD)$, and formally transverse embeddings $\Embt^f(M,\SD)$. It was stated in the introduction that 
\[ \Immt(M,\SD) \longrightarrow \Immt^f(M,\SD), \qquad \Embt(M,\SD) \longrightarrow \Embt^f(M,\SD) \]
are weak equivalences whenever the corank of $\SD$ is larger than $1$, thanks to convex integration.
\begin{assumption} \label{assumption:corank1}
Whenever we work with transverse curves, we do so under the assumption that the corank of $\SD$ is one.
\end{assumption}

\subsubsection{Coorientations}

Suppose that $\SD$ is coorientable and fix a coorientation. We do not need this assumption for our results. However, we will make use of it as follows: due the relative nature of the arguments, we will reduce our theorems to $h$-principles in which the target manifold is euclidean space. In this local picture, the distribution is parallelisable and co-parallelisable. Furthermore, formal transverse curves induce a preferred coorientation. This will allow us to define a suitable replacement of $\varepsilon$-horizontality in the transverse setting.

\begin{definition}
A curve $\gamma: \NS^1 \to M$ is \textbf{positively transverse} if $\gamma'$ defines the positive coorientation in $TM/\SD$.
\end{definition}
If $\SD$ is cooriented, $\Immt(M,\SD)$, $\Embt(M,\SD)$, $\Immt^f(M,\SD)$, and $\Embt^f(M,\SD)$ split into two different path components (the positively transverse and the negatively transverse). In order not to overload notation, we will follow the convention that if $\SD$ is cooriented, we focus on the positive component.

We can now fix a riemannian metric $g$ on $M$ and define the \textbf{oriented angle} $\measuredangle(v(p), \SD_p)$ between a vector $v \in T_pM$ and the corank-$1$ distribution $\SD$. Its absolute value agrees with the (unsigned) angle $\angle(v, \SD)$ and its sign is positive if $v$ is positively transverse.

\begin{figure}[h] 
	\includegraphics[scale=0.8]{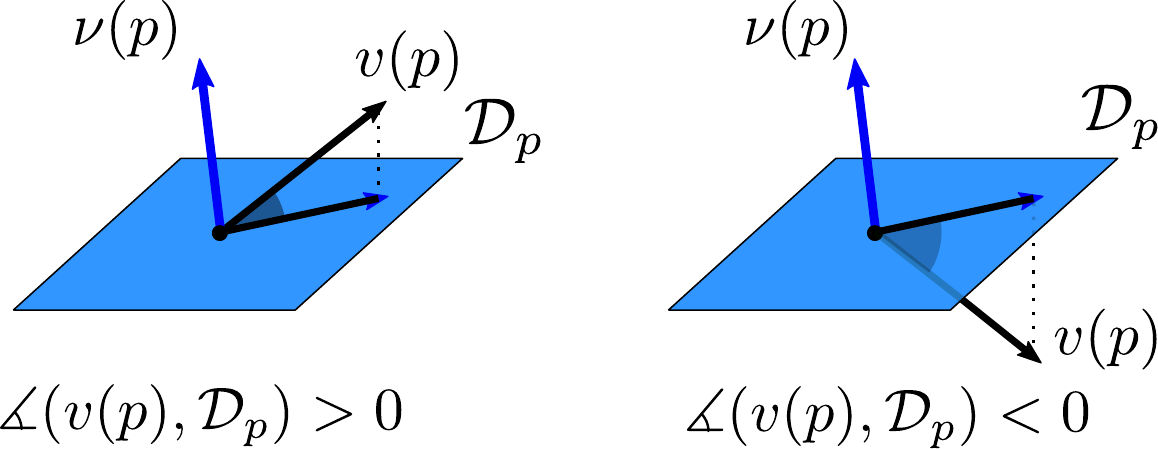}
	\centering
	\caption{On the left we depict a vector field $\nu$, providing the global choice of coorientation, and a positively transverse vector $v \in \T_pM$; the two lie on the same side of $\SD$. On the right we depict a negatively transverse vector.}\label{OrientedAngle}
\end{figure}

\subsubsection{Immersions}

If $I$ is a $1$-manifold, we write
\[ \Immt(I;M,\SD) \to \Immt^f(I;M,\SD), \quad \Embt(I;M,\SD) \to \Embt^f(I;M,\SD) \]
for the spaces of transverse immersions, formal transverse immersions, transverse embeddings and formal transverse embeddings of $I$ into $(M,\SD)$. Once again, if $\SD$ is cooriented, these denote only the positively transverse component.

\subsection{$\varepsilon$-transverse curves} \label{ssec:epsilonTransverseCurves}

Working under coorientability assumptions allows us to introduce the notion of $\varepsilon$-transversality.
\begin{definition}
Let $\SD$ be a cooriented distribution of corank-$1$. Fix a positive number $\pi/2 >\varepsilon>0$. The space of \textbf{$\varepsilon$-transverse embeddings} is defined as:
\[ \Embt^\varepsilon(M,\SD) = \left\{\gamma \in \Emb(M):\measuredangle(\gamma',\SD) > -\varepsilon \right\}. \]
\end{definition}
We can also consider its formal analogue, the space of \textbf{formal $\varepsilon$-transverse embeddings}:
\begin{align*}
\Embt^{f,\varepsilon}(M,\SD) = \left\{\left(\gamma,(F_s)_{s\in[0,1]}\right): \right. & \quad\gamma \in \Emb(M), \quad F_s \in \Mon_{\NS^1}(T\NS^1,\gamma^*TM), \\
																															        &	\quad\left. F_0 = \gamma', \quad \measuredangle(F_1,\gamma^*\SD) > -\varepsilon \right\}.
\end{align*}

\subsubsection{Flexibility statements}

The following is an analogue of Proposition \ref{prop:hPrincipleEpsilon}, with a milder condition on the rank. The proof is analogous, using convex integration and projection to $\SD$:
\begin{proposition}
Let $\SD$ be a cooriented corank-$1$ distribution with $\rank(\SD) \geq 1$. Then, the following inclusions are weak equivalences:
\[ \Embt^\varepsilon(M,\SD) \longrightarrow \Embt^{f,\varepsilon}(M,\SD), \qquad \Emb^f(M,\SD) \longrightarrow \Emb^{f,\varepsilon}(M,\SD). \]
\end{proposition}

\begin{figure}[h] 
	\includegraphics[scale=0.55]{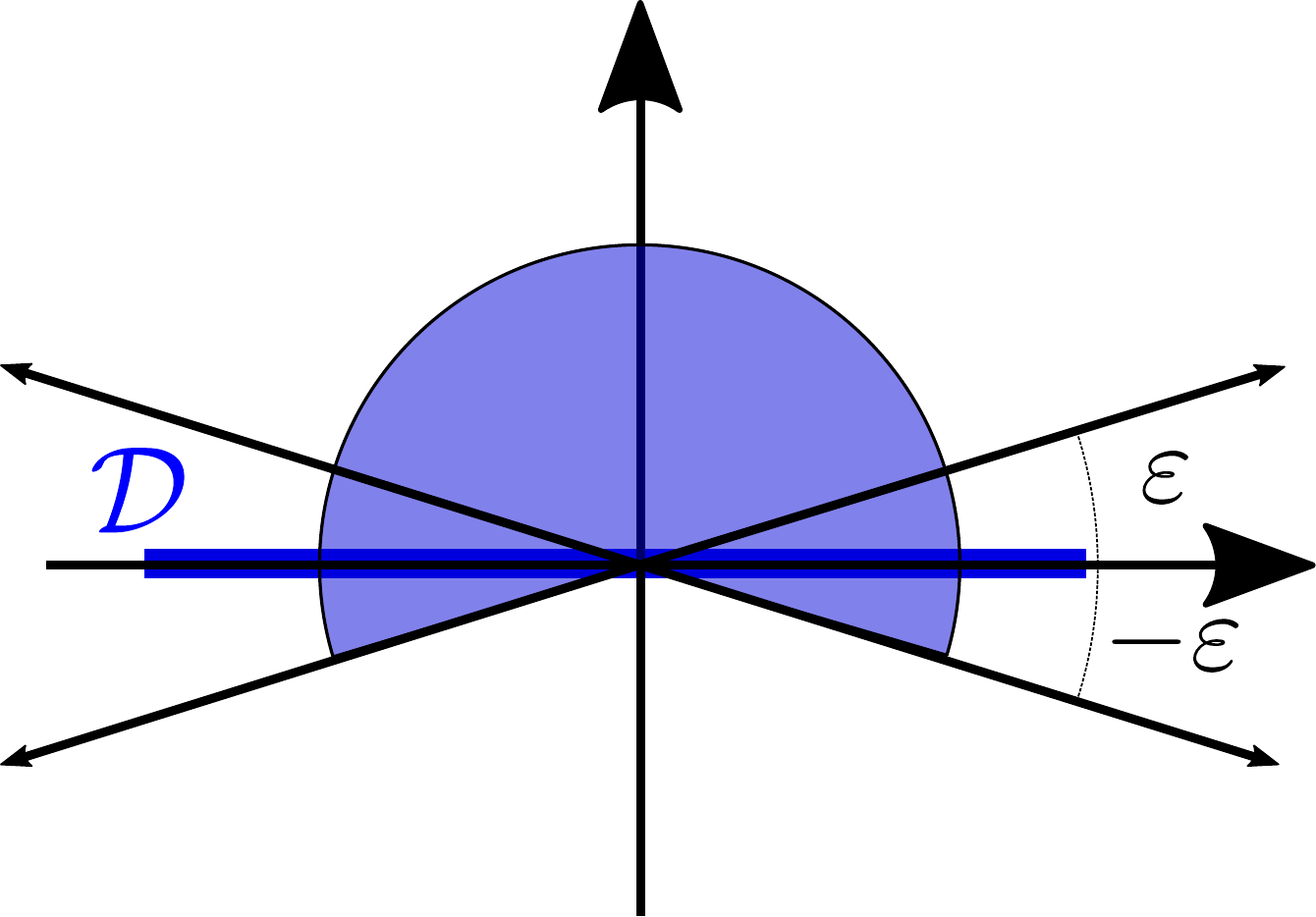}
	\centering
	\caption{The fiberwise relation defining $\varepsilon$-transversality defines a relation that is ample.}\label{fig:epsilonTransverse}
\end{figure}

A statement that is relative in the parameter, relative in the domain, and $C^0$-close also holds:
\begin{proposition} \label{prop:hPrincipleEpsilonTransverse}
Let $K$ be a compact manifold. Let $(M,\SD)$ be a manifold endowed with a cooriented corank-$1$ distribution of rank at least $1$. Suppose we are given a map $(\gamma,F_s): K \to \Embt^f([0,1];M,\SD)$ satisfying:
\begin{itemize}
\item $(\gamma,F_s)(k)|_{\Op(\{0,1\})}$ is a $\varepsilon$-transverse embedding for all $k \in K$,
\item $(\gamma,F_s)(k) \in \Embt^\varepsilon([0,1];M,\SD)$ for $k \in \Op(\partial K)$.
\end{itemize}

Then, $(\gamma,F_s)$ extends to a homotopy $(\widetilde\gamma,\widetilde{F_s}): K \times [0,1] \to \Embt^f([0,1];M,\SD)$ that:
\begin{itemize}
\item restricts to $(\gamma,F_s)$ at time $s=0$,
\item maps into $\Embt^\varepsilon([0,1];M,\SD)$ at time $s=1$
\item is relative in the parameter (i.e. relative to $k \in \Op(\partial K)$),
\item is relative in the domain of the curves (i.e relative to $t \in \Op(\{0,1\})$),
\item has underlying curves $\widetilde\gamma(k,s)$ that are $C^0$-close to $\gamma(k)$ for all $k$ and $s$.
\end{itemize}
\end{proposition}

\subsubsection{The punchline}

We obtain the following commutative diagram:
\begin{center}
\begin{tikzcd}
\Embt(M,\SD) \arrow[r] \arrow[d] & \Embt^f(M,\SD) \arrow[d, "\cong"] \\
\Embt^\varepsilon(M,\SD) \arrow[r, "\cong"] & \Embt^{f,\varepsilon}(M,\SD),
\end{tikzcd}
\end{center}
telling us that we should focus on the inclusion $\Embt(M,\SD)\hookrightarrow \Embt^\varepsilon(M,\SD)$. We will do so to prove Theorem \ref{thm:TransverseEmbeddings}.

\subsubsection{The immersion case}

We can also define the space of $\varepsilon$-transverse immersions $\Immt^\varepsilon(M,\SD)$ and deduce that:
\begin{lemma} \label{lem:eTransverselImm}
Let $\SD$ be a cooriented, corank-$1$ distribution of rank at least $1$. The map $\Immt^\varepsilon(M,\SD) \to \Immt^f(M,\SD)$ is a weak homotopy equivalence.
\end{lemma}
This $h$-principle is also relative in the parameter, relative in the domain, and $C^0$-close. We will henceforth focus on the inclusion $\Immt(M,\SD)\hookrightarrow \Immt^\varepsilon(M,\SD)$ in order to prove Theorem \ref{thm:TransverseImmersions}.

\subsubsection{Almost transversality}

To wrap up this section, consider the following definition:
\begin{definition}
Let $\SD$ be a cooriented, corank-$1$ distribution. Let $I$ be a $1$-dimensional manifold. The space of \textbf{almost transverse embeddings} is:
\[ \EmbAT(I;M,\SD) := \left\{\gamma \in \Emb(I,M) \,\mid\, \measuredangle(\gamma',\SD) \geq 0 \right\}. \]
We write $\EmbAT(M,\SD)$ in the particular case of loops.
\end{definition}
This may be regarded as the closure of the space of positively transverse embeddings $\Embt(I,M,\SD)$. We only introduce it because it will allow us to translate flexibility statements about horizontal curves to the transverse setting.

\section{Graphical models} \label{sec:graphicalModels}

One of the two standard projections used in Contact Topology to study legendrian (i.e. horizontal) knots in standard contact $(\R^3,\xi_\std = \ker(dy-zdx))$ is the so-called Lagrangian projection:
\begin{align*}
\pi_L: \R^3 \quad&\longrightarrow\quad \R^2\\
	(x,y,z)		\quad&\mapsto\quad (x,z).
\end{align*}
It projects $\xi_\std$, at each $p\in\R^3$, isomorphically onto the tangent space $T_{\pi_L(p)}\R^2$ of the base. This projects a legendrian knot to an immersed planar curve. From the projected curve one can recover the missing $y$-coordinate by integration:
\[ z(t) \quad=\quad z(t_0)+\int_{t_0}^t x(s)y'(s)ds. \]
Indeed, the integral on the right-hand side, when evaluated over the whole curve, computes the area it bounds due to Stokes' theorem. This turns the problem of manipulating Legendrian knots into a problem about planar curves satisfying an area constraint.

In this Section we introduce \emph{graphical models}. These are opens in Euclidean space, endowed with a bracket-generating distribution that is graphical over some of the coordinates; we call these the \emph{base}. Projecting to the base and manipulating curves there is analogous to using the Lagrangian projection. This line of reasoning is also classic in Geometric Control Theory: the tangent space of the base, upon choosing a framing, corresponds to the space of controls.

Graphical models are introduced in Subsection \ref{ssec:graphicalModels}. The related notion of ODE model appears in Subsection \ref{ssec:ODEModels}. In Subsection \ref{ssec:adaptedCharts} we explain how to cover any $(M,\SD)$ by graphical models. These local models will be used in Section \ref{sec:microflexibility} to manipulate horizontal and transverse curves.

\subsection{Graphical models} \label{ssec:graphicalModels}

Fix a rank $q$, an ambient dimension $n$, and a step $m$. We now introduce the main definition of this section. It may remind the reader of the ideas used to construct balls in Carnot-Caratheodory geometry \cite{Ge,Gro96,Mon}:
\begin{definition} \label{def:graphicalModel}
A \textbf{graphical model} is a tuple consisting of:
\begin{itemize}
\item a radius $r > 0$,
\item a constant-growth, bracket-generating, rank-$q$ distribution $\SD$ defined over the ball $B_r \subset \R^n$,
\item the standard projection $\pi: B_r \subset \R^n \rightarrow \R^q$ to the so-called \textbf{base},
\item a framing $\{X_1,\cdots,X_n\}$ of $TB_r$.
\end{itemize}
The Lie flag of $\SD$ will be denoted by
\[ \SD = \SD_1 \subset \SD_2 \subset \cdots \subset \SD_m = TB_r, \]
and we write $q_i = \rank(\SD_i)$.

This tuple must satisfy the following conditions:
\begin{itemize}
\item $\{X_1,\cdots,X_{q_i}\}$ is a framing of $\SD_i$.
\item Given $j = q_{i-1}+1,\cdots, q_i$, there is a formal bracket expression $A_j$ and a collection of indices $l^j_a=1,\cdots,q$ satisfying:
\[ X_j = A_j(X_{l_1^j},\cdots,X_{l_i^j}), \]
\item[a. ] $d\pi(X_j) = \partial_j$ for all $j=1,\cdots,q$. In particular, $d\pi$ is a fibrewise isomorphism between $\SD$ and $T\R^q$. 
\item[b. ] $X_j(0) = \partial_j$ for all $j=1,\cdots,n$.
\end{itemize}
\end{definition}
The first two conditions are unnamed because they simply state that the given framing is compatible with the Lie flag. Condition (a) says that the distribution is a connection over $\R^q$ and that its framing is the lift of the standard coordinate framing. Condition (b) controls $\{X_1,\cdots,X_n\}$, saying that they agree with the standard coordinate directions at the origin. This will allow us to describe, quantitatively, how paths in the base lift to horizontal curves. As one may expect, we will be able to estimate this up to an error of size $O(r)$.

\subsection{ODE models} \label{ssec:ODEModels}

Since $\SD$ is a connection over $\R^q$, any curve $\gamma: \R \longrightarrow \R^q$ can be lifted to a horizontal curve of $\SD$, uniquely once a lift of $\gamma(0)$ has been chosen. Conversely, any horizontal curve is recovered uniquely from its projection by lifting (using the appropriate initial point). This is a consequence of the fundamental theorem of ODEs. 

The caveat is that $\SD$ is only defined over $B_r$, so the claimed lift may escape the model and therefore not exist for all times; this is the usual issue one encounters with non-complete flows. Still, the lift 
\[ \lift(\gamma): U \longrightarrow (B_r,\SD) \]
is uniquely defined over some maximal open interval $U \subset \R$ that contains zero. In order to discuss this a bit further, we introduce:
\begin{definition}
Consider a curve $\gamma: \R \longrightarrow \R^q$ mapping to the base of a graphical model. Thanks to $\pi$, $\R^n$ can be seen as a fibration $F$ over $\R^q$, allowing us to consider the tautological map
\[ \Psi: \gamma^*F \cong \R \times \R^{n-q} \quad\longrightarrow \quad \R^n \]
that is transverse to $\SD$ (whenever the later is defined).

The \textbf{ODE model} associated to $\gamma$ (and to the graphical model) consists of:
\begin{itemize}
\item an open subset $D \subset \gamma^*F$,
\item the tautological map $\Psi: D \longrightarrow B_r$,
\item the line field $\Psi^*\SD$, whose domain of definition is $D$.
\end{itemize}
\end{definition}
That is, $(D,\Psi)$ parametrises the region of $B_r$ that lies over $\gamma$. Do note that $\Psi$ is an immersion/embedding only if $\gamma$ itself is immersed/embedded. Our discussion above states that lifting $\gamma$ amounts to choosing a basepoint in $D$, solving the ODE $\Psi^*\SD$, and pushing forward with $\Psi$.

\begin{figure}[h] 
	\includegraphics[scale=0.35]{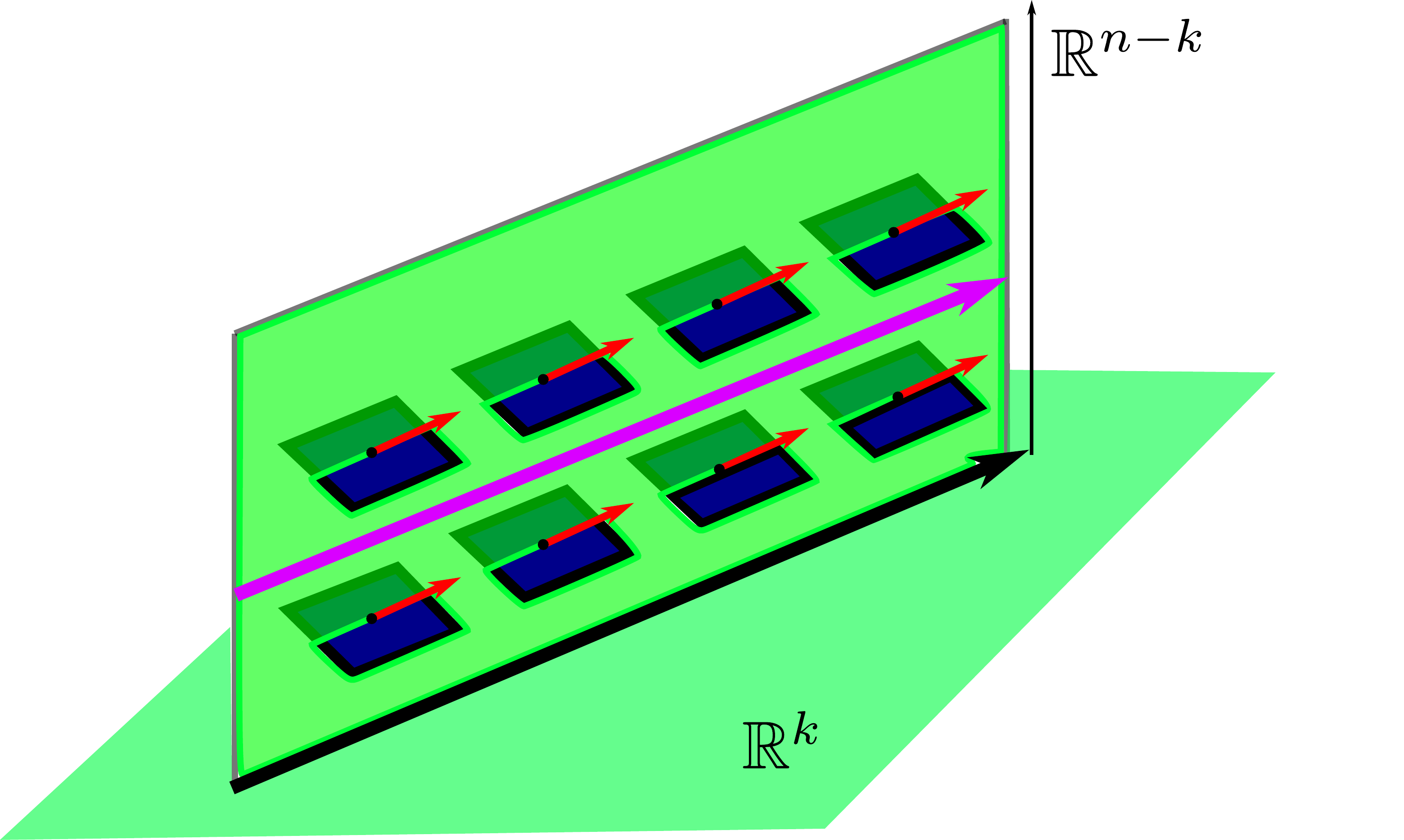}
	\centering
	\caption{A graphical model and an ODE model. The distribution is seen as a connection, where $T\R^n = \SD \oplus \ker d\pi$. A curve $\gamma$ is shown in the base, with its ODE model lying above. $\SD$ restricts to the ODE model as a line field. A lift of $\gamma$ is shown in magenta.}\label{GraphicalModel}
\end{figure}

The reason behind introducing ODE models is that they allow us to state the following trivial lemma:
\begin{lemma} \label{lem:ODEModel}
Fix a graphical model and consider the following objects:
\begin{itemize}
\item A curve $\gamma: [0,1] \longrightarrow \R^q$ mapping to the base. 
\item A (defined for all time) lift $\lift(\gamma): [0,1] \longrightarrow B_r$ of $\gamma$.
\item The ODE model $(D,\Psi,\Psi^*\SD)$ of $\gamma$.
\item The unique integral curve $\nu: [0,1] \longrightarrow (D,\Psi^*\SD)$, graphical over $[0,1]$, such that $\Psi \circ \nu = \lift(\gamma)$.
\end{itemize}
Then, there is a constant $\delta > 0$ and coordinates $\phi: [0,1] \times B_\delta \longrightarrow D$ such that:
\begin{itemize}
\item $\phi(t,0) = \nu(t)$,
\item $\phi^*\Psi^*\SD$ is spanned by $\partial_1$.
\end{itemize}
\end{lemma}
This is a consequence of the flowbox theorem, so we will call $\phi$ \textbf{flowbox coordinates}. This allows us to see horizontal/transverse curves as graphs of functions and see horizontality/transversality in terms of their slope. See Figure \ref{fig:ODEmodel}.

\begin{figure}[h] 
	\includegraphics[scale=0.5]{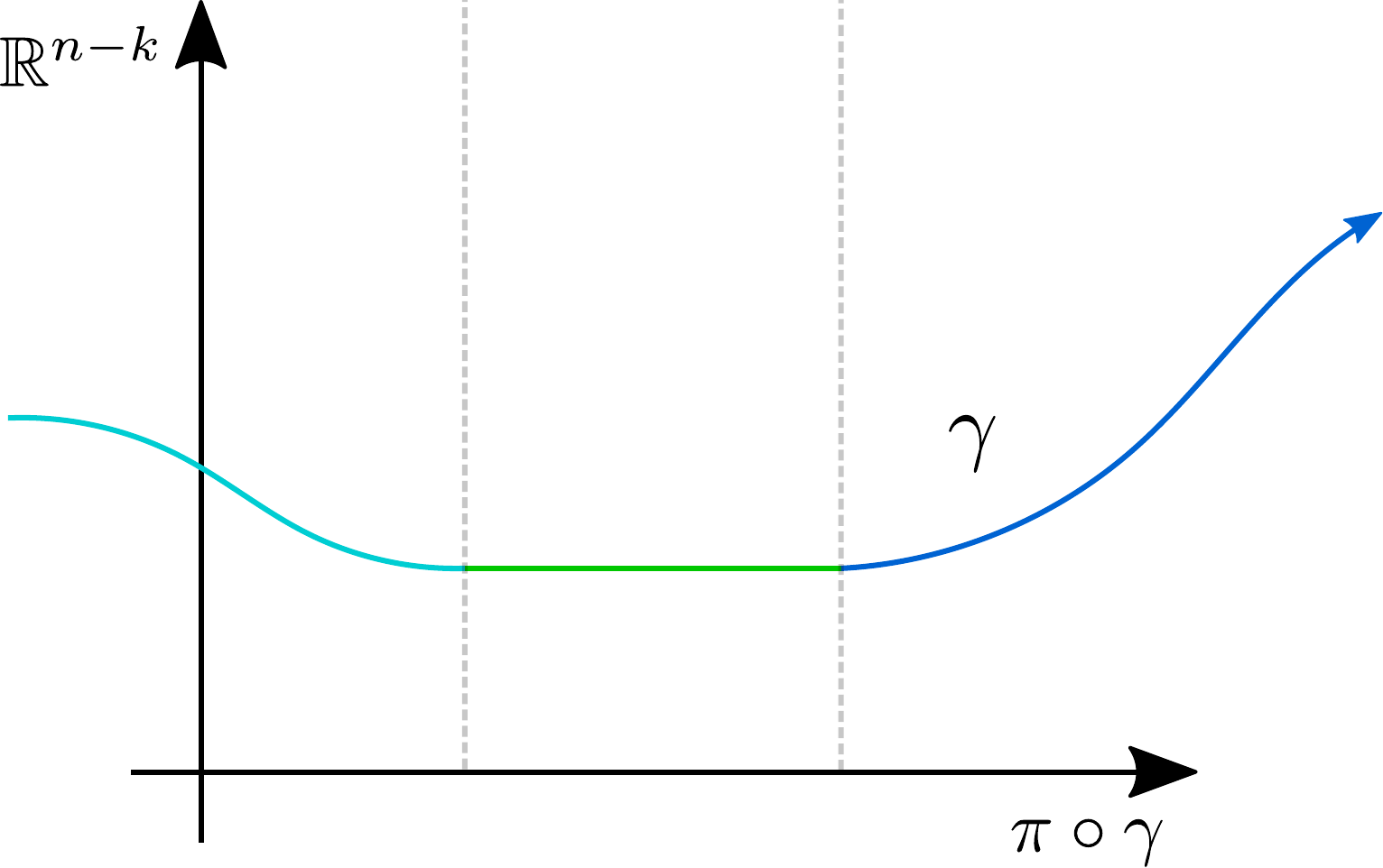}
	\centering
	\caption{An ODE model in flowbox coordinates. The distribution is of corank-$1$ and its coorientation is given by ``going up''. It contains a $\varepsilon$-horizontal curve $\gamma$ with  three differentiated regions. The left region is negatively transverse. In right region the curve is positively transverse. In the middle region $\gamma$ is parallel to the $x$-axis and is thus horizontal.}\label{fig:ODEmodel}
\end{figure}

\subsubsection{Size of ODE models} \label{sssec:ODEsize}

Using the properties of a graphical model, we can estimate how large the constant $\delta$ appearing in Lemma \ref{lem:ODEModel} can be.

Write $\pi_v: D \rightarrow \R^{n-q}$ for the standard projection to the vertical. According to the definition of graphical model, $\SD$ and the foliation by planes parallel to $\R^q$ differ by $O(r)$. In particular, the slope of $\SD$ is bounded by $O(r)$. The same holds for $\Psi^*\SD$ in $D$. This implies a bound for the vertical displacement:
\begin{equation} \label{eq:naiveBound}
|\pi_v \circ \lift(\gamma)(1) - \pi_v \circ \lift(\gamma)(0)| < \len(\gamma)O(r)
\end{equation}
where the right-most term is the length of $\gamma$. It is valid as long as it is smaller than the distance $d$ of $\gamma(0)$ to the boundary of the model. Choosing $\len(\gamma)$ sufficiently small, we can choose $\delta$ to be of the magnitude of $d$.

The punchline is that, in order to manipulate horizontal curves effectively on a manifold $(M,\SD)$, it will be necessary to cover it with graphical models whose radii are very small, as this will allow us to estimate vertical displacement of curves in an effective manner.

\subsection{Adapted charts} \label{ssec:adaptedCharts}

Fix a manifold $M$ and a bracket-generating distribution $\SD$. We now prove that $(M,\SD)$ can be covered by graphical models. For notational ease, let us introduce a definition first. Given a point $p \in M$, an \textbf{adapted chart} is a graphical model $(B_r \subset \R^n,\SD_U)$ together with a chart 
\[ \phi: (B_r,\SD_U) \longrightarrow (M,\SD), \]
such that $\phi^*\SD = \SD_U$ and $\phi(0) = p$.

\begin{proposition}\label{prop:adaptedCharts}
Let $(M,\SD)$ be a manifold endowed with a bracket-generating distribution. Then, any point $p \in M$ admits an adapted chart.
\end{proposition}
\begin{proof}
We argue at a fixed but arbitrary point $p \in M$. Fix a basis $\left\{Y_1,\cdots, Y_n\right\}$ of $T_pM$ such that $\left\{Y_1,\cdots, Y_q\right\}$ spans $\SD_p$. We construct local coordinates around $p$ using applying the exponential map. In these new coordinates we have that $Y_i$ is $\partial_i$. Condition (a) in the definition of graphical model follows.

In the new coordinates, we have local projection $\pi: \R^n \to \R^q$. This allows us to define a framing $F_\SD := \left\{X_1,\cdots, X_q\right\}$ of $\SD$ by lifting the coordinate vector fields $\partial_i$ of $\R^q$. Due to the bracket-generating condition, all vector fields around $p$ can be written as linear combinations of Lie brackets involving vector fields in $\SD$. Since all such vector fields are themselves sums of elements in $F_\SD$, it can be deduced that $TM$ is spanned by bracket expressions involving only $F_\SD$. This allows us to extend $F_\SD$ to a frame $\left\{X_1,\cdots, X_n\right\}$ such that:
\begin{itemize}
\item $\left\{X_1,\cdots, X_{q_i}\right\}$ spans $\SD_i$.
\item $X_j = A_j(X_{i_1},\cdots,X_{i_l})$, with $i_a \leq q$ and $A_j$ some bracket-expression.
\end{itemize}

By construction, the elements in $F_\SD$ commute with one another upon projecting to $\R^q$. I.e. their Lie brackets are purely vertical, meaning that the vector fields $\left\{X_{q+1},\cdots, X_n\right\}$ are tangent to the fibres of $\pi$. This implies that, by applying a linear transformation fibrewise, we can produce new coordinates in which $X_l(p)$ is $\partial_l$. Condition (b) follows.
\end{proof}

\subsubsection{Covering by nice adapted charts}

It is crucial for our arguments to be able to produce coverings by graphical models that are arbitrarily fine, and whose behaviour is controlled regardless of how fine we need them to be. That is the content of the following corollary:
\begin{lemma} \label{lem:adaptedCharts}
Let $(M,\SD)$ be a compact manifold endowed with a bracket-generating distribution. Then, there are constants $C, r > 0$ such that:
\begin{itemize}
\item[i. ] Any point $p \in M$ admits an adapted chart of radius $r$.
\item[ii. ] The bound $|X_j(x) - \partial_j| < C|x|$ holds for all such adapted charts and all elements $X_j$ in the corresponding framing.
\end{itemize}
The measuring in both properties is done using the Euclidean distance given by the adapted chart.
\end{lemma}
\begin{proof}
The construction in Proposition \ref{prop:adaptedCharts} is parametric on $p$. This is certainly true for the exponential map, which yields the resulting local coordinates, the projection $\pi$, and thus the frame $F_\SD$. This is not necessarily the case, globally, for the choice of bracket expressions $A_j$, but it is still true if we argue on opens of some sufficiently fine cover $\{U_i\}$ of $M$. The fibrewise linear transformation chosen at the end of the argument is unique and thus parametric.

It follows that the statement holds for constants $C_i,r_i > 0$ over each $U_i$. We extract a finite covering to conclude the argument.
\end{proof}

Do note that, upon zooming-in at $p$, the distribution $\SD$ converges to a Carnot group called the nilpotentisation \cite[Chapter 4]{Mon}. That is, to a nilpotent Lie group endowed with a bracket-generating distribution that is left-invariant and invariant under suitable weighted scaling. This implies that, by taking a sufficiently fine cover of $(M,\SD)$, one can produce graphical models that are as close as required to a Carnot group. This is an improvement on Lemma \ref{lem:adaptedCharts}, but it is not needed for our arguments.

\section{Microflexibility of curves} \label{sec:microflexibility}

The results in this paper follow an overall strategy that is standard in $h$-principle. Namely: we first perform a series of simplifications that are meant to reduce the proof to a problem that is localised in a small ball. We call this \emph{reduction}. Reduction arguments can be technical but often follow some standard heuristics and patterns. Once we have passed to a localised setting, the second step begins. This is the core of the proof and requires some input that is specific to the geometric setup at hand. This step is called \emph{extension} because it often amounts to extending a solution from the boundary of small ball to its interior\footnote{The proof of Theorem \ref{thm:TransverseEmbeddings} follows these general lines. The proof of Theorem \ref{thm:embeddings} presents some subtleties that force us to do something slightly different; see Remark \ref{rem:whyNotTriangulate}}.

In this section we prove several lemmas dealing with deformations of horizontal and transverse curves; they are meant to be used in the reduction step. These deformations often take place along stratified subsets of positive codimension, and can therefore be understood as microflexibility phenomena for horizontal and transverse curves\footnote{Do note that, due to rigidity, horizontal curves are not microflexible in general.}. Proofs boil down to patching up local constructions happening in graphical models (Section \ref{sec:graphicalModels}).

In Subsection \ref{ssec:Thurston} we review Thurston's jiggling, which we use to triangulate our manifolds and thus argue one simplex at a time. Local statements taking place in graphical models are presented in Subsection \ref{ssec:projectingLifting}. We globalise these constructions in Subsections \ref{ssec:horizontalisation} (for horizontal curves) and \ref{ssec:transversalisation} (for transverse curves).

\subsection{Triangulations} \label{ssec:Thurston}

In order to reduce our arguments to a euclidean situation, we fix a (sufficiently fine) triangulation and we work on neighbourhoods of the simplices. For our purposes it is important that these triangulations are well-behaved. This can be achieved using the Thurston jiggling Lemma. We state it for the case of line fields (which is all we need):
\begin{lemma}[Thurston, \cite{Thur}] \label{lem:jiggling}
Let $N$ be a smooth $n$-manifold equipped with a line field $\xi$. Fix a metric $g$. Then, there exists a sequence of triangulations $\ST_b$ satisfying the following properties:
\begin{itemize}
\item[i. ] Each simplex $\Delta \in \ST$ is transverse to $\xi$.
\item[i'. ] Each $n$-simplex is a flowbox of $\xi$.
\item[ii. ] The radius of each simplex in $\ST_b$ is bounded above by $1/b$.
\item[iii. ] The number of simultaneously incident simplices in $\ST_b$ is bounded above by a constant independent of $b$.
\end{itemize}
\end{lemma}
Conditions (i) and (i'') say that $\xi$ is almost constant in the coordinates provided by $\Delta$. Condition (ii) says that the triangulations are becoming finer as $b$ increases (and indeed all of them can be assumed to be refinements of some given triangulation). Condition (iii) states that the combinatorics of the triangulation remain controlled upon refinement (which is needed to prove Condition (i')).

We will also need a version for manifolds with the boundary:
\begin{corollary} \label{cor:jigglingBoundary}
In Lemma \ref{lem:jiggling}, suppose that $N$ has boundary. Then we can furthermore assume that:
\begin{itemize}
\item $\ST_b$ extends a triangulation of the boundary $\partial N$.
\item Conditions (i), (i'), (ii), (iii) hold for all simplices of $\ST_b$ not fully contained in $\partial N$.
\item The pair $(\partial N, \ST_b|_{\partial N})$ satisfies the conclusions of the Lemma.
\end{itemize}
\end{corollary}

\subsection{Local arguments} \label{ssec:projectingLifting}

We now present a series of statements dealing with families of curves mapping into graphical models. In order to streamline notation, let us denote the target graphical model by $(V,\SD)$. Its projection to the base $\R^q$ is denoted by $\pi$ and the projection to the vertical by $\pi_\vert$. We also write $K$ for the smooth compact manifold serving as the parameter space of the families.

\subsubsection{Horizontalisation in graphical models}

We will often construct horizontal curves as lifts of curves in $\R^1$. The following is a quantitative statement about the existence of lifts: 
\begin{lemma} \label{lem:horizontalisationGraphical}
Consider a family $\gamma: K \rightarrow \Emb^\varepsilon([0,1];\D_{r_1},\SD)$. Then, there is a unique family $\nu: K \rightarrow \Emb([0,\delta];\D_r,\SD)$ satisfying
\begin{itemize}
\item $\pi \circ \nu = \pi \circ \gamma$.
\item $\nu(k)(0) = \gamma(k)(0)$.
\end{itemize}
Furthermore: Let $l$ be an upper bound for the velocity $||(\pi \circ \gamma(k))'||$. Then we can assume
\[ \delta > \dfrac{r-r_1}{l(\varepsilon+O(r))}. \]
\end{lemma}
\begin{proof}
The uniqueness of $\nu$ is immediate from the discussion in Subsection \ref{ssec:ODEModels}, since $\nu$ is obtained from $\pi \circ \gamma$ by lifting with a given initial value. The bound on $\delta$ follows from the fact that the slope of $\SD$ with the horizontal is at most $\varepsilon+O(r)$, so the difference between $\gamma$ and $\nu$, which is purely vertical is bounded by $\delta.l.(\varepsilon+O(r))$. For $\nu$ to remain within the $r$-ball, this quantity must be smaller than $r-r_1$, yielding the claim.
\end{proof}
Do note that the coefficient in the expression $O(r)$ can be bounded above in terms of the derivatives of $\SD$. In Lemma \ref{lem:adaptedCharts} we observed that this coefficient can be bounded globally over a compact manifold.

\subsubsection{Stability of horizontalisation}

Given a family of horizontal curves, we may want to produce a nearby horizontal family by manipulating its projection to the base. The following lemma says that this is indeed possible:
\begin{lemma} \label{lem:horizontalisationStability}
Consider families 
\[ \gamma: K \rightarrow \Emb([0,1];\D_r,\SD), \qquad \alpha: K \rightarrow \Imm([0,1];\R^q) \]
such that $\pi \circ \gamma$ and $\alpha$ are $C^0$-close and their lengths are close. Then, there is a lift $\nu: K \rightarrow \Imm([0,1];\D_r,\SD)$ of $\alpha$ that is $C^0$-close to $\gamma$.
\end{lemma}
\begin{proof}
The family $\nu$ is obtained by lifting $\alpha$, as in Subsection \ref{ssec:ODEModels}. The conclusion forces us to choose an initial value that is close to $\gamma(k)(0)$. The hypothesis on $\alpha$ (closeness in $C^0$ and length) imply (Lemma \ref{lem:horizontalisationSkeleton}) that the ODE behind the lifting process is close to the ODE associated to $\pi \circ \gamma$. This implies that the lifting exists over $[0,1]$ and is close to $\gamma$.
\end{proof}
In concrete instances we will be able to argue that the resulting family $\nu$ also consists of embedded curves. This will follow from the specific properties of the family $\alpha$ under consideration.

\subsubsection{Interpolation statements}

We sometimes consider deformations of $\varepsilon$-horizontal curves in which the projection to the base remains fixed and the vertical component changes. This is explained in the following lemma, whose proof we leave to the reader:
\begin{lemma} \label{lem:varepsilonInterpolation}
Fix a family of curves $\alpha: K \rightarrow \Imm([0,1];\R^q)$. Then, there exists a constant $\delta>0$ such that any two families of curves
\[ \gamma, \nu: K \rightarrow \Imm^\varepsilon([0,1];\D_r,\SD) \]
lifting $\alpha$ and satisfying $|\gamma(k) - \nu(k)|_{C^0} < \delta$ are homotopic through a family of $\varepsilon$-horizontal curves also lifting $\alpha$.

Furthermore, this homotopy may be assumed to be relative to $\Op(\partial(K \times I))$ if the families already agree there.
\end{lemma}

The analogue for transverse curves reads:
\begin{lemma} \label{lem:transverseInterpolation}
Suppose $\SD$ is of corank-1 and cooriented. Fix a family of curves $\alpha: K \rightarrow \Imm([0,1];\R^q)$. Then, there exists a constant $\delta>0$ such that any two families of curves
\[ \gamma, \nu: K \rightarrow \Immt([0,1];\D_r,\SD) \]
lifting $\alpha$ and satisfying $|\gamma(k) - \nu(k)|_{C^0} < \delta$ are homotopic through a family of transverse curves also lifting $\alpha$.

Furthermore, this homotopy may be assumed to be relative to $\Op(\partial(K \times I))$ if the families already agree there.
\end{lemma}

\subsubsection{Deforming $\varepsilon$-horizontal curves}

The following is an analogue of Lemma \ref{lem:horizontalisationStability} in the $\varepsilon$-horizontal setting.
\begin{lemma} \label{lem:varepsilonStability}
Consider families
\[ \gamma: K \rightarrow \Emb^\varepsilon([0,1];\D_r,\SD) \qquad \alpha: K \rightarrow \Imm([0,1];\R^q) \]
such that $\pi \circ \gamma$ and $\alpha$ are $C^0$-close and their lengths are close. Then, there is a lift $\nu: K \rightarrow \Imm^\varepsilon([0,1];\D_r,\SD)$ of $\alpha$ that is $C^0$-close to $\gamma$.
\end{lemma}
\begin{proof}
Constructing $\nu$ amounts to choosing its vertical component $\pi_\vert \circ \nu$. Naively, we could set $\pi_\vert \circ \nu = \pi_\vert \circ \gamma$, but there is no reason why this would preserve $\varepsilon$-horizontality. The strategy to be pursued instead is to mimic the proof of Lemma \ref{lem:horizontalisationStability}. Namely, we want to see $\gamma$ as the solution of an ODE (that makes an angle of at most $\varepsilon$ with respect to $\SD$) and produce $\nu$ as a solution of a similar ODE. 

We define families of vector fields 
\[ X,Y,Z,W: K \times I \rightarrow \Gamma(TV) \]
satisfying the condition:
\begin{itemize}
\item $X$ is tangent to $\SD$ and satisfies $d\pi(X(k,t)) = d\pi(\gamma(k)'(t))$ over all points lying over $\pi \circ \gamma(k)(t)$.
\item $Y$ is vertical and satisfies $Y(k,t) = X(k,t) - \gamma(k)'(t)$. It follows that, along $\gamma$, we have the inequality:
\[ \dfrac{|Y(k,t)|}{|\gamma(k)'(t)|} = \angle(\gamma(k)'(t),\SD) < \sin(\varepsilon). \]
\item $W$ is tangent to $\SD$ and satisfies $d\pi(W(k,t)) = \alpha(k)'(t)$ over all points projecting to $\alpha(k)(t)$
\item $Z$ is vertical and given by the expression $Z(k,t) = Y(k,t)\dfrac{|\alpha(k)'(t)|}{|X(k,t)|}$. 
\end{itemize}
According to these definitions, $\gamma(k)$ is an integral line of $X(k,t) + Y(k,t)$. We define $\nu(k)(t)$ to be the integral line of $W(k,t) + Z(k,t)$ with initial condition $\gamma(k)(0)$.

By construction, $\pi \circ \nu = \alpha$ and therefore $\nu$ is immersed. Furthermore, due to our definition $Z$, $\nu$ is $\varepsilon$-horizontal as long as $\alpha$ is sufficiently close to $\pi \circ \gamma$. Lastly, $C^0$-closeness of $\nu$ and $\gamma$ follows from the closeness of $\alpha$ and $\pi \circ \gamma$ in $C^0$ and length.
\end{proof}

\subsection{Horizontalisation} \label{ssec:horizontalisation}

We now present semi-local analogues of Lemma \ref{lem:horizontalisationGraphical}. Since the Lemma only provides short-time existence of horizontal curves, generalisations must also present this feature. The reader should think of the upcoming statements as analogues of the holonomic approximation theorem \cite[Theorem 3.1.1]{EM}. However, they involve no wiggling.

The general setup is the following: We fix a pair $(M,\SD)$. The distribution need not be bracket-generating. Our families of curves are parametrised by a compact manifold $K$ and have the unit interval $I = [0,1]$ as their domain. The product $K \times I$ contains a stratified subset $A$ such that all its strata are transverse to the $I$-factor.

\subsubsection{Horizontalisation along the skeleton}

The following result shows that any family of $\varepsilon$-horizontal curves can be made horizontal on a neighbourhood of .
\begin{lemma} \label{lem:horizontalisationSkeleton}
Given a family $\gamma: K \rightarrow \Emb^\varepsilon(I;M,\SD)$, there exists a family 
\[ \widetilde\gamma: K \times [0,1] \rightarrow \Emb^\varepsilon(I;M,\SD) \]
such that:
\begin{itemize}
\item[i. ] $\widetilde\gamma(k,0) = \gamma(k)$.
\item[ii. ] $\widetilde\gamma(k)(t) = \gamma(k)(t)$ if $(k,t)\in (K\times I) \setminus \Op{A}$
\item[iii. ]  $\widetilde\gamma(k,s)$ is $C^0$-close to $\gamma(k)$, for all $s$.
\item[iii'. ] the length of $\widetilde\gamma(k,s)$ is close to the length of $\gamma(k)$, for all $s$.
\item[iv. ] $\widetilde\gamma(k,1)$ is horizontal close to $A$.
\end{itemize}
\end{lemma}
\begin{proof}
The proof is inductive on the strata of $A$, starting from the smallest one. At a given step, working with a stratum $B$, we will achieve Property (iv) over $B$, preserving it as well along smaller strata. The other properties will follow as long as our perturbations are small and localised close to $A$.

Let $U$ be a neighbourhood of the smaller strata in which $\widetilde\gamma(k,1)$ is already horizontal. We can then consider a closed submanifold $B' \subset B$ such that $\{B',U\}$ cover $B$. We can triangulate $B'$ using Lemma \ref{lem:jiggling}, turning it into a stratified set itself, so that each simplex $\Delta$ is mapped by $\gamma$ to some adapted chart $(V,\phi)$. We then proceed inductively from the smaller simplices. A crucial observation is that simplices along $\partial B$ are contained in $U$ and therefore no further changes are required there. 

For the inductive step consider an $l$-simplex $\Delta$. The inductive hypothesis is that there is a family of curves $\beta$ that has been obtained from $\gamma$ by a homotopy satisfying Properties (i) to (iii') and that is already horizontal over all smaller simplices (and $U$). Since $B$ is transverse to the $I$-direction, $\Delta$ has a neighbourhood parametrised as 
\[ \Phi: \D^l \times \D^{k-l} \times (-\delta,\delta) \longrightarrow K \times I. \]
The map $\Phi$ preserves the foliation in the direction of the last component and $\Phi|_{\D^l \times 0}$ is an arbitrarily small extension of $\Delta$ to a smooth disc. We write $\eta = \phi \circ \beta \circ \Phi$ for the restriction of the family to this neighbourhood, mapping now into the graphical model $V$. From the induction hypothesis it follows that $\eta$ is horizontal over $\Op(\partial \D^l) \times \D^{k-l} \times (-\delta,\delta)$.

We have thus reduced the claim to the situation in which our stratified set is just a disc, and we have to work relatively to the boundary of $\D^l \times \D^{k-l} \times (\-\delta,\delta)$. There is a unique family of horizontal curves $\nu$ such that $\nu$ and $\eta$ share the same projection to the base of $V$ and such that $\nu(k)(0) = \eta(k)(0)$ for all $k \in \D^l \times \D^{k-l}$. We can argue that this family exists for all time if our triangulation was fine enough. Alternatively, we just observe that there is some $\delta' >0$ such that $\nu$ is defined over $\D^l \times \D^{k-l} \times (-\delta',\delta')$ and $\eta$ lives within an ODE model associated to it (Lemma \ref{lem:ODEModel}).

We now deform $\eta$, relatively to the boundary of the model, to a family that agrees with $\nu$ over $\D^l \times \D^{k-l} \times (-\delta,\delta)$. We can do so keeping the projection to the base the same (Lemma \ref{lem:horizontalisationGraphical}). This, together with a sufficiently small choice of $\delta'$, guarantees Properties (iii) and (iii'). This concludes the inductive argument to handle a stratum $B$ and thus the inductive argument across all strata.
\end{proof}
It is immediate from the proof that the statement also holds relatively to regions of $A$ in which the curves are already horizontal.

\begin{corollary} \label{cor:horizontalisationSkeleton}
Assume that $\SD$ is cooriented of corank-$1$ and that $\gamma$ is positively transverse. Then, the conclusions of Lemma \ref{lem:horizontalisationSkeleton} hold and additionally $\widetilde\gamma$ can be assumed to be almost transverse. 
\end{corollary}
\begin{proof}
The horizontalisation process described in the proof of Lemma \ref{lem:horizontalisationSkeleton} was based on passing locally to some ODE model. In such a model it is immediate that introducing zero slope (making the curves horizontal) can be done while preserving non-negative slope everywhere (being almost transverse).
\end{proof}


\subsubsection{Direction adjustment}

Lemma \ref{lem:horizontalisationSkeleton} explained to us how to perturb a family of $\varepsilon$-horizontal curves so that it becomes horizontal along $A$. The next lemma states that one can prescribe the behaviour along $A$, as long as $A$ is contractible.
\begin{lemma}\label{lem:directionAdjust}
Suppose that $A$ is a $k$-disc and $\SD$ has rank at least $2$. Fix a family $\gamma: K \rightarrow \Emb^\varepsilon(I;M,\SD)$ and a family of horizontal curves $\nu$, defined only on a neighbourhood of $A$. Assume that $\nu|_A = \gamma|_A$.

Then, there exists a family $\widetilde\gamma: K \times [0,1] \rightarrow \Emb^\varepsilon(I;M,\SD)$ such that the conclusions of Lemma \ref{lem:horizontalisationSkeleton} hold and, additionally: 
\begin{itemize}
\item[v. ] $\widetilde\gamma(-,1)$ agrees with $\nu$ in $\Op(A)$.
\end{itemize}
\end{lemma}
\begin{proof}
Since the argument takes place on a neighbourhood of $A$ and is relative to its boundary, we may as well assume that $K = \Op(\D^k)$ and $A = \D^k \times \{1/2\} \subset K \times [0,1]$. 

Since $A$ is contractible and $\SD$ has rank at least $2$, we can find a tangential rotation 
\[ (v_\theta)_{\theta \in [0,1]}: \D^k \longrightarrow TM \]
such that:
\begin{itemize}
\item $v_\theta(k) \in T_{\gamma(k)(1/2)} M$ is $\varepsilon$-horizontal.
\item $v_0(x) = \gamma(k)'(1/2)$.
\item $v_1(x) = \nu(x)'(1/2)$.
\end{itemize}
That is, $v$ is a lift of $\gamma|_A$ providing a tangential rotation of its velocity vector to the velocity vector of $\nu$.

A further simplification enters the proof now: $\gamma$ may be assumed to take values in a graphical model $V$. Otherwise we triangulate $A$ in a sufficiently fine manner and argue inductively on neighbourhoods of its simplices. There is then a homotopy of linear maps
\[ (\Phi_\theta)_{\theta \in [0,1]}: \D^k \longrightarrow \GL(\R^q,\R^q) \]
that satisfies $\Phi_0(k) = \Id$ and $\Phi_\theta(k)(d\pi(\gamma(k)'(1/2))) = d\pi(v_\theta(k))$. It exists due to the homotopy lifting property. It provides us with a rotation of $\R^q$ extending the tangential rotation $d\pi \circ v_\theta$.

Let $\chi$ be a cut-off function that is $1$ on a neighbourhood of $A$ and zero away from it. Consider the homotopy of curves given by
\[ \alpha_\theta(k)(t) := \Phi_{\chi(k,t)\theta}(k)(\pi \circ \gamma(k)(t)). \]
We claim that $\alpha_\theta(k)$ is in fact embedded. This will indeed be the case if the support of $\chi$ is sufficiently small, since the curves $\pi \circ \gamma|_{\Op(A)}$ are then small embedded intervals resembling a straight line.

By construction, $\alpha_1(k)$ is tangent to $\pi \circ \nu(k)$ at $t=1/2$. This allows us to define a further homotopy $(\alpha_\theta)_{\theta \in [1,2]}$ so that $\alpha_2(k)(t) = \pi \circ \nu(k)(t)$ for every $t \in \Op(\{1/2\})$. This latter homotopy may be assumed to be $C^1$-small and supported in an arbitrarily small neighbourhood of $A$. 

Over $\Op(A)$, we have that $\pi \circ \gamma$ and $\alpha_\theta$ are $C^0$-close and of similar length. It follows from Lemma \ref{lem:varepsilonStability} that there is a family of $\varepsilon$-horizontal curves $\beta_\theta$ lifting $\alpha_\theta$, that is $C^0$-close to $\gamma$. Applying Lemma \ref{lem:varepsilonInterpolation} allows us to assume that $\beta_\theta$ agrees with $\gamma$ outside a neighbourhood of $A$. We can then apply Lemma \ref{lem:horizontalisationSkeleton} to $\beta_2$ in order to horizontalise. This yields a homotopy to some $\varepsilon$-horizontal family $\beta_3$ that close to $A$ agrees with $\nu$.
\end{proof}

%

\subsection{Transversalisation} \label{ssec:transversalisation}

In this subsection we explain the transverse analogues of the results presented in Subsection \ref{ssec:horizontalisation}. We fix a distribution $(M,\SD)$. We write $K$ for a compact manifold and $I$ for $[0,1]$. $A \subset K \times I$ is a stratified set transverse to the second factor.

\subsubsection{Transversalisation of almost-transverse curves}

The following lemma explains that almost transverse curves can be pushed slightly to become transverse.
\begin{lemma}\label{lem:transversalisationAT}
Suppose $\SD$ is of corank $1$. Given a family of curves $\gamma: K \longrightarrow \EmbAT([0,1];M,\SD)$, there exists a $C^1$-deformation 
\[ \widetilde\gamma: K \times [0,1] \longrightarrow \EmbAT([0,1];M,\SD) \]
 such that
\begin{itemize}
\item $\widetilde\gamma(k,0) = \gamma(k)$.
\item $\widetilde\gamma(k,1)$ is transverse.
\item Assume that $\gamma$ is transverse along $\Op(\partial(K \times I))$. Then this homotopy is relative to the boundary.
\end{itemize}
\end{lemma}
\begin{proof}
The argument is carried out one adapted chart at a time. If $K \times I$ is covered by sufficiently small opens, we can pass to ODE charts (Subsection \ref{ssec:ODEModels}), where the statement is obvious and relative.
\end{proof}
Do observe that this process may not be assumed to be relative if the starting family was purely horizontal. In fact, the argument will certainly displace the endpoint of the curves upwards.

\begin{figure}[h] 
	\includegraphics[scale=0.35]{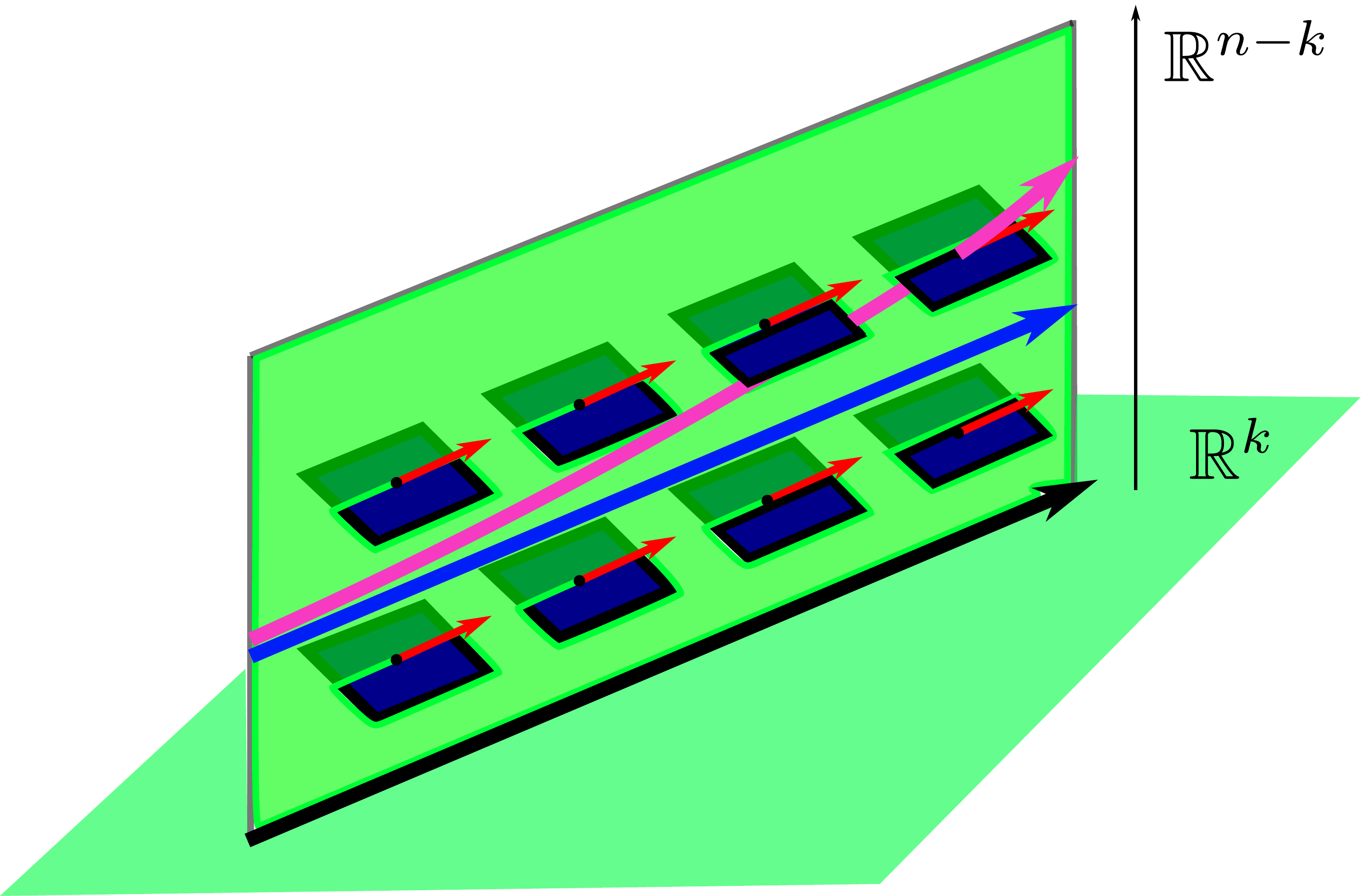}
	\centering
	\caption{We depict transversalisation of a horizontal curve. On a graphical model, it is easy to construct the desired transverse curve $\widetilde\gamma(-,1)$ (in magenta) by adding a small slope. This is not relative to the final endpoint.}\label{ConnectionFlowbox}
\end{figure}

\begin{figure}[h] 
	\includegraphics[scale=0.6]{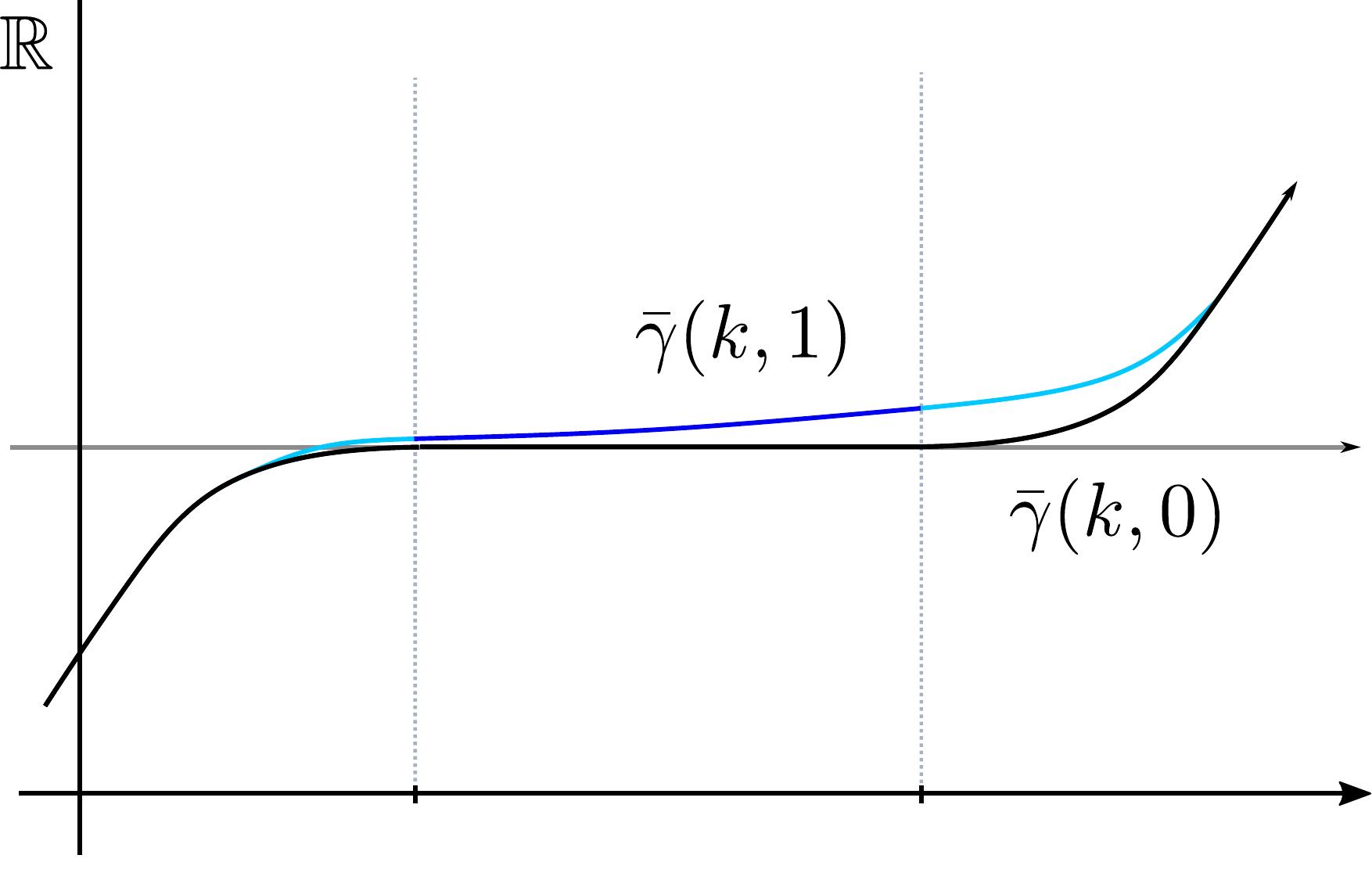}
	\centering
	\caption{Transversalisation for almost horizontal curves, relative to the boundary. The argument reduces to adding slope in an ODE model.}\label{HorizontalToTransverse}
\end{figure}

\begin{remark}
From this lemma it follows that there is a weak homotopy equivalence between $\Embt(M,\SD)$ and the subspace of $\EmbAT(M,\SD)$ consisting of curves that are somewhere (positively) transverse. This can be refined further to include those curves of $\EmbAT(M,\SD)$ that are regular horizontal. We leave this as an exercise for the reader.
\end{remark}

\subsubsection{Transversalisation of $\varepsilon$-transverse curves}

The following lemma achieves the transverse condition in a neighbourhood of $A$.
\begin{lemma} \label{lem:transversalisationSkeleton}
Suppose that $\SD$ is of corank-$1$ and cooriented. Given a family $\gamma: K \rightarrow \Embt^\varepsilon(I;M,\SD)$, there exists a family 
\[ \widetilde\gamma: K \times [0,1] \rightarrow \Embt^\varepsilon(I;M,\SD) \]
such that:
\begin{itemize}
\item[i. ] $\widetilde\gamma(k,0) = \gamma(k)$.
\item[ii. ] $\widetilde\gamma(k)(t) = \gamma(k)(t)$ if $(k,t)\in (K\times I) \setminus \Op{A}$
\item[iii. ]  $\widetilde\gamma(k,s)$ is $C^0$-close to $\gamma(k)$, for all $s$.
\item[iii'. ] the length of $\widetilde\gamma(k,s)$ is close to the length of $\gamma(k)$, for all $s$.
\item[iv. ] $\widetilde\gamma(k,1)$ is transverse close to $A$.
\end{itemize}
\end{lemma}
\begin{proof}
As in the proof of Lemma \ref{lem:horizontalisationSkeleton}, we proceed inductively on the strata of $A$, each of which is in turn processed one simplex at a time. This reduces the proof to the analogous statement in which $(M,\SD)$ is a graphical model, $K$ is $\D^k$, and the curves of $\gamma$ have arbitrarily small length and image. Due to the $\varepsilon$-transverse condition, we have that the curves $\gamma(k)$ are then either (positively) vertical with respect to the base projection, in which case we do not need to do anything, or graphical over $\SD$. In the latter case we work in an ODE model and add positive slope. This is relative in the parameter and domain.
\end{proof}


\subsubsection{Transversalisation of formally transverse embeddings}

We also need a transversalisation statement, in the spirit of Lemma \ref{lem:transversalisationSkeleton}, that applies instead to formal transverse embeddings:
\begin{lemma} \label{lem:transversalisationFormalSkeleton}
Given a family $\gamma: K \rightarrow \Embt^f(I;M,\SD)$, there exists a family $\widetilde\gamma: K \times [0,1] \rightarrow \Embt^\varepsilon(I;M,\SD)$ such that:
\begin{itemize}
\item[i. ] $\widetilde\gamma(k,0) = \gamma(k)$.
\item[ii. ] $\widetilde\gamma(k)(t) = \gamma(k)(t)$ if $(k,t)\in (K\times I) \setminus \Op{A}$
\item[iii. ]  $\widetilde\gamma(k,s)$ is $C^0$-close to $\gamma(k)$, for all $s$.
\item[iii'. ] the length of $\widetilde\gamma(k,s)$ is close to the length of $\gamma(k)$, for all $s$.
\item[iv. ] $\widetilde\gamma(k,1)$ is transverse close to $A$.
\end{itemize}
\end{lemma}
\begin{proof}
We work inductively over the strata of $A$ and inductively over the simplices of a triangulation of each stratum. This reduces the proof to a local and relative statement happening in an adapted chart. Then, the conclusion follows as in the proof of Lemma \ref{lem:directionAdjust}. Namely, the tangential homotopy given by $\gamma$ can be used to rotate the velocity vectors of $\gamma$ along $A$ to make them transverse to $\SD$.
\end{proof}

\section{Tangles}\label{sec:tangles}

In this section we introduce \emph{tangles}. These are particular local models for curves in the base $\R^q$ of a graphical model $(V,\SD)$. Upon lifting, they act as building blocks for horizontal curves. The reader should think of them as analogues of the stabilization in Contact Topology, seen in the Lagrangian projection.

\begin{remark}
We have chosen the name ``tangle'' because they are reminiscent of tangles in $3$-dimensional Knot Theory \cite{Con}. In the classical sense, a tangle $(\D^3, \T)$ consists of a ball $\D^3$ with a finite number of properly imbedded disjoint arcs $\T$. This allows for the factorisation of knots into elementary pieces \cite{Blei}. 
\end{remark}
Our tangles are similar: they are presented as boxes containing a homotopy of curves, with fixed endpoints. This allows us to attach them to any given family of curves in $\R^q$.

The construction of a tangle amounts to concatenating suitable flows and smoothing the resulting flowlines, taking care of the embedding condition of the lift. This is the natural approach: afterall, the bracket-generating condition explains how to produce motion in arbitrary directions by considering commutators of flows tangent to the distribution. We recommend that the reader takes a quick look at Appendix \ref{Appendix}, which recaps some elementary results in this direction. In Subsection \ref{ssec:flows} we introduce some further notation about bracket-expressions and concatenating flows.

\emph{Pretangles} are defined in Subsection \ref{ssec:pretangles}. These are simply curves in $\R^q$ given as flowlines of commutators of coordinate vector fields. These curves are just piecewise smooth. In order to address this, we introduce smoothing. This is done, for simple bracket-expressions, using \emph{s-pretangles} (Subsection \ref{sssec:spretangles}). We then introduce attaching models (Subsection \ref{ssec:attachingModels}) which will allow us to smooth out more complicated configurations of curves (Subsections \ref{ssec:pretangleModelsL2} and \ref{ssec:tanglesHigherLength}). These are shown to be embedded and we explain how to insert them into existing curves. In Subsection \ref{ssec:areaIsotopy} we explain how to manipulate these models to adjust the lifting of their endpoints. 

Tangles are finally introduced in Subsection \ref{ssec:tangles}.

\subsection{Flows} \label{ssec:flows}

This subsection introduces some of the notation about flows that will be used later in this section.

\subsubsection{Concatenation of flows}

Let $\phi_t$ be a flow, possibly time-dependent. We write 
\[ (\phi_{a\to b})_t := \phi_{t+a} \circ \phi^{-1}_a \]
for the flow in the interval $[a,b]$, shifted so that $\phi^{a\to b}_0$ is the identity.

Fix a second flow $\psi_t$ and real numbers $a< b$ and $c < d$. Then, we define the \textbf{concatenation} of $\phi_{a\to b}$ and $\psi_{c\to d}$ to be:
\[ \left(\phi_{a \to b} \quad\#\quad \psi_{c \to d} \right)_t := \begin{cases} 
     (\phi_{a\to b})_t & t\in\left[0,b-a\right], \\
		(\psi_{c\to d})_{t-(b-a)} \circ	(\phi_{a\to b})_{b-a} & t\in\left[b-a,(b-a)+(d-c)\right]. \\
   \end{cases}
\]
This is a time-dependent flow that is piecewise smooth in $t$, due to the switch at $t=b-a$.

In general, given flows $(\phi_i)_{i=1,\cdots,k}$ and real numbers $(a_i < b_i)_{i=1,\cdots,k}$ we can iterate the previous construction:
\[ \#_{i=1,\cdots,l}\quad (\phi_i)_{a_i \to b_i} :=
   \left( \#_{i=1,\cdots,l-1}\quad (\phi_i)_{a_i \to b_i} \right) \quad\#\quad (\phi_k)_{a_k \to b_k}. \]

\subsubsection{Generalised bracket-expressions}

We now generalise the formal bracket-expressions from Subsection \ref{sssec:LieFlag}. The aim is to consider iterates of formal bracket-expressions.
\begin{definition}
We say that the pair $(a,k)$, written as $a^{\#k}$, depending on the variable $a$ and the integer $k$, is a \textbf{generalised bracket expression} of length $1$. Similarly, we say that the expression $[a_1,a_2]^{\#k}$, depending on the variables $a_1$, $a_2$ and the integer $k$, is a generalise bracket expression of length $2$. Inductively, we define a generalised bracket expression of length $n$ to be an expression of the form 
\[ [A(a_1,\cdots,a_j),B(a_{j+1},a_n)], \qquad 0 < j < n \]
with $A$ and $B$ generalised bracket expressions of lengths $j$ and $n-j$, respectively.
\end{definition}

\subsection{Pretangles and s-pretangles} \label{ssec:pretangles}

We now fix a graphical model $(V,\SD)$. All the constructions in this section take place within it. We write $\pi: V \rightarrow \R^q$ for its projection to the base. The framing reads $\{X_1,\cdots, X_n\}$, with $\{X_1,\cdots, X_q\}$ a framing of $\SD$ lifting the coordinate framing $\{\partial_1,\cdots, \partial_q\}$ of $\R^q$. We write $\phi_i$ for the flow of $\partial_i$, here $i=1,\cdots,q$. The flow of $X_i$ is denoted by $\Phi_i$.

\subsubsection{Pretangles}

The following construction produces time-dependent flows that are iterates of a given commutator:
\begin{definition}
Let $A = [a,b]^{\# m}$ be a generalised bracket expression of length $2$. Let $\phi$ and $\psi$ be flows. We define
\[ A(s) := \left(\left(\phi_{0\to\frac{s}{\sqrt{m}}}\right) \# \left(\psi_{0\to\frac{s}{\sqrt{m}}}\right) \#
										 \left(\phi_{0\to\frac{s}{\sqrt{m}}}^{-1}\right) \# \left(\psi_{0\to\frac{s}{\sqrt{m}}}^{-1}\right)\right)^{\# m}.
\]
\end{definition}

We can introduce an analogous definition for bracket expressions of greater length, inductively:
\begin{definition}
Let $B = [a_1, A(a_2,\cdots,a_l)]^m$ a generalised bracket expression. Consider flows $(\phi_i)_{i=1,\cdots,l}$. Then we denote:
\[ B(s) :=
 \left(\left(\phi_{0\to\frac{s}{\sqrt{m}}}\right) \#  A(s/\sqrt{m}) \#
 \left(\psi_{0\to{\frac{s}{\sqrt{m}}}}^{-1}\right) \# A(s/\sqrt{m})^{-1} \right)^{\# m}.
\]
\end{definition}

The following is the main definition of this subsection:
\begin{definition}
Let $A$ be a generalised bracket-expression with $\phi_1,\cdots, \phi_\ell$ as inputs. An integral curve, depending on the parameter $s$, of the flow defined by the expression $A(s)$ is called a \textbf{pretangle}. We will denote such a curve by $\gamma^{A(s)}$.
\end{definition}

\subsubsection{S-pretangles} \label{sssec:spretangles}

We will introduce the notion of \textbf{S-pretangles}, where the ``s'' stands for ``smooth''.

We can describe a way of smoothing the corners where the previously defined pretangles failed to be smooth. We will define two different ways of smoothing a corner. Define first the following time-dependent vector field:
\[Y_t^{(i,j),\delta} :=  \dfrac{\delta-t}{\delta}\cdot X_j+\dfrac{t}{\delta}\cdot X_{j}, \phantom{aa} t\in\left[0, \delta\right].\]

Denote by $s^{i,j}_{0\to\delta}$ the flow associated to the vector field $Y_t^{(i,j),\delta}$. Note that the concatenation of flows $\phi^i_{0\to\eta-\delta/2} \: \# \: s^{i,j}_{0\to\delta} \: \# \: \phi^j_{\delta/2\to\tau}$ or, in short, $\phi^i_{0\to\eta-\delta/2}\:\#^{\delta}_{i,j} \: \phi^j_{\delta/2\to\tau}$,  can be made $C^\infty-$close to $ \phi^{i}_{0\to\eta}\#\phi^{j}_{0\to\tau}$ by taking $\delta$ small enough:

\[\phi^i_{0\to\eta-\delta/2}\:\#_{i,j}^{\delta} \: \phi^j_{\delta/2\to\tau}\xrightarrow[\delta \to 0]{||\cdot||_{C^\infty}}\phi^{i}\:\#\:\phi^{j}.\]

\begin{figure}[h] 
	\includegraphics[scale=0.4]{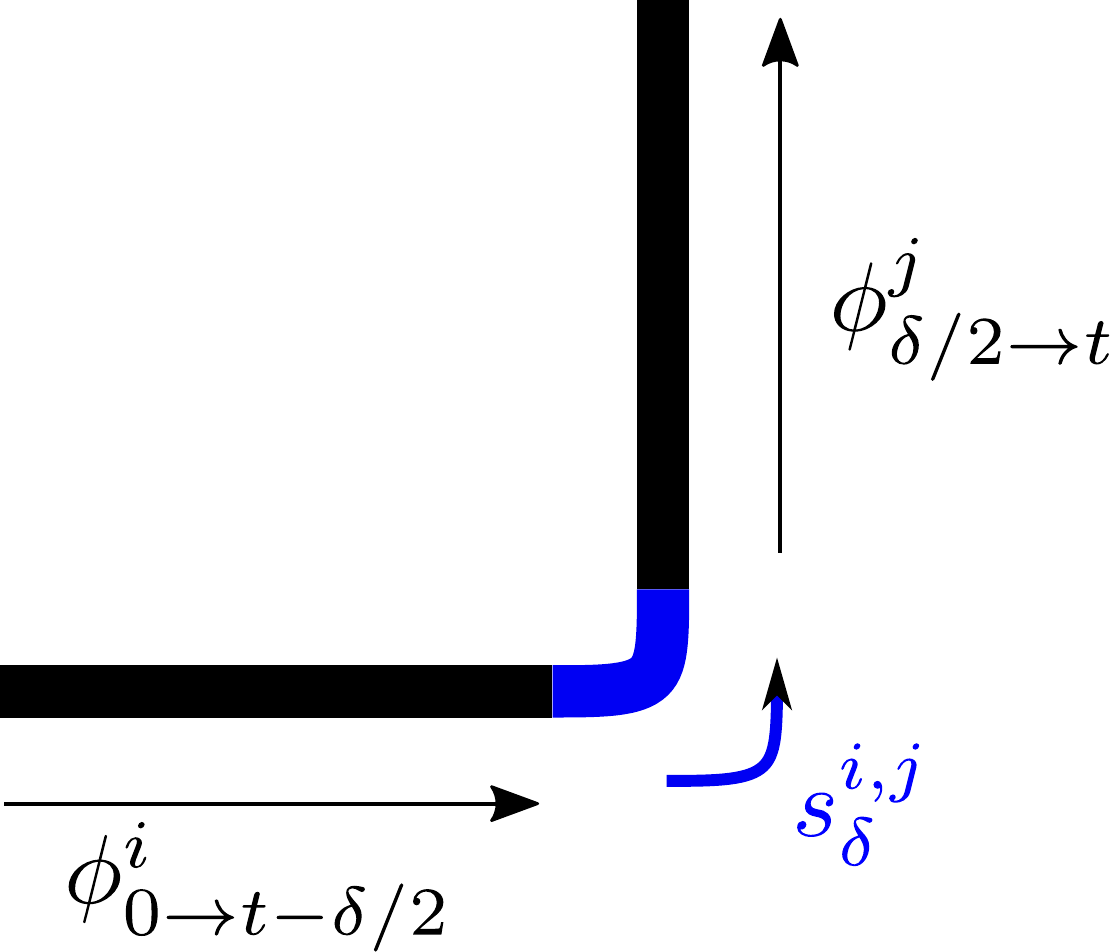}
	\centering
	\caption{Schematic description of $\phi^i_{0\to t-\delta/2} \: \#_{i,j}^\delta \: \phi^j_{\delta/2\to t}$. The flow $s^{i,j}_\delta$ provides a way of smoothing the previously defined concatenation of two given vector field flows. }\label{smoothing1}
\end{figure}

The flows $s^{i,j}_\delta$ play the role of smoothing the concatenation of two vector field flows when concatenated in between. Indeed, note that $\phi^i_{0\to\eta-\delta/2} \: \#^\delta_{i,j} \: \phi^j_{\delta/2\to\tau}$ is a $C^\infty-$flow.

Let us introduce now a different way of smoothing a corner. Denote by $\delta':=\delta/4-\delta/50$, $\delta'':=\delta/25$ and $\delta'''=\delta/50$. Consider the following flow (see Figure \ref{smoothing2}):

\[d^{j,i}_\delta:=\phi^j_{0\to\delta'} \: \#_{j,i}^{\delta''} \: \phi^i_{\delta'''\to\delta'} \: \#^{\delta''}_{i,-j}\: \phi^{-j}_{\delta'''\to\delta'} \: \#^{\delta''}_{-j,-i} \:\phi^{-i}_{\delta'''\to\delta/4}\]

\begin{figure}[h] 
	\includegraphics[scale=0.4]{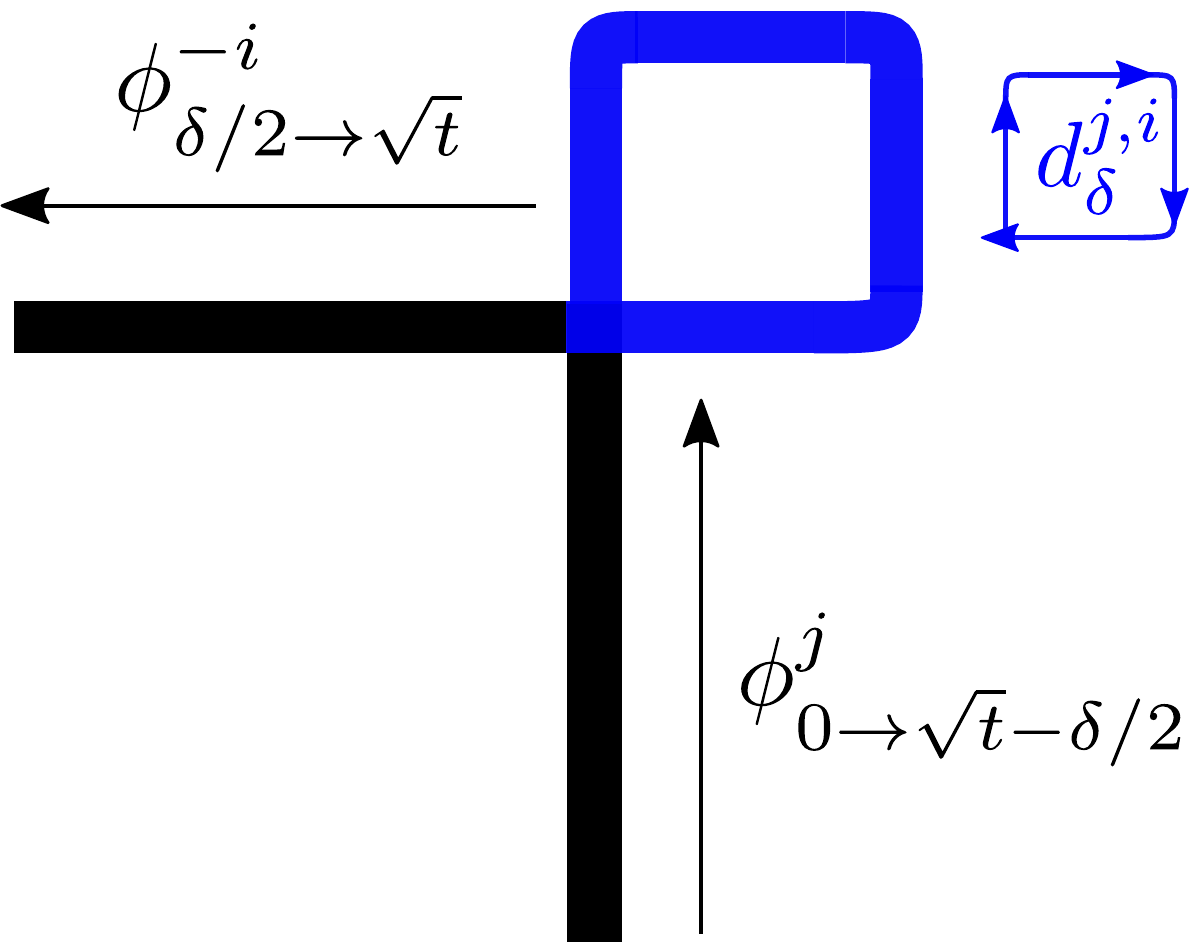}
	\centering
	\caption{Schematic description of $\phi^j_{0\to\sqrt{t}-\delta/2}\#d^{i,j}_\delta\#\phi^j_{\delta/2\to\sqrt{t}}$. The flow $d^{i,j}_\delta$ provides a essentially different way of smoothing the previously defined concatenation of two given vector field flows. }\label{smoothing2}
\end{figure}

We now define smooth pretangles $\gamma_\mathcal{SPT}^{t,\delta}$ for two given vector field flows $\phi^i_t, \phi^j_t$, as the curve in Figure \ref{fig:pretangleModels}, defined as a pretangle for $\left[\phi^i, \phi^j\right]$ whose corners have been smoothed by using the flows $\#^\delta_{\pm i, \pm j}$ and $d^{j, -i}$ (see the upright corner). A precise formula for the curve can be given. Indeed, $\gamma_\mathcal{SPT}^{t, \delta}$ is an integral curve of the flow:

\[\mathcal{SPT}^{t,\delta}_{\left[\phi^i, \phi^j\right]}:=\phi^i_{0\to t-\delta/2}\: \#^{\delta}_{i,j}\: \phi^j_{\delta/2\to t-\delta/2} \: \# \: d^{j, -i}_\delta \: \# \: \phi^{-i}_{\delta/2\to t-\delta/2}\: \#_{-i,-j}^{\delta}\: \phi^{-j}_{\delta/2\to t-3\delta/2}\: \#^\delta_{-j,i}\#  \phi^{i}_{\delta/2\to t-\delta}\#_{i,-j}^{\delta}\#_{-j,i}^{\delta}\]

We refer to $\delta$ as the \textbf{smoothing parameter of the s-pretangle} $\mathcal{SPT}^{t,\delta}_{\left[\phi^i, \phi^j\right]}$. (See Figure \ref{fig:pretangleModels}, where an integral curve of such a flow is depicted inside the grey box).



\subsection{Attaching models} \label{ssec:attachingModels}


As we saw earlier, pretangles can be interpreted as a local model that can be attached to a family of curves in the base of a graphical model in order to quantitatively control the endpoints of the lift (see Proposition \ref{EndpointBracketCurves}). Nonetheless, we face two fundamental problems when we try to do that: these curves are not smooth. Furthermore, they are not (topologically) embedded and it is not readily apparent whether their lifts are embedded.

We now introduce some alternate models by carefully modifying our prior constructions. These models will depend on certain small ``smoothing parameters'' and will converge to pretangles, in the $C^0$-norm, as we make these parameters tend to zero.

The general definition reads:
\begin{definition} \label{def:attachingModel}
Let $\gamma:I\to \R^q$ be a curve that is integral for a coordinate vector field $X_i$ in the adapted frame. We call an \textbf{attaching model} with axis $\partial_i$ to a choice of: 
\begin{itemize}
\item[i)] a size $\eta>0$,
\item[ii)] two attaching points $p_1=\gamma(t_1), p_2=\gamma(t_2)$ at distance $d_g(p_1, p_2) = \eta$,
\item[iii)] a hypercube $\SB(\eta) \subset \R^q$, called the box of the model, of side $2\eta$ so that $p_1$ and $p_2$ are in opposite faces of $\SB$.
\item[iv)] a curve $\beta$ with endpoints $p_1, p_2$  satisfying:
\begin{itemize}
\item[iv.a)] its image lies inside $\SB(\eta)$.
\item[iv.b)] The curve $\tilde\gamma: I\to \R^q$ defined as $\tilde\gamma|_{I\setminus[t_1, t_2]}=\gamma|_{I\setminus[t_1, t_2]}$, $\tilde\gamma|_{[t_1, t_2]}=\beta$ is continuous. If it has $C^r$ regularity, we say that the model is $C^r-$regular.
\end{itemize}
\end{itemize}
\end{definition}

\begin{figure}[h] 
	\includegraphics[scale=0.2]{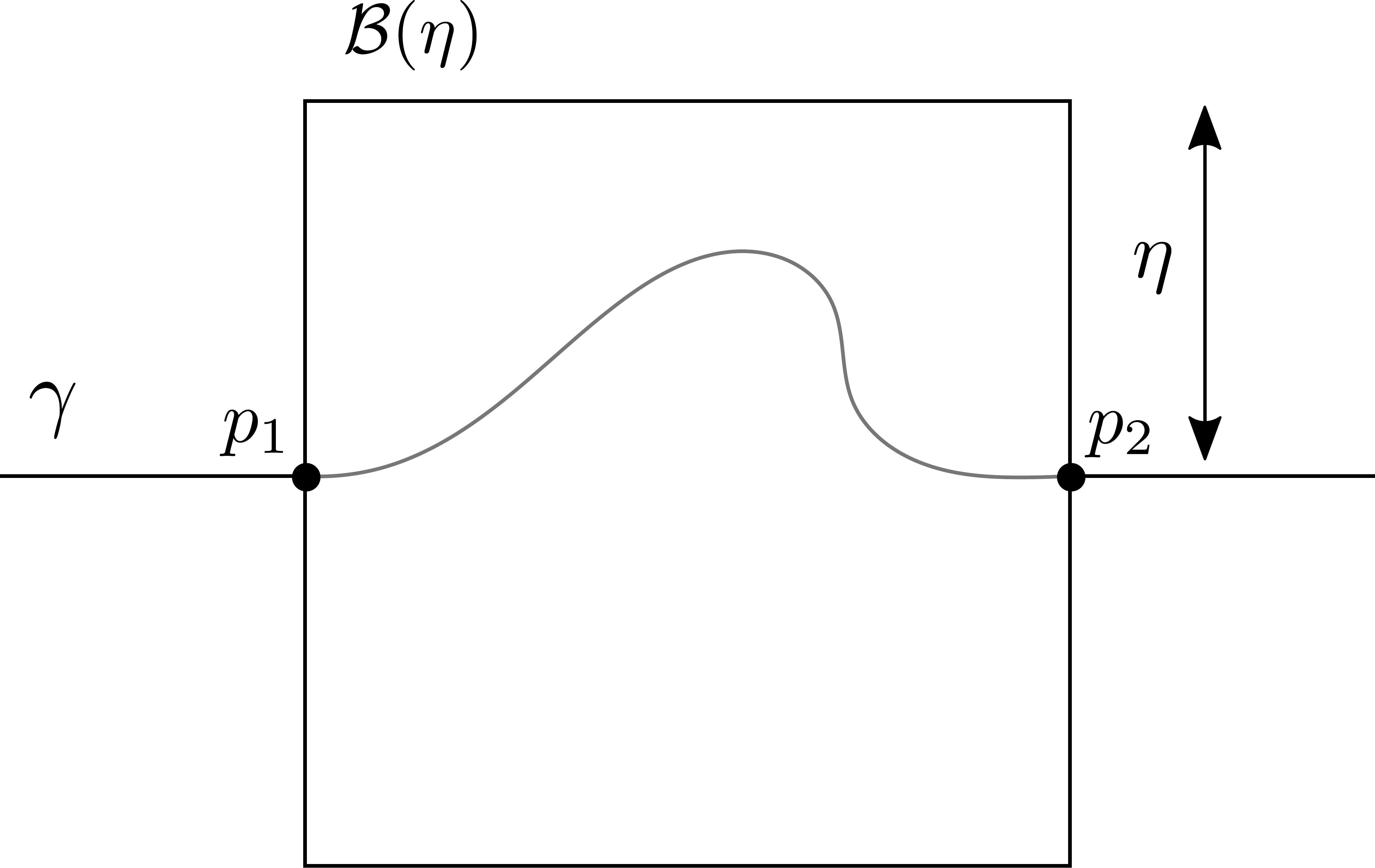}
	\centering
	\caption{Schematic depiction of an \textbf{attaching model} for certain choice of $\eta>0$, attaching points $p_1, p_2$, a hypercube $\mathcal{B}(\eta)$ and a curve $\beta$.}\label{BoxModel}
\end{figure}

\subsubsection{Pretangle models} \label{sssec:pretangleModels}

A concrete instance of Definition \ref{def:attachingModel} to be used in the next section reads:
\begin{definition}
Let $A$ be a generalised bracket expression of the form $A(\phi_i,\cdots,\phi_n)=[\phi_i, B(\phi_\ell,\cdots\phi_n)]^{\#k}$ with inputs flows $\phi_i,\cdots, \phi_n$, and where $B$ is a bracket expression of smaller length. A \textbf{pretangle model} associated to $A$ is an attaching model where  the curve $\beta:[t_1,t_2]\to \R^q$ inside the box satisfies:
\begin{itemize}
\item it is a pretangle for $t\in(t_1+\epsilon, t_2-\epsilon)$,
\item it coincides with the straight segment in the direction $\partial_i$ joining $p_1,p_2$ for $t\in(t_1, t_1+\epsilon)\cup(t_2-\epsilon, t_2)$.
\end{itemize} 
\end{definition}
We call length of the pretangle model to the length of the pretangle inside the box.

\begin{figure}[h] 
	\includegraphics[scale=0.7]{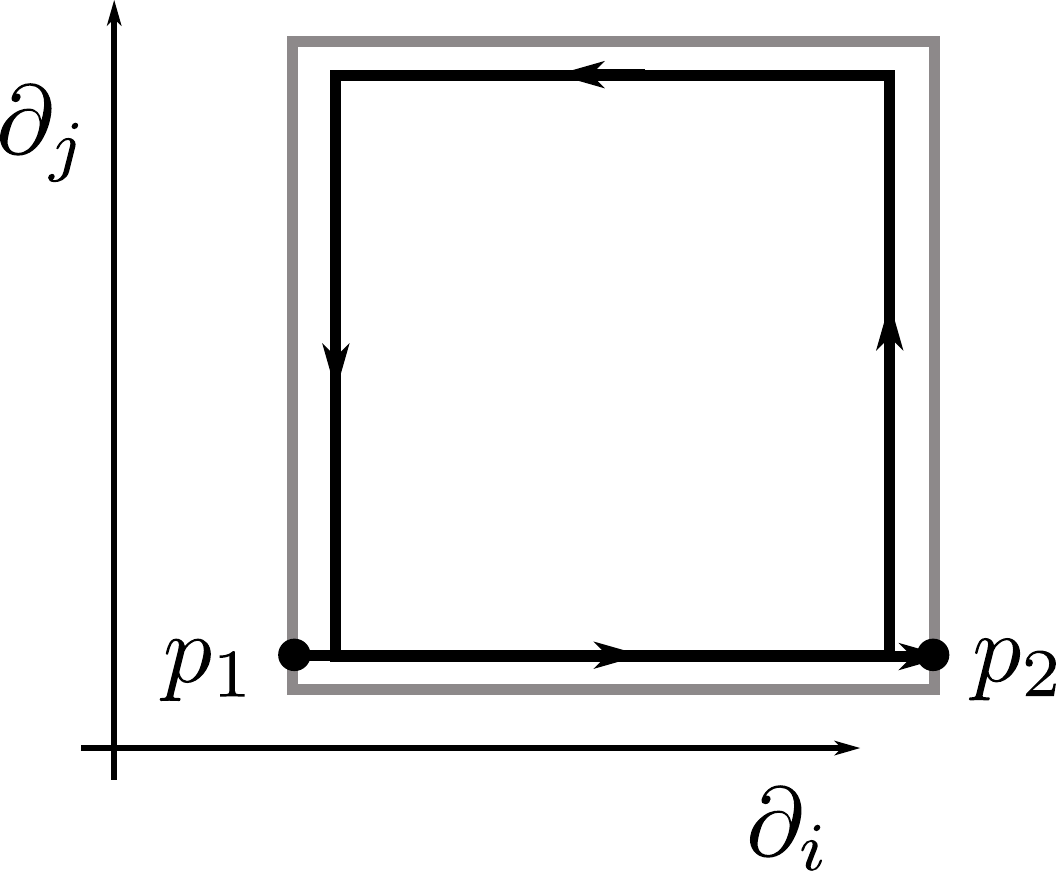}
	\centering
	\caption{Pretangle model.}\label{fig:pretangleModels}
\end{figure}

\subsection{Length-$2$ tangle models} \label{ssec:pretangleModelsL2}

We now introduce the building blocks of the main objects of interest in the section. We present first a specific type of attaching model that we call basic length $2$ tangle model. These are meant to be better behaved than pretangle models, whose regularity is only $C^0$.
\begin{definition}
Let $A$ be a generalised bracket expression of the form $A(\phi_i,\cdots,\phi_n)=[\phi_i, B(\phi_\ell,\cdots\phi_n)]^{\#k}$ with inputs flows $\phi_i,\cdots, \phi_n$, and where $B$ is a bracket expression of smaller length. A \textbf{length-$2$ base tangle model} associated to $A$  is the attaching model described by Figure \ref{Length2BasicModel}. 

The curve $\beta_\delta:[t_1,t_2]\to \R^q$ inside the box is immersed and satisfies:
\begin{itemize}
\item  it is a smooth pretangle $\mathcal{SPT}^{\mu,\delta}_{\left[\phi^i, \phi^j\right]}$ for $t\in(t_1+\epsilon, t_2-\epsilon)$,
\item it is $\delta-$close to the straight segment in the direction $\partial_i$ joining the points $p_1,p_2$ for $t\in(t_1, t_1+\epsilon)\cup(t_2-\epsilon, t_2)$,
\item the size of the box of the attaching model is $\mu+2\delta$.
\end{itemize}
We say that $\delta$ is the smoothing parameter of the model.
\end{definition}
Whenever it is clear from the context we will refer as the tangle to the curve $\beta$ in the tangle model. 

\begin{figure}[h] 
	\includegraphics[scale=0.7]{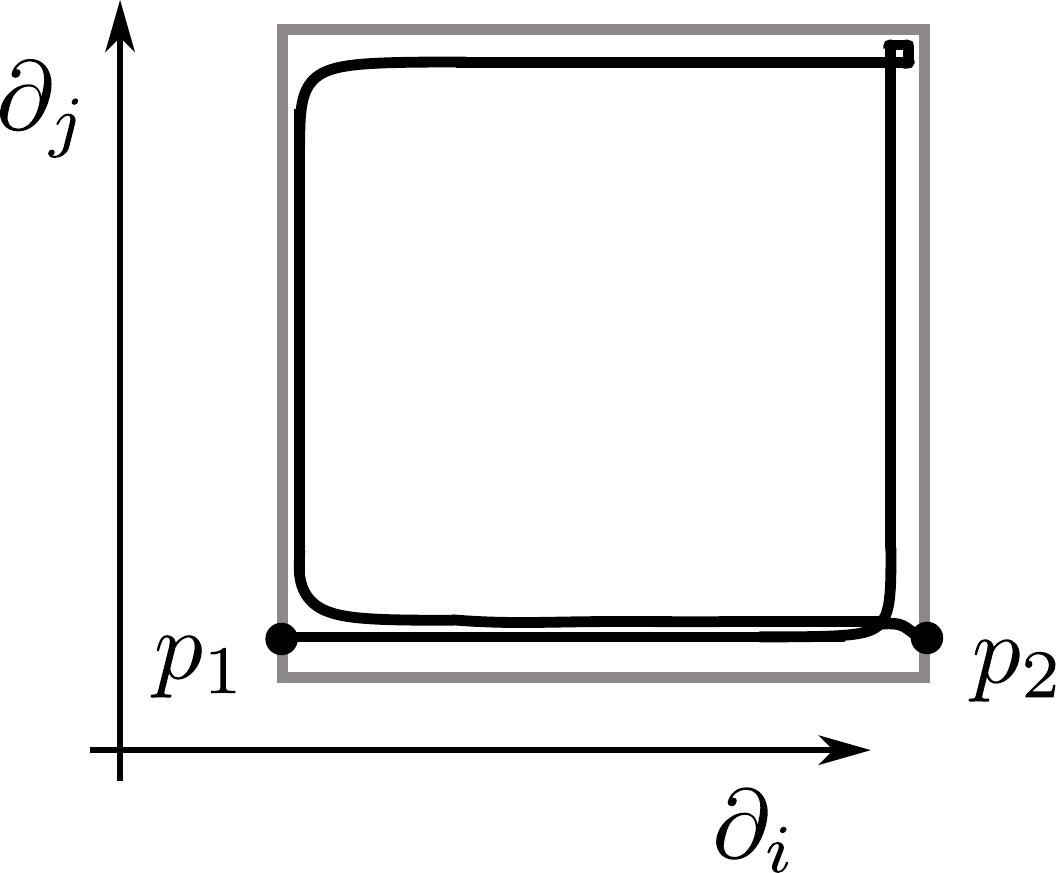}
	\centering
	\caption{Length $2$ base tangle model. The curve inside the box is a smooth pretangle $\mathcal{SPT}^{\mu,\delta}_{\left[\phi^i, \phi^j\right]}$ for $t\in(t_1+\epsilon, t_2-\epsilon)$ and coincides with the straight segment in the direction $\partial_i$ joining the points $p_1,p_2$ elsewhere.}\label{Length2BasicModel}
\end{figure}

\begin{remark}
Note that any length $2$ pretangle model can be $C^0-$approximated by a length 2 base tangle model by taking the smoothing parameter $\delta$ small enough.
\end{remark}

\subsubsection{Birth homotopy for length-$2$ base tangle models}

Our goal is to present a homotopy that introduced a length-$2$ base tangle. We first the following result:
\begin{proposition}\label{prop:AlzadoAreas}
Assume $[X_i, X_j=B(X_1,\cdots,X_{j})]=X_z$. Denote by $\alpha_z$ be the covector dual to $X_z$.

Then, any curve $\gamma:\NS^1\to\R^q$ enclosing area $A$ in the plane $\R^2=\langle \partial_i,\partial_j\rangle$ lifts to $\SD$ as a curve $\tilde{\gamma}:[0,1]\to \R^n$ satisfying
\[\int_{\gamma}\alpha_z=A\left(1+O(r)\right).\]
\end{proposition}
\begin{proof}
Denote by $\Gamma_{(\tilde{\gamma}(1), \tilde{\gamma}(0))}$ the oriented segment connecting the points $\tilde{\gamma}(1)$ and $\tilde{\gamma}(0)$. Denote by $\beta:=\tilde{\gamma}\#\Gamma_{(\tilde{\gamma}(1), \tilde{\gamma}(0))}$  the concatenation of the curves $\tilde{\gamma}$ and $\Gamma_{\left(\tilde{\gamma}(1), \tilde{\gamma}(0)\right)}$. Note that 
\[\int_{\beta}\alpha_z=-\int_{\Gamma_{\left(\tilde{\gamma}(1), \tilde{\gamma}(0)\right)}}\alpha_z\]
and, thus, this integral measures the difference of the $\partial_z$-coordinate values of the points $\tilde{\gamma}(1)$ and $\tilde{\gamma}(0)$. 

Consider a topological disk $\SD_{\tilde{\gamma}}$ bounded by $\beta$ and whose boundary gets projected to $\gamma$ in the projection onto the plane $\langle\partial_i, \partial_j\rangle$. By Stokes' theorem,
\[\int_{\beta}\alpha_z=\int_{\SD_{\tilde{\gamma}}}d\alpha_z.\]
By Cartan's formula we have that
\[d\alpha_z(X_i,X_j)=\alpha_z\left(\left[X_j,X_i\right]\right).\]
Thus, if we particularize this equation at the point $p\in M$, we get that
\[d\alpha_z(p)(X_i(p),X_j(p))=1,\]
and it vanishes when evaluated at any other combination of two elements of the framing associated to the coordinate chart. Thus, the $2-$form $d\alpha$ coincides with  $dx_i\wedge dx_j$ in the origin at the level of $0-$jets. As an application of Taylor's Remainder Theorem we get that
\[\int_{\SD_{\tilde{\gamma}}}d\alpha_z=\int_{\SD_{\tilde{\gamma}}}dx_i\wedge dx_j+O(r)=A+A\cdot O(r), \]
yielding the claim.
\end{proof}

Then the following statement holds:
\begin{proposition}[Birth homotopy for length-$2$ base tangle models] \label{prop:introtanglesbase}
Let $X_i, X_j$ be two elements in the adapted framing such that $[X_i(p), X_j(p)]=X_{\ell}(p)$. Let $\gamma:[-\delta, \delta]\to \R^q$ be a horizontal curve in a graphical model. 

There exists a homotopy of embedded horizontal curves $(\gamma^u)^{u\in[0,1]}$ such that: 
\begin{itemize}
\item[i)] $\gamma^0=\gamma$.
\item[ii)] $\pi\circ\gamma^{1}|_{[-\delta/2, \delta/2]}$ is endowed with a length-$2$ base tangle model associated to the generalised bracket expression $[\phi_i, \phi_j]$.
\item[iii)] $\pi\circ\gamma^u(t)=\pi\circ\gamma(t)$ for $t\in\Op(\lbrace -\delta, \delta\rbrace)$ and all $u\in[0,1]$.
\end{itemize}
\end{proposition}
Property $iii)$ guarantees that this homotopy, when projected into the base, is relative to both endpoints. This, in particular, implies that the lifted homotopy $\gamma^u$ through horizontal curves is also relative to the starting point.

\begin{proof}
We construct the homotopy in the base; i.e. we will define $\pi\circ\gamma^u$ and, being the lifting onto the connection unique, the claim will follow. 

Since iterated models are constructed iteratively on the base length $2$ model, it suffices to show the result for that the latter.
 We first locally homotope $\pi\circ\gamma$ to an integral curve for $X_i$ in the base, any segment around $\gamma(t_0)$ is as described by the first frame in Figure \ref{tanglebase}. Consider the local isotopy of immersed cuves in the base described by the Figure \ref{tanglebase}. 

\begin{figure}[h] 
	\includegraphics[scale=0.1]{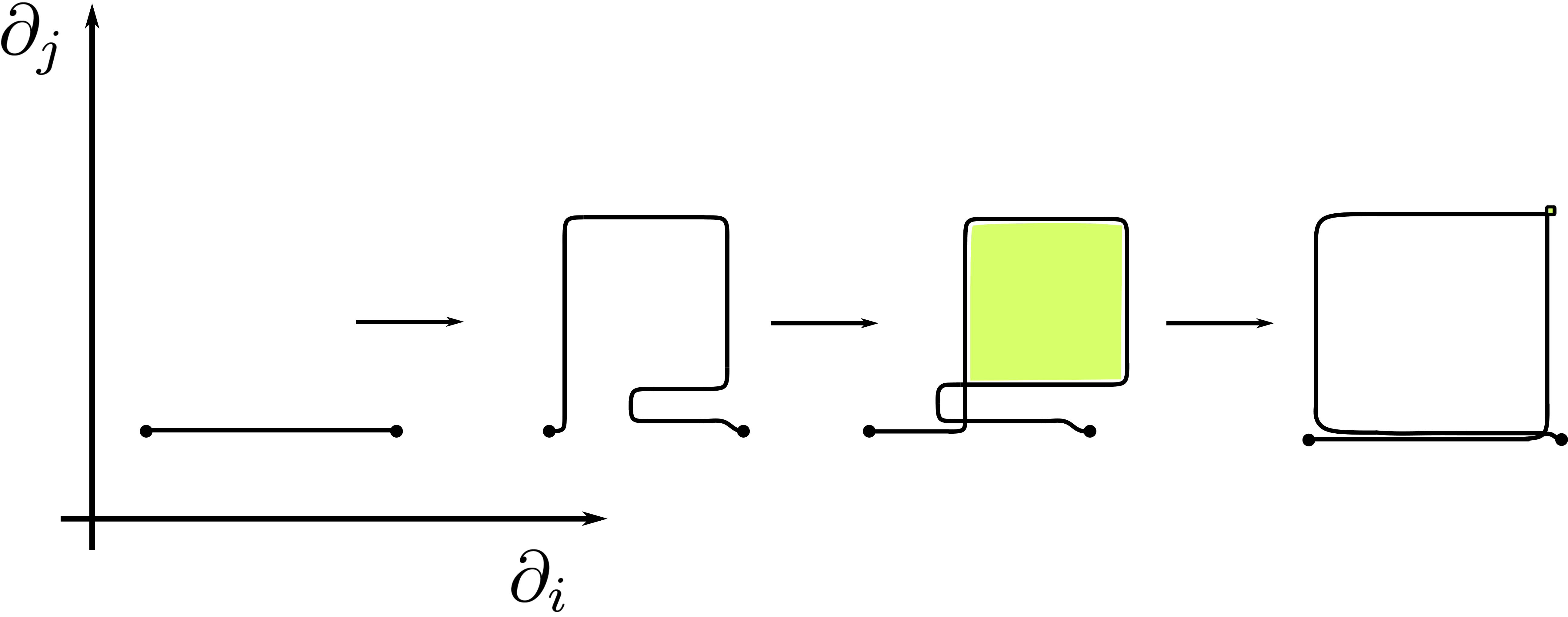}
	\centering
	\caption{Birth homotopy for a length $2$ base tangle model $\mathcal{T}^t_{\left[\phi^i, \phi^j\right]}$.}\label{tanglebase}
\end{figure}

The first three depicted frames in the movie correspond to the isotopy $(\pi\circ\gamma^u)^{u\in[0,\frac{1}{2}]}$, while  the fourth one completes it to $(\pi\circ\gamma^u)^{u\in[0,1]}$. 

Points $i), ii)$ and $iii)$ readily follow from the isotopy taking place in the projected plane $\langle \partial_i, \partial_j\rangle$. Therefore all we have to check is that embeddedness holds when we lift the curve to $\SD$. We will verify that any pair of intersection points taking place in the base (at most two pairs, depending on the value of $u\in[0,1]$) lift to different points upstairs.

Note that this fact can be achieved trivially if the rank of the distribution $\SD$ is greater than $2$, since we can use an additional coordinate $\partial_z$  in order to perform the homotopy while remaining embedded already in the base. This fact is depicted in Figure \ref{EmbeddedCrossings}, where it is shown how to avoid any crossing in the $3-$plane $\langle\partial_i,\partial_j, \partial_z\rangle$ during the isotopy. 

\begin{figure}[h] 
	\includegraphics[scale=0.25]{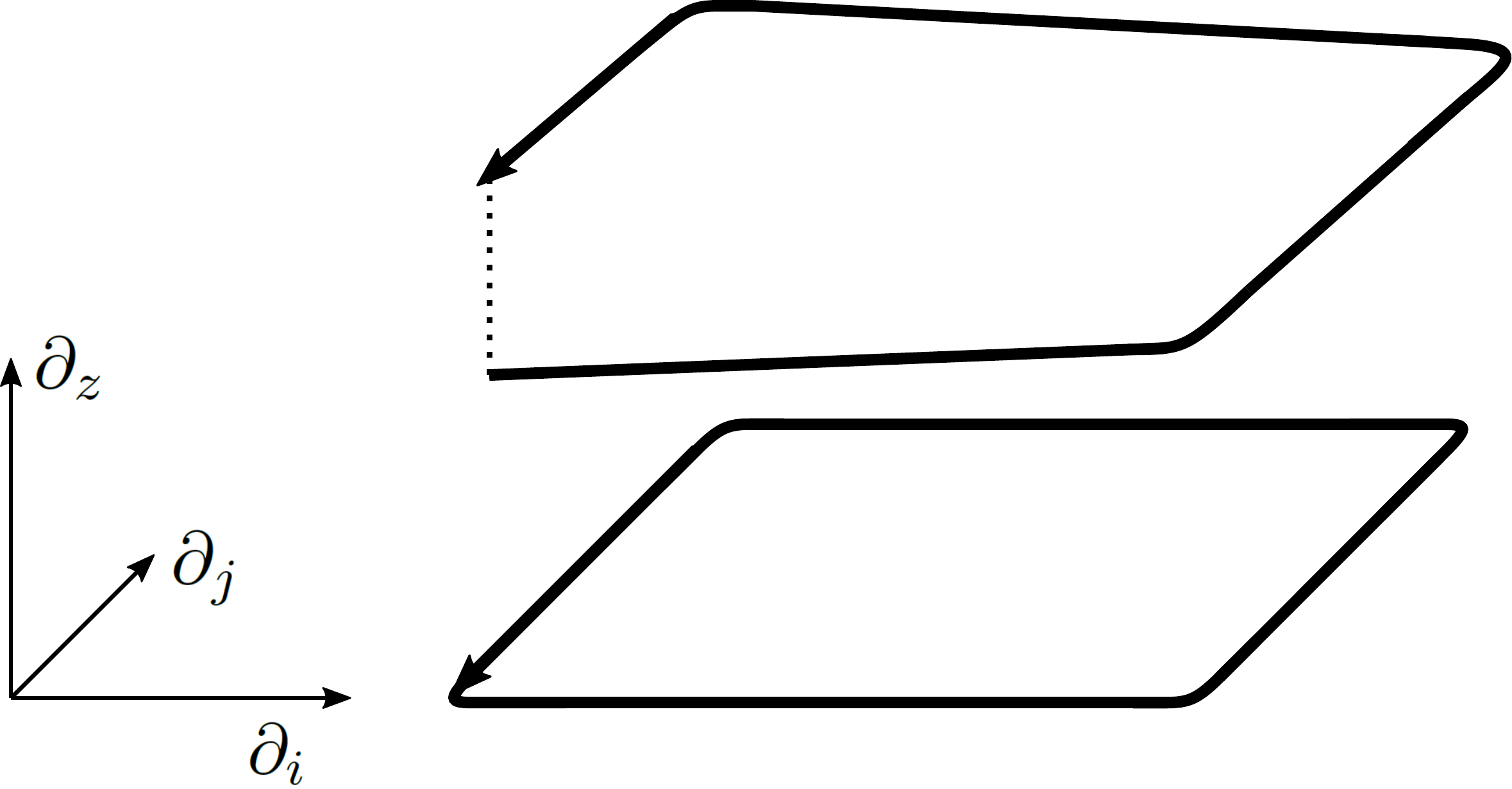}
	\centering
	\caption{Increase of the additional coordinate $\partial_z$ during the homotopy in order to achieve embeddedness in distributions $\SD$ of rank greater than $2$. }\label{EmbeddedCrossings}
\end{figure}

So, let us assume that the distribution is of rank $2$ and, thus, because of $dim(M)>3$ and the bracket-generating condition, we can assume that either $[X_i(p), X_\ell(p)]=X_z(p)$ or $[X_i(p), X_\ell(p)]=X_z(p)$, where $X_z$ is some other element in the adapted frame. Without loss of generality, we assume $[X_i(p), X_\ell(p)]=X_z(p)$.

Let us denote by $\left(\pi\circ\gamma^u(t_1), \pi\circ\gamma^u(t_2)\right)$ the $1-$parametric family of pairs of points corresponding to the upper-right autointersection in the homotopy in Figure \ref{tanglebase}. By Proposition \ref{prop:AlzadoAreas}
the difference in the values of the $\partial_{\ell}$-coordinate between the liftings of the points $\pi\circ\gamma^u(t_1)$ and $\pi\circ\gamma^u(t_2)$ is $A^u\left(1+O(r)\right)$, where $A^u$ is the area enclosed by the curve $\pi\circ\gamma^u|_{[t_1,t_2]}$ in the plane $\langle\partial_i,\partial_j\rangle$. Therefore for a sufficiently small choice of $r>0$,  $A^u\left(1+O(r)\right)$ is a positive number.

Denote by $\left(\pi\circ\gamma^u(t'_1), \pi\circ\gamma^u(t'_2)\right)$ the $1-$parametric family of pair of points corresponding to the other autointersection in Figure \ref{tanglebase}. If the lifting $\gamma|_{[t'_1,t'_2]}$ of the curve $\pi\circ\gamma|_{[t'_1,t'_2]}$ projects onto the plane $\langle \partial_i,\partial_\ell\rangle$ as an opened curve then we are done, since this means that the $\partial_z$-coordinates of the points $\gamma^u(t'_1)$ and $\gamma^u(t'_2)$ are different. 

Otherwise, we get a closed loop that  encloses area $B^u$ in the plane $\langle\partial_i,\partial_{\ell}\rangle$ and that implies that, again by Proposition \ref{prop:AlzadoAreas}, $\pi\circ\gamma(t'_1)$ and $\pi\circ\gamma(t'_2)$ differ an ammount of $B^u\left(1+O(r)\right)$ in the $\partial_\ell$-coordinate. We conclude then that the $\partial_\ell$-coordinates of the liftings of both points are different by the same argument as above.

Properties $i), ii)$ and  $iii)$ are satisfied by construction. 
\end{proof}

\begin{figure}[h] 
	\includegraphics[scale=0.10]{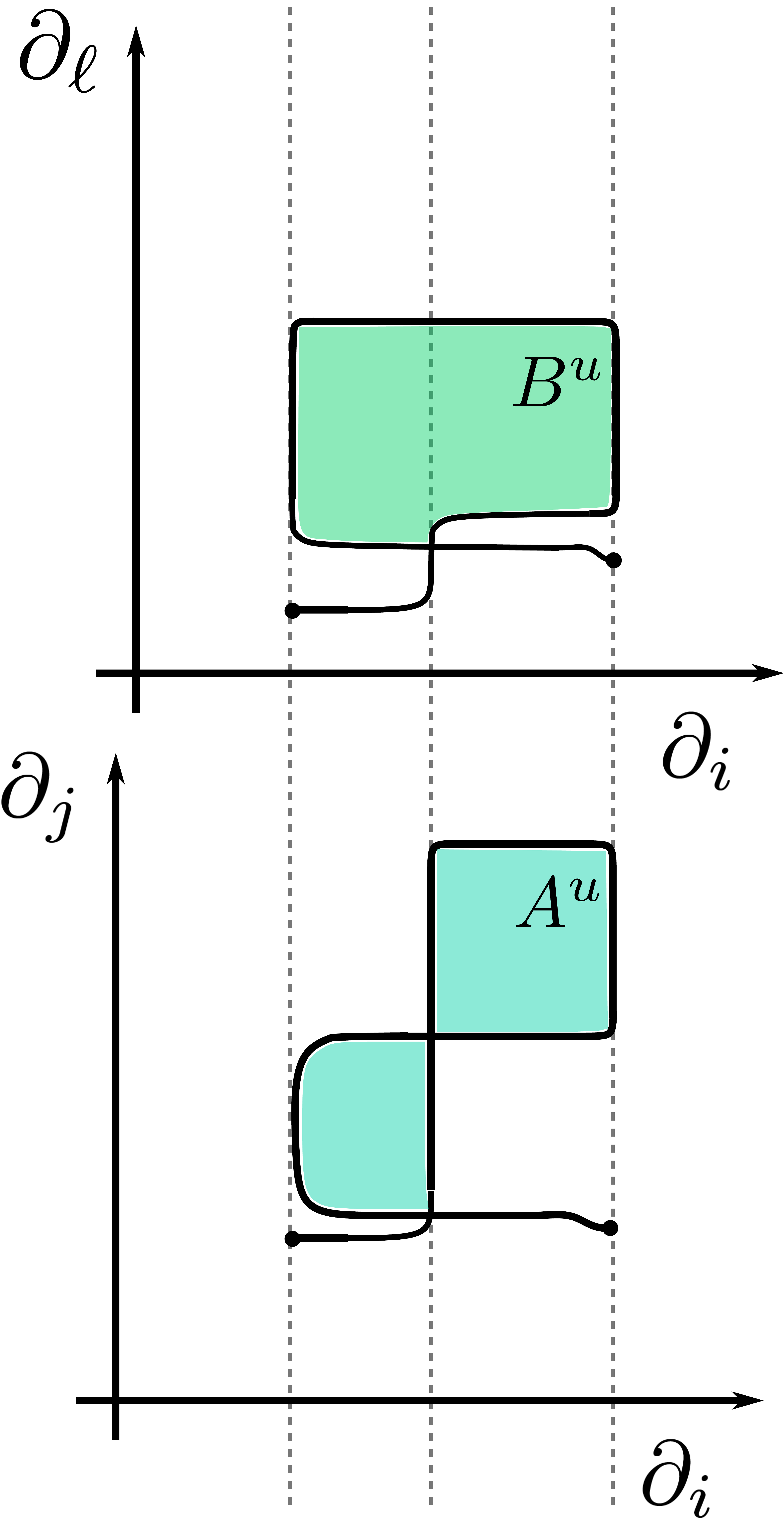}
	\centering
	\caption{When we look at the projection of the curve into the plane $\langle\partial_i,\partial_{\ell}\rangle$ we get a closed loop that encloses area $B^u$. }\label{fig:EmbeddedCrossings2}
\end{figure}

\begin{remark}
Note that we can inductively choose two attaching points $q_1, q_2$ inside a length $2$ base tangle model and a box whose boundary intersects the curve only at $q_1, q_2$ as in Figure \ref{Length2IteratedModel}. This way, we can insert to the given model another length $2$ base tangle model (see Figure \ref{Length2IteratedModel}) which is $2\delta-$close, in the $C^0-$norm, to the given one. We call \textbf{k times nested length $2$ tangle curves} with smoothing parameter $\delta$ to the curve obtained after repeating this process $k$ times.
\end{remark}
\begin{figure}[h] 
	\includegraphics[scale=0.7]{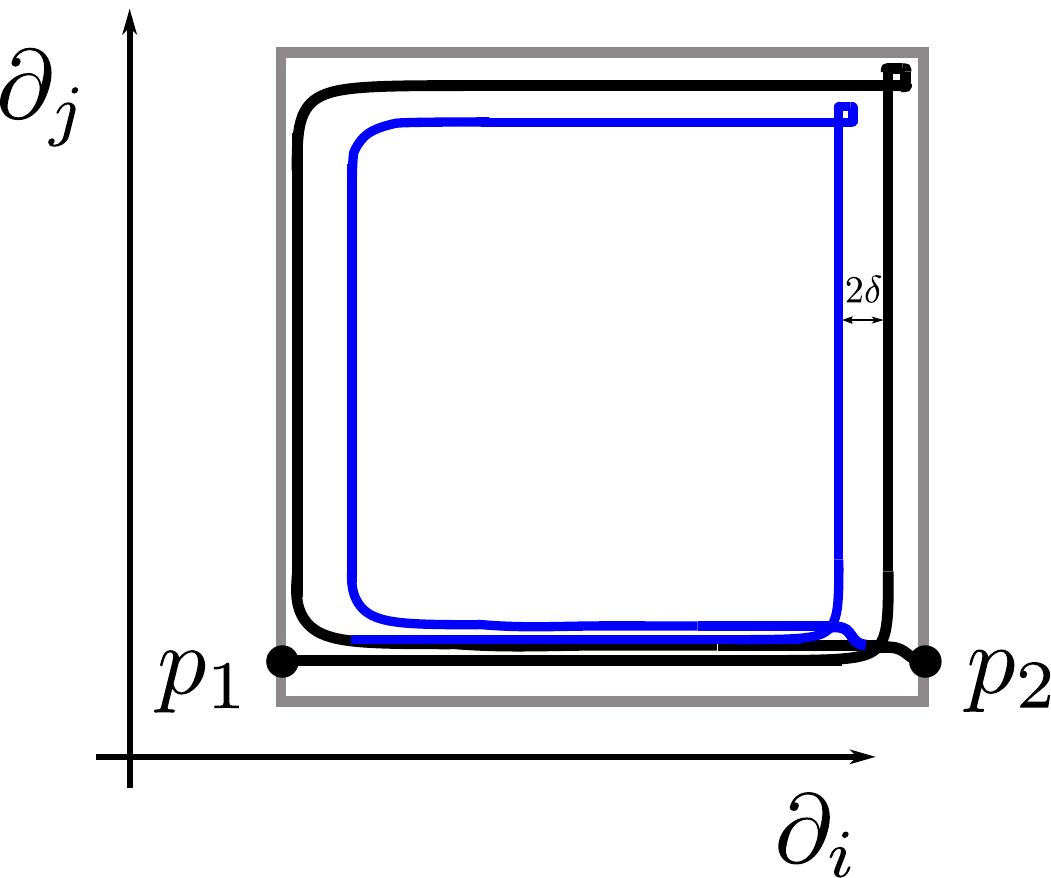}
	\centering
	\caption{Length $2$ base tangle model on the left with two marked attaching points $q_1, q_2$ and a choice of box for inserting another length $2$ base tangle model inside. On the right, a $2$ times nested length $2$ tangle curve.}\label{Length2IteratedModel}
\end{figure}

\subsubsection{Iterated length-$2$ tangle models}

We introduce now a variation on the previously defined model:
\begin{definition}
A \textbf{$k-$times iterated length-$2$ tangle model} associated to the generalised bracket expression $[\phi_i, \phi_j]^{\#k}$ with inputs $\phi_i, \phi_j$ is an attaching model described by Figure \ref{Length2IteratedModel}. Note that the curve inside the box satisfies:
\begin{itemize}
\item it coincides with the straight segment in the direction $p_1,p_2$ for $t\in(t_1, t_1+\epsilon)\cup(t_2-\epsilon, t_2)$,
\item the curve $\beta$ inside the box is a k times nested length $2$ tangle curve with smoothing parameter $\delta$.
\item the size of the box of the attaching model is $\mu+2\delta$,
\end{itemize}
\end{definition}
Note that a length-$2$ base tangle model is just a $1-$time iterated length $2$ tangle model.

\begin{remark}\label{convergenceC0}
A key remark is the following one: note that as $\tau, \delta\to 0$, any $k$-times iterated length $2$ tangle model converges to a pretangle model in the $C^0$-norm. 
\end{remark}

Their birth homotopy is explained in the following proposition:
\begin{proposition}
Let $X_i, X_j$ be two elements in the adapted framing such that $[X_i(p), X_j(p)]=X_{\ell}(p)$. Let $\gamma:[-\delta, \delta]\to \R^q$ be a horizontal curve in a graphical model. 
 There exists a homotopy of embedded horizontal curves $(\gamma^u)^{u\in[0,1]}$ such that: 
\begin{itemize}
\item[i)] $\gamma^0=\gamma$.
\item[ii)] $\pi\circ\gamma^{1}|_{[-\delta/2, \delta/2]}$ is endowed with a $k-$times iterated length-$2$ tangle model  associated to the generalised bracket expression $[\phi_i, \phi_j]^{\#k}$.
\item[iii)] $\pi\circ\gamma^u(t)=\pi\circ\gamma(t)$ for $t\in\Op(\lbrace -\delta, \delta\rbrace)$ and all $u\in[0,1]$.

\end{itemize}
\end{proposition}
\begin{proof}
The claim follows by inductively applying Proposition \ref{prop:introtanglesbase} starting from the outermost curve.
\end{proof}

\subsection{Tangle models of higher length} \label{ssec:tanglesHigherLength}

Consider a bracket expression of the form $A(\phi_1,\cdots,\phi_m)=[\phi_i, B(\phi_r,\cdots, \phi_\ell)]$ with flows $\phi_1,\cdots, \phi_m$ as inputs. A \textbf{length-$n$ base tangle model} is an attaching model associated to $A$. It is described inductively on its length, which is the length of $A$. The inductive step is described in Figure \ref{lengthNmodel}. 

\begin{figure}[h] 
	\includegraphics[scale=0.17]{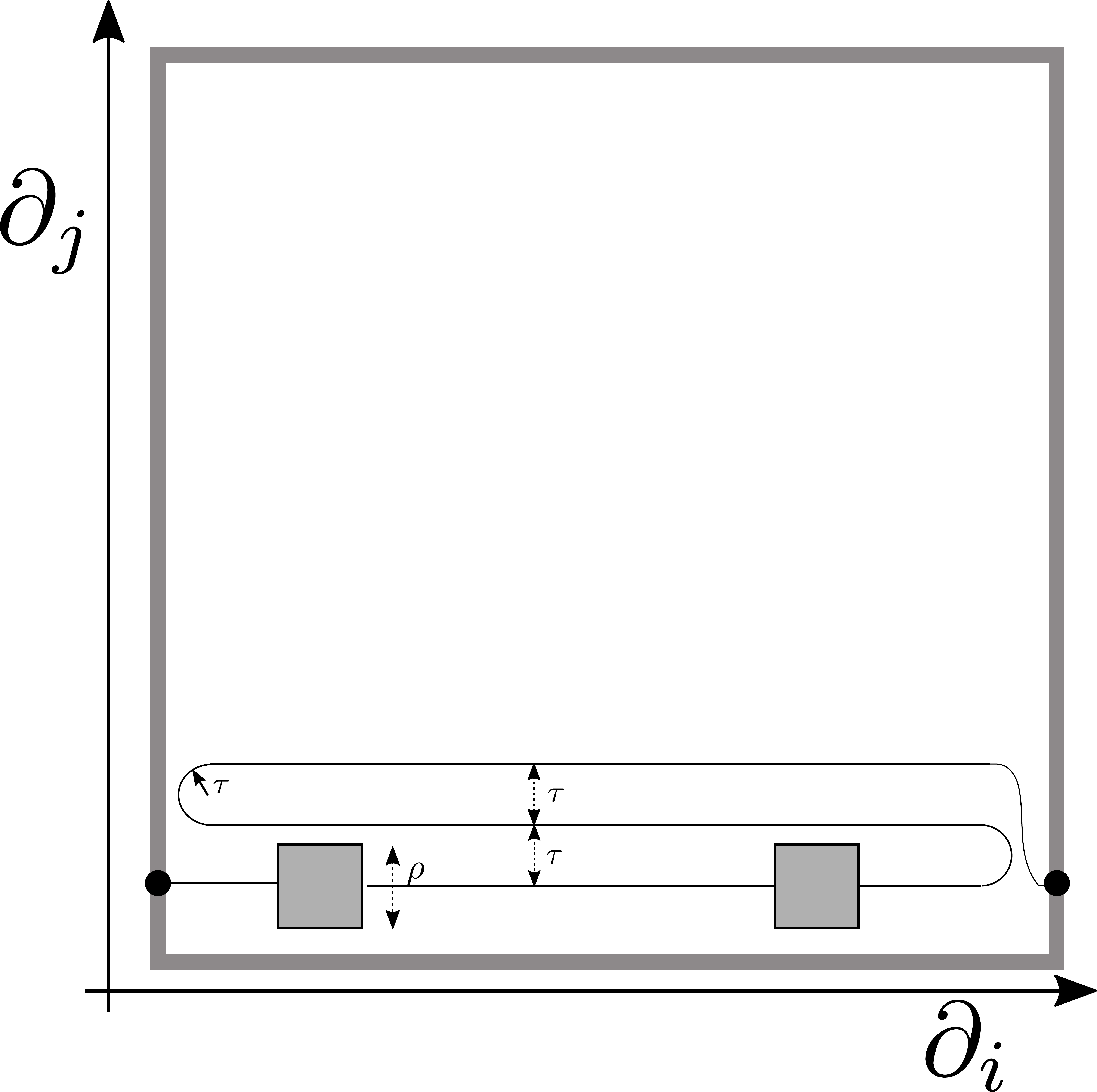}
	\centering
	\caption{Length $N>2$ base tangle model. The grey boxes represent tangle models of size $\rho>0$ of one unit smaller length. The real numbers $\tau, \rho$ are called smoothing parameters associated to the inductive step and are all greater than all the smoothing parameters defined in previous steps. The direction $\partial_j$ is associated to the coordinate flow $\phi_j$, which is the first entry appearing in the generalised bracket expression $A$, different from $\phi_i$. }\label{lengthNmodel}
\end{figure}

The grey boxes in Figure  \ref{lengthNmodel} represent tangle models of size $\eta>0$ associated to the expression $B(\phi_r,\cdots, \phi_\ell)$. The direction $\partial_j$ is associated to the coordinate flow $\phi_j$, which is the first entry appearing in the generalised bracket expression $A$, different from $\phi_i$. All the model, except for the pieces inside the grey boxes, is described in the plane $\partial_i,\partial_j$. The real numbers $\tau, \rho$ are called smoothing parameters associated to the inductive step and are all greater than all the smoothing parameters defined in previous steps.

For length-$N$ tangle models, we can also choose two attaching points $q_1, q_2$ and a box ($2\tau-$close in the $C^0-$norm to the outermost box) in such a model and iterate the construction, in the same fashion as in length $2$, thus constructing another model over the given one.

\begin{figure}[h] 
	\includegraphics[scale=0.15]{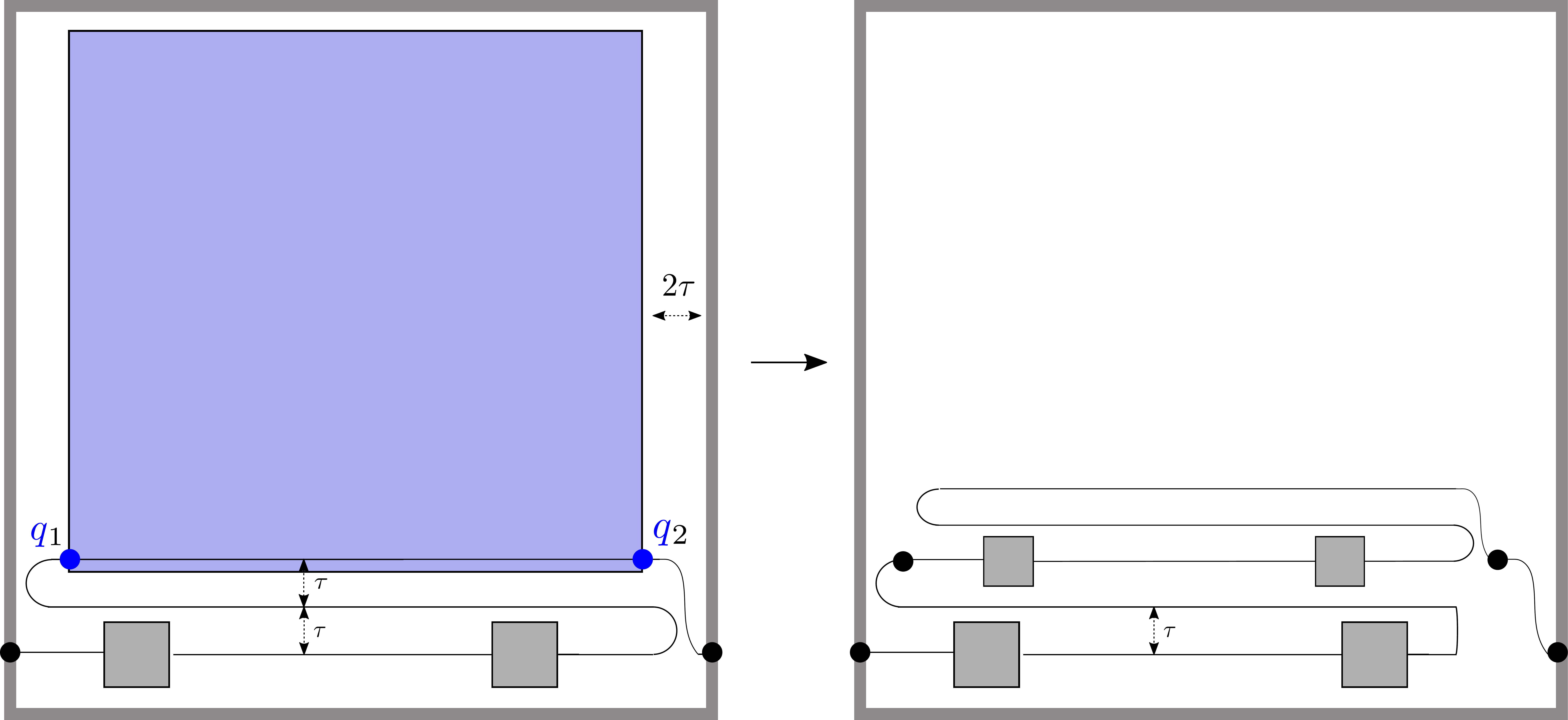}
	\centering
	\caption{Length $N>2$ base tangle model on the left with two marked attaching points $q_1, q_2$ and a choice of box for inserting another length $N$ base tangle model inside. On the right, a $2$ times iterated length $N$ tangle curve.}\label{lengthNIteratedModel}
\end{figure}

\begin{remark}\label{convergenceC0general}
Note that as all the smoothing parameters of any $k$-times iterated length $n$ tangle model tend to zero, the model converges to a pretangle model in the $C^0$-norm. Indeed, it is clear that the result is true for length $2$ models (See Remark \ref{convergenceC0}). On the other hand, assuming that the grey boxes in Figure \ref{lengthNIteratedModel} contained pretangle models instead of tangle models, observe that as $\tau$ and $\rho$ tend to $0$, the whole construction would converge to a pretangle model. Combining both facts the claim follows. We call \textbf{ the pretangle model associated to the tangle model} to such a pretangle model.
\end{remark}

\subsubsection{Birth homotopy for higher length tangle models}

The birth homotopy is given by the following result:
\begin{proposition} \label{prop:insertTangle}
Consider a bracket expression $A(\phi_1,\cdots,\phi_m)$ with inputs the flows $\phi_1,\cdots, \phi_m$. Consider
\[ \pi \circ \gamma:[-\delta, \delta]\to \R^q \]
a family of curves given by a horizontal lift $\gamma$.

Then, there exists a homotopy of embedded horizontal curves $(\gamma^u)_{u\in[0,1]}$ such that: 
\begin{itemize}
\item[i)] $\gamma^0 = \gamma$.
\item[ii)] $\pi\circ\gamma^{1}|_{[-\delta/2, \delta/2]}$ is endowed with a $k-$times iterated length-$n$ tangle model associated to the generalised bracket expression $A$.
\item[iii)] $\pi\circ\gamma^u(t) = \pi\circ\gamma(t)$ for $t\in\Op(\lbrace -\delta, \delta\rbrace)$ and all $u\in[0,1]$.
\end{itemize}
\end{proposition}
\begin{proof}
It is easy to construct the homotopy in the base by defining $\pi\circ\gamma^u$. The length $n$ case iterated model (Figure \ref{lengthNIteratedModel}) reduces to the non-iterated model (Figure \ref{lengthNmodel}) since the birth homotopy for the former can be constructed inductively by using the birth homotopy of the latter.

Note that the birth homotopy for the non-iterated length $n$ model can be constructed inductively. Indeed, assume we already now how to introduce length $n-1$-models of sufficiently small size at any given point of a curve and proceed as follows. We first  homotope the given curve in the box to the curve in Figure \ref{lengthNmodel}) (but omitting the grey boxes). Now we perform the birth homotopies for the $n-1$ tangle models in the grey boxes and we are done.

The base case corresponds to the case of length $2$-tangle models, which we already explained how to do (See Proposition \ref{prop:introtanglesbase}).
\end{proof}

\subsection{Area isotopy} \label{ssec:areaIsotopy}


Associated to a $2-$length tangle realizing the bracket $[X_i, X_j]$, we have a notion of \textbf{increasing} or \textbf{decreasing} its ``\textbf{area}'' just by geometrically increasing or decreasing the area enclosed by the tangle in the $\langle\partial_i, \partial_j\rangle$ plane. In a sense, the increasing/decreasing of such area parametrizes  (\textbf{controls}) the increase of the $\partial_z$ coordinate of the lifted curve. We will extend this notion of area controlling for higher length tangle expressions.

Assume $\partial_z(p) = A(X_1,\cdots,X_n)(p)$ in a graphical model based at $p\in M$ with $\len(A)=\lambda$. The following statement allows us to estimate how a pretangle controls the endpoint: 
\begin{proposition} \label{EndpointBracketCurves}
A pretangle $\gamma^{A(\mu)}$ into the graphical model associated to the generalised bracket-expression $A(\phi^1_s,\cdots,\phi^{n}_s)$ lifts to the distribution as a curve $\tilde{\gamma}:[0,1]\to \R^n$ where the difference between the endpoints $\gamma(1)$ and $\gamma(0)$ is $\mu^\lambda\left(1+O(r)\right)\cdot\partial_z$. 
\end{proposition}
\begin{proof}
By Proposition \ref{prop:bracketk} (Subsection \ref{Appendix}) the following equality holds
\[ A\left(\varphi_t^{X_1},\cdots,\varphi^{X_\lambda}_t\right) = \varepsilon_{t^\lambda}\circ\phi^{A(X_1,\cdots,X_\lambda)}_{t^\lambda}.\]
On one hand we have that $\partial_z(p) = A(X_1,\cdots,X_n)(p)$ and, thus, by Taylor's Remainder Theorem we have that for nearby points $q\in\Op(p)$ the following equality holds 
\[ A(X_1,\cdots,X_n)(q)=\partial_z+O(r)\]
Combining both inequalities:
\[ A\left(\varphi_t^{X_1},\cdots,\varphi^{X_\lambda}_t\right)(q) = \varepsilon_{t^\lambda}\circ\phi^{\partial_z+O(r)}_{t^\lambda}=t^\lambda\cdot (\partial_z+O(r)), \]
where the error associated to $\varepsilon_{t^\lambda}$ has been subsummed by $O(r)$. Now taking $t=\mu$ implies the claim.
\end{proof}

The lifting of a curve into a connection does not depend on its reparametrization. Therefore,
\begin{definition}
By Proposition \ref{EndpointBracketCurves}, we have associated to a pretangle $A(\mu)$ a real number $\mu^\lambda$ which is independent of its reparametrization and that we call its \textbf{total area}.
\end{definition}
Let $\gamma:I\to\R^q$ be a curve in the base of the graphical model equipped with a pretangle model associated to $A$ with attaching points $\gamma(t_1)=q_1$ and $\gamma(t_2)=q_2$

\begin{corollary}
The curve $\gamma$ equipped with the pretangle model lifts to the distribution as a curve $\tilde{\gamma}:[0,1]\to \R^n$ where the difference between the endpoints $\tilde\gamma(1)$ and $\tilde\gamma(0)$ is $\mu^\lambda\left(1+O(r)\right)\cdot\partial_z$. 
\end{corollary}

\begin{definition}
We define the \textbf{total area of a pretangle model} as the total area of the pretangle in the model.
\end{definition}

\subsubsection{Area isotopy for pretangle models}

We now describe a way of increasing/decreasing the total area of a given pretangle model by appropriately manipulating it.

Assume that the generalised bracket expression $A$ is of the form $A(\phi_i,\cdots,\phi_n)=[\phi_i, B(\phi_\ell,\cdots\phi_n)]^{\#k}$ with inputs flows $\phi_i,\cdots, \phi_n$, where $B$ is a bracket expression of smaller length. Let $r>0$ be the total number of times that $\phi_i$ appears as an entry in the expression  $A(\phi_i,\cdots,\phi_n)$.  Consider a pretangle model $\mathcal{PM}_A$ associated to $A$ with box a hypercube $\mathcal{B}\subset\R^q$ and total area $\mu^\lambda$.

Take coordinates in $\R^q$ in such a way that the hypercube $\mathcal{B}$ has its center at the origin. Take a bump function $\psi:\R^q\to[0,1]$ in $\R^k$ with support $\Op(M^{1/r}\cdot\mathcal{B})$, where $M^{1/r}\cdot\mathcal{B}$ denotes the hypercube with side $M^{1/r}$ times the one of $\mathcal{B}$. $M$ denotes the size of the maximal box onto which we can extend $\mathcal{B}$.
\begin{definition} \label{def:areaIsotopy}
We define the Area isotopy $(\Psi^u)_{u\in[0,M]}$  of the pretangle model $\mathcal{PM}_A$ as:
\begin{center}
		$\begin{array}{rccl}
		(\Psi^u)_{u\in[0,M]}: & \R^k & \longrightarrow & \R^k \\
		& x=(x_1,\cdots, x_i,\cdots, x_k) & \longmapsto & (x_1,\cdots, \psi_{\mathcal{B}}(x)\cdot (u)^{1/r}x_i,\cdots, x_k)
		\end{array}$
	\end{center}
\end{definition}
Recall that $\partial_z(p) = A(X_1,\cdots,X_n)(p)=$ in the graphical model based at $p\in M$.

The upcoming proposition explains how the Area isotopy $\Psi^u\circ\gamma$ applied to the curve $\gamma$ equipped with the pretangle model $\mathcal{PM}_A$ behaves with respect to endpoints.

A combination of Proposition \ref{EndpointBracketCurves} together with a rutinary inductive argument based on Proposition \ref{prop:AlzadoAreas} imply the following result:
\begin{proposition}[Endpoint of lifted pretangle models under the Area isotopy] \label{prop:endpointsLiftPretangles} 
$\Psi^u\circ\gamma$ lifts to the distribution as a curve $\tilde{\gamma}^u:[0,1]\to \R^n$ where the difference between the lifts of the attaching points $\tilde\gamma(t_1)$ and $\tilde\gamma(t_2)$ is $(u\cdot\mu^\lambda)\left(1+O(r)\right)\cdot\partial_z$. 
\end{proposition}

\subsubsection{Area isotopy for tangle models}

We have explained so far how the Area isotopy controls the displacement of the lifted endpoints of pretangle models. Nonetheless, we can extend the discussion to tangle models.

Recall that as all the smoothing parameters of any tangle model tend to zero, the model converges to a pretangle model in the $C^0$-norm (see Remark \ref{convergenceC0general}). We call the \textbf{total smoothing parameter} of a given tangle model to the real number $\delta:=\max_i \lbrace \delta_i\rbrace$, where $\lbrace \delta_i\rbrace_{i}$ is the set of all smoothing parameters of a given tangle model. Then, we have that in the limit case where $\delta\to 0$, tangle models converge to pretangle models in the $C^0$-norm.

\begin{definition}
We define the \textbf{total area of a tangle model} as the total area of its associated pretangle model. 
\end{definition}

We can analogously define the Area isotopy for a tangle model:
\begin{definition}
We define the \textbf{Area isotopy for a tangle model} as the Area isotopy of its associated pretangle model.
\end{definition}

\begin{figure}[h] 
	\includegraphics[scale=0.34]{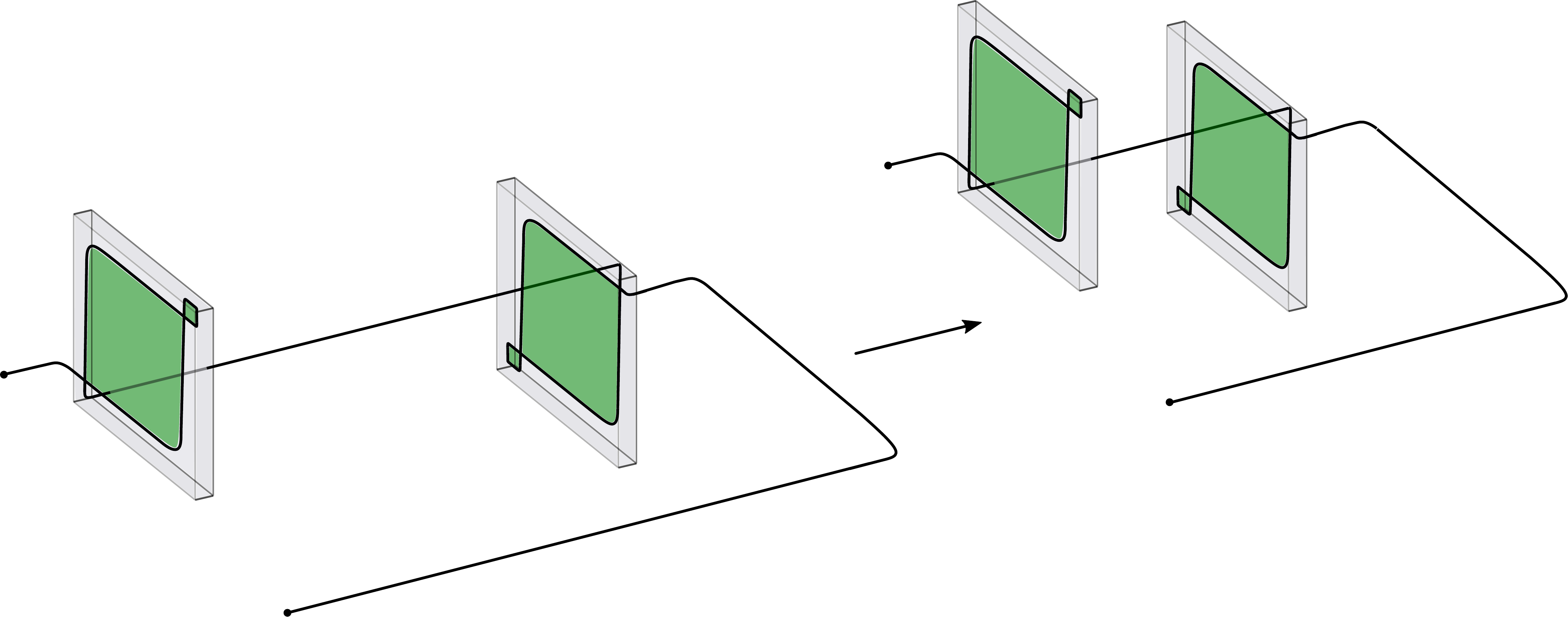}
	\centering
	\caption{Schematic description of a length $3$ tangle model. The transition from the first frame to the second one represents the isotopy $\Psi^u$ acting on the model. The long segments $X_{i}$ get shrunk while the length $2$ subtangle models do not get contracted nor shrunk.} 
\end{figure}

As a consequence of the whole discussion until this point we deduce the following key result:
\begin{proposition}[Endpoint of lifted tangle models under the Area isotopy] \label{endpointsLift} 
Let $\gamma:I\to\R^q$ be a curve equipped with a tangle model associated to $A$ with attaching points $\gamma(t_1)=q_1$ and $\gamma(t_2)=q_2$. Then, $\Psi^u\circ\gamma$ lifts to the distribution as a curve $\tilde{\gamma}^u:[0,1]\to \R^n$ where the difference between the lifts of the attaching points $\tilde\gamma(t_1)$ and $\tilde\gamma(t_2)$ is $(u\cdot \mu^\lambda)\left(1+O(r)\right)\left(1+O(\delta)\right)\cdot\partial_z$. 
\end{proposition}

\begin{remark}
Note that, as the radius $r$ of the graphical model gets close to zero, the quantity $(1+O(r))$ becomes close to $1$. The same phenomenon holds when the total smoothing parameter $\delta$ of the tangle model gets close to zero, $(1+O(\delta))$ becomes close to $1$. Therefore, in the limit, adjusting the endpoints of lifted tangle models is practically equivalent to adjusting the endpoints of the corresponding lifted pretangle models. 
\end{remark}

\subsection{Tangles} \label{ssec:tangles}

Let $X_j$ be an element in the framing of $TV$, with $A$ a bracket-expression generating it. We assume that $A$ is of the form $[\phi_i, B(-)]^{\#k}$. Consider the following data:
\begin{itemize}
\item a size $R>0$,
\item a curve $\gamma: [0,1] \to \R^q$ parallel to $\partial_i$
\item two attaching points $\gamma(t_1)=q_1$ and $\gamma(t_2)=q_2$
\end{itemize}
Then there exists a tangle model $\mathcal{TM}$, associated to the bracket-expression $A$, and endowed with:
\begin{itemize}
\item Total area $\mu^\lambda$ (determined by $R$).
\item Smoothing parameter $\delta>0$.
\item A birth homotopy for $\mathcal{TM}$, parametrised by $\theta \in [0,1]$ and given by Proposition \ref{prop:insertTangle}.
\item An area isotopy (Definition \ref{def:areaIsotopy}) parametrised as $d \mapsto \Psi^{(d/\mu^{\lambda})}$.
\item An upper bound $h := M\mu^\lambda$ for the area isotopy. 
\end{itemize}

We now put together all the ingredients introduced in this section:
\begin{definition} \label{def:tangle}
An \textbf{$X_j$-tangle} is a family of curves
\[ \ST: (0,h] \times [0,1]  \quad\longrightarrow\quad \Imm([0,1];\R^q)\]
given by the previously introuced tangle model $\mathcal{TM}$. It is parametrised by
\begin{itemize}
\item an \textbf{estimated-displacement} $d \in (0,h]$ that governs the area isotopy,
\item the \textbf{birth-parameter} $\theta \in [0,1]$.
\end{itemize}
The number $h$ is called the \textbf{maximal-displacement} of the tangle.
\end{definition}

\subsubsection{Error in the displacement}

The following statement bounds the difference between the estimated-displacement and the actual displacement of the endpoint upon lifting a tangle:
\begin{lemma} \label{lem:errorTangles}
Lift the tangle $\ST$ using Lemma \ref{lem:horizontalisationGraphical}. Then the following estimate holds:
\[ \lift(\ST(d,1))(1) = \lift(\ST(0,1))(1) + (0,\cdots,0,d,0,\cdots,0) + O(r + \delta)d, \]
where $r$ is the radius of the graphical model and $\delta$ is the smoothing parameter of the tangle.
\end{lemma}
The proof is immediate from Proposition \ref{endpointsLift}.

\section{Controllers} \label{sec:controllers}

In the previous Section \ref{sec:tangles} we introduced the notion of tangle. The purpose of a tangle is to displace the endpoint of a horizontal curve in a given direction. In this section we introduce the notion of \emph{controller}. This is a sequence of tangles, located one after the other, in order to be able to control the endpoint of a curve fully.

In Subsection \ref{ssec:controllability} we talk about general (finite-dimensional) families of horizontal curves. The goal is to discuss their endpoint map and make quantitative statements about their controllability. We then particularise to controllers (Subsection \ref{ssec:controllers}), which are specific families of horizontal curves built out of tangles. The process of adding a controller to a horizontal or $\varepsilon$-horizontal curve is explained in Subsection \ref{ssec:insertionControllers}.

\subsection{Controllability} \label{ssec:controllability}

We introduced the notion of regularity in Subsection \ref{ssec:regularity}. This meant that the endpoint map of the horizontal curve under consideration was an epimorphism, which should be understood as a form of infinitesimal controllability (every infinitesimal displacement of the endpoint can be followed by a variation of the curve). In this Subsection we pass from infinitesimal to local.

\subsubsection{Controlling families}

For our purposes we need to work on a parametric setting. Fix $(M,\SD)$, a manifold endowed with a bracket-generating distribution, and a compact fibre bundle $E \rightarrow K$. We write $E_k$ for the fibre over $k \in K$.

Given a family of horizontal curves
\[ \gamma: E \longrightarrow \Maps([0,1];M,\SD) \]
we have evaluation maps $\ev_0, \ev_1:  E \quad\longrightarrow\quad M$ defined by the expression $\ev_a(e) := \gamma(e)(a)$. We require that $\ev_0$ is constant along the fibres of $E$. 

\begin{definition} \label{def:controlling}
The family $\gamma$ is \textbf{controlling} (in a manner fibered over $K$) if $\ev_1|_{E_k}$ is a submersion for all $k \in K$.
\end{definition}
Given a section $f: K \to E$ we can produce a family $\gamma \circ f: K \to \Maps([0,1];M,\SD)$. We say that $\gamma \circ f$ is \textbf{controllable} and that $\gamma$ is a \textbf{controlling extension}. That is, we are interested in $\gamma \circ f$ and we think of the controlling family $\gamma$ as a device that allows us to control its endpoints. 

\subsubsection{Controllability}

It is immediate from the parametric nature of the implicit function theorem that this implies local controllability:
\begin{lemma} \label{lem:implicitFunctionTheorem}
Given a controlling family $\gamma$ and a section $f: K \to E$, there are constants $C, \eta_0>0$ such that:
\begin{itemize}
\item for any $0 < \eta < \eta_0$,
\item and any smooth choice of endpoint $q_k \in \D_\eta(\gamma \circ f(k)(1))$,
\end{itemize}
there exists a section $g: K \to E$ such that:
\begin{itemize}
\item $g$ and $f$ are homotopic by a homotopy of $C^0$-size at most $C\eta$.
\item $\gamma \circ g(k)(1) = q_k$.
\end{itemize}
Furthermore, if $q_k = \gamma \circ f(k)(1)$, it can be assumed that $g(k) = f(k)$.
\end{lemma}

The following variation will be useful for us:
\begin{lemma} \label{lem:inverseFunctionTheorem}
Let $\gamma$ be a controlling family such that its evaluation map $\ev_1|_{E_k}$ is an equi-dimensional embedding for all $k \in K$. Let $q: K \rightarrow M$ with $q \in \ev_1(E_k)$. Then, there is a unique section $f: K \rightarrow E$ such that $\gamma(f(k))(1) = q(k)$.
\end{lemma}
It follows from the inverse function theorem.

\subsubsection{Existence of controlling families}

The following statement follows from standard control theoretical arguments:
\begin{lemma}\label{lem:controlling}
A family of regular horizontal curves $\gamma: K \longrightarrow \Maps([0,1];M,\SD)$ admits a controlling family.
\end{lemma}
In the non-parametric case, this was proven in \cite[Proposition 4 and Corollary 5]{Hsu}. The proof amounts to choosing infinitesimal variations (given by the regularity condition) and integrating these to a controlling family. The parametric case was then explained in \cite[Section 8]{PS} and requires us to patch these variations parametrically in $k \in K$. The idea is that the infinitesimal variations can be ``localised'' in $[0,1]$ by an appropriate use of bump functions. This can be exploited to show that variations do not interfere with each other when patching.

\subsection{Defining controllers} \label{ssec:controllers}

A controller will depend on the following input data: a graphical model $(V,\SD)$, a maximal-displacement $h>0$, a radius $R>0$, a size-at-rest $S$, and a smoothing-parameter $\delta$. Recall that $\pi: V \rightarrow \R^q$ is the projection to the base in the model and we have a framing $\{X_1,\cdots,X_n\}$ of $TV$ such that $\{X_1,\cdots,X_q\}$ is a framing of $\SD$ obtained as a lift of the coordinate framing of $\R^q$.

\subsubsection{Setup}

We consider the cube $C = [-R,R]^q \subset \R^q$ and we write $\gamma: [0,1] \rightarrow C$ for the curve $\gamma(t) = (2Rt-R,0,\cdots,0)$ parametrising its first coordinate axis. We call it the \emph{axis} of the controller to be built. For each vertical direction $i=q+1,\cdots,n$ in the model, we define a pair of points 
\[ x_{i,+} = \gamma\left(\dfrac{i-q}{n}\right), \qquad x_{i,-} = \gamma\left(\dfrac{i-q}{n} = \dfrac{1}{2n}\right), \]
that are meant to serve as insertion points for tangles. We fix a box $C_{i,\pm}$ centered at $x_{i,\pm}$ and of side $1/4n$. In this manner all the boxes are disjoint.

For each $i = q+1,\cdots,n$, we insert (Proposition \ref{prop:insertTangle}) a $(\pm X_i)$-tangle at $x_{i,\pm}$; we denote it by $\ST^{i,\pm}$. The radius of these tangles should be smaller than the side $1/4n$, so that they are contained in the boxes $C_{i,\pm}$. We require that the maximal-displacement of the tangles is $2h$. We write $\ST^{j,\pm}_{d,\theta}$ whenever we need to include their estimated-displacement $d$ and birth parameter $\theta$ in the discussion. We use $\delta$ as their smoothing-parameter.

\subsubsection{The definition}

The estimated-displacement of $\ST^{i,+}$ pushes positively along $X_i$ by an amount in the range $(0,2h]$. Similarly, the estimated-displacement of $\ST^{i,-}$ pushes along $-X_i$. We want to combine these two displacements in order to produce motion in the $X_i$-direction in the range $[-h,h]$. To do so, we recall the size-at-rest $S$ parameter that we fixed earlier, and we consider a bump function $\chi_S: \R \to [0,1]$ satisfying 
\[ \chi_S{(-\infty,-S]} = S, \qquad \chi_S|_{[0,\infty)}(a) = a+S, \qquad -S < \chi_S'|_{(0,1)} < S. \]

\begin{definition} \label{def:controller}
A \textbf{controller} is a family of curves 
\[ \SC: [-h,h]^{n-q} \times [0,1] \quad\longrightarrow\quad \Imm([0,1]; C \subset \R^q) \]
parametrised by:
\begin{itemize}
\item a \textbf{estimated-displacement} $d = (d_{q+1},\cdots,d_n) \in [-h,h]^{n-q}$,
\item and a \textbf{birth parameter} $\theta \in [0,1]$.
\end{itemize}
Each curve $\SC(d,\theta)$ is obtained from $\gamma$ by inserting 
\[ \ST^{i,+}_{\chi_S(d_i),\theta} 			\quad\text{ at }\quad x_{i,+}, \qquad 
   \ST^{i,+}_{\chi_S(-d_i),\theta} \quad\text{ at }\quad x_{i,-}. \]
In particular, $\SC(d,\theta)$ agrees with $\gamma$ close to the boundary of $C$. The controller $\SC$ depends on the following parameters:
\begin{itemize}
\item the radius $R>0$ of the box $C$ that contains it,
\item the maximal-displacement $h>0$ bounding the estimated-displacement,
\item the size-at-rest $S$ that defines the interpolating function $\chi_S$,
\item the smoothing-parameter $\delta>0$ of its tangles.
\end{itemize}
\end{definition}

\subsection{Insertion of controllers} \label{ssec:insertionControllers}

Let $(V,\SD)$ be a graphical model and $K$ a compact manifold that serves as parameter space.
\begin{proposition} \label{prop:insertionControllers}
Let the following data be given:
\begin{itemize}
\item a family $\gamma: K \rightarrow \Emb([0,1];D_r \subset V,\SD)$,
\item a maximal-displacement $0 < h < 2r$,
\item a $k$-disc $A \subset K$,
\item a time $t_0 \in (0,1)$,
\item a constant $\eta > 0$ and a sufficiently small $\tau > 0$,
\item a sufficiently small radius $R>0$, size-at-rest $S>0$, and smoothing-parameter $\delta>0$.
\end{itemize}
Then, assuming that $r$ is sufficiently small, there are:
\begin{itemize}
\item a function $\theta: K \rightarrow [0,1]$ that is identically one on $\nu_\tau(A)$ and zero in the complement of $\nu_{2\tau}(A)$,
\item a family $\widetilde\gamma: K \times [-h,h]^{n-q} \times [0,1] \rightarrow \Emb([0,1];V,\SD)$,
\end{itemize}
such that the following statements hold:
\begin{itemize}
\item $\widetilde\gamma(k,d,0) = \gamma(k)$.
\item $\widetilde\gamma(k,d,s)(t) = \gamma(k)(t)$ if $t \in \Op(0)$.
\item $\pi \circ \widetilde\gamma(k,d,s)(t) = \pi \circ \gamma(k)(t)$ outside of $\nu_{2\tau}(A) \times \nu_{2\tau}(t_0)$.
\item The length of $\widetilde\gamma(k,d,s)$ in the region $\nu_{2\tau}(t_0) \setminus \nu_{\tau}(t_0)$ is bounded above by $\eta$.
\item For all $k \in \nu_\tau(A)$, all $t \in \nu_\tau(t_0)$, and all $s \in [1/2,1]$, it holds that
\[ \pi \circ \widetilde\gamma(k,d,s)(t) = \SC(d, (2s-1)\theta(k))\left(\dfrac{t-t_0-\tau}{2\tau}\right) + \pi \circ \gamma(k)(t_0). \]
Here the left-hand-side is a reparametrised and translated copy of the controller $\SC$ with radius $R$, size-at-rest $S$, maximal-displacement $h$, and smoothing-parameter $\delta$. 
\end{itemize}
\end{proposition}
The family $\widetilde\gamma$ is said to be obtained from $\gamma$ by inserting the controller $\SC$ along $A$. The last item asserts that $\SC$ has been inserted. The other items provide quantitative control for the insertion.

\begin{proof}
We first homotope the family $\gamma$ in the vicinity of $A \times \{t_0\}$ in order to align its projection with the axis of the controller $\SC$. This follows from an application of Lemma \ref{lem:directionAdjust}. We denote the resulting homotopy of $\varepsilon$-horizontal curves by $\beta(k,s)|_{s \in [0,1/2]}$. When we apply Lemma \ref{lem:directionAdjust}, we use a constant $\tau>0$ such that: $\beta(k,s)$ agrees with $\gamma(k,s)$ outside of $\nu_{2\tau}(A) \times \nu_{2\tau}(t_0)$. Furthermore, over $\nu_{\tau}(A) \times \nu_{\tau}(t_0)$, the projection $\pi \circ \beta(k,1/2)$ agrees with a translation of the axis of $\SC$. The desired bound $\eta$ tells us how small $\tau$ must be chosen. In particular, it should be small enough so that $\len(\beta(k,s))$ is bounded above by $\len(\gamma(k)) + \eta/2$.

We then apply Lemma \ref{lem:horizontalisationGraphical} to $\beta$, relative to $t=0$. This yields a family of horizontal curves $\widetilde\gamma(k,d,s)|_{s \in [0,1/2]}$ such that $\pi \circ \widetilde\gamma(k,d,s) = \pi \circ \beta(k,s)$. Thanks to the bound $\eta$ and the fact that $\gamma$ is horizontal, we can invoke Lemma \ref{lem:horizontalisationStability} to assert that $\widetilde\gamma(k,d,s)$ is indeed defined for all $t \in [0,1]$. Due to the uniqueness of lifts, $\widetilde\gamma(k,d,0) = \gamma(k)$. We denote $\gamma'(k) = \widetilde\gamma(k,d,1/2)$.

By construction, over $\nu_{\tau}(A) \times \nu_{\tau}(t_0)$, $\pi \circ \gamma'$ agrees with a translation of the axis of $\SC$. Our choice of $\tau$ determines the length of the curves $\pi \circ \gamma'$ in the region $\nu_{\tau}(A) \times \nu_{\tau}(t_0)$. This is the available length for placing the axis of the controller. As such, it provides an upper bound for our choice of $R$. We then define 
\[ \pi \circ \widetilde\gamma|_{k \in \nu_\tau(A),\quad t \in \nu_\tau(t_0),\quad s \in [1/2,1]} \]
using the formula appearing in the last item of the statement. For other values of $k$ and $t$, we set $\pi \circ \widetilde\gamma(k,d,s) = \pi \circ \gamma'(k)$.

The family $\widetilde\gamma(k,d,s)|_{s \in [1/2,1]}$ itself is given from its projection by lifting horizontally. This is done with Lemmas \ref{lem:horizontalisationGraphical} and \ref{lem:horizontalisationStability}. Here is where the smallness of $r$ enters. It must be sufficiently small to control the error in the estimated-displacement of $\SC$ (which amounts to the error in the estimated-displacement of its tangles; Lemma \ref{lem:errorTangles}). How small $r$ must be depends only on the graphical model $(V,\SD)$. It follows that $r$ being small enough implies that a lift exists for all times and, due to the properties of tangles, it yields embedded curves.
\end{proof}

\subsubsection{Controllability}

We now address how the insertion of a controller allows us to control the endpoint of the corresponding horizontal curves.
\begin{lemma} \label{lem:errorControllers}
Consider the setup and conclusions of Proposition \ref{prop:insertionControllers}. Consider the endpoint map
\[ \epoint: K \times [-h,h]^{n-q} \quad\longrightarrow\quad \R^{n-q} \]
defined by $\epoint(k,d) := \pi_\vert \circ \widetilde\gamma(k,d,1)(1)$. Then the following estimate holds:
\[ \epoint(k,d) = \gamma(k) + (0,d)(1+O(r)+O(\delta)) \]
for all $k \in A$.
\end{lemma}
\begin{proof}
This is immediate from the analogous statement about tangles, namely Lemma \ref{lem:errorTangles}.
\end{proof}
As in Lemma \ref{lem:adaptedCharts} we can be more precise and say that there is a constant $C$, depending only on the graphical model, such that the error is bounded above by $C.(r+\delta)$. It follows that imposing $\delta, r << 1/C$ implies that $\epoint(k,-)$ is an equidimensional embedding whose image contains a ball of radius $h/2$, centered at $\gamma(k)$.

\subsubsection{Insertion in the $\varepsilon$-horizontal case}

For our purposes, we will need the following variation of Proposition \ref{prop:insertionControllers}:
\begin{lemma} \label{lem:insertionControllers}
Let $r$, $h$, $A$, $t_0$, $\eta$, $R$, $S$, and $\delta$ be as in Proposition \ref{prop:insertionControllers}. Given $\gamma: K \rightarrow \Emb^\varepsilon([0,1];D_r,\SD)$, there is a family $\widetilde\gamma: K \times [-h,h]^{n-q} \times [0,1] \rightarrow \Emb([0,1];V)$ such that:
\begin{itemize}
\item $\widetilde\gamma(k,d,0) = \gamma(k)$.
\item $\widetilde\gamma(k,d,s) = \gamma(k)$ outside of $\nu_{2\tau}(A) \times \nu_{2\tau}(t_0)$.
\item The length of $\widetilde\gamma(k,d,s)$ in the region $\nu_{2\tau}(t_0) \setminus \nu_{\tau}(t_0)$ is bounded above by $\eta$.
\item For all $k \in \nu_\tau(A)$, all $t \in \nu_\tau(t_0)$, and all $s \in [1/2,1]$, it holds that
\[ \pi \circ \widetilde\gamma(k,d,s)(t) = \SC(d, (2s-1)\theta(k))\left(\dfrac{t-t_0-\tau}{2\tau}\right) + \pi \circ \gamma(k)(t_0). \] 
\end{itemize}
\end{lemma}
\begin{proof}
We use Lemmas \ref{lem:horizontalisationGraphical} and \ref{lem:horizontalisationStability} to horizontalise $\gamma$, yielding some family $\gamma'$. We apply to it Proposition \ref{prop:insertionControllers}, yielding a family $\widetilde\gamma'$ that is also horizontal. We then adjust its vertical component to yield the claimed $\widetilde\gamma$. Due to the displacement introduced by the controller, it may be the case that $\widetilde\gamma$ is not $\varepsilon$-horizontal. However, it is still graphical over $\R^q$.
\end{proof}

\section{h-Principles for horizontal curves} \label{sec:hPrincipleHorizontal}

In this section we prove our main Theorem \ref{thm:embeddings}, the classification of regular horizontal embeddings. This uses all the tools that we have presented in previous sections. The corresponding statement for immersions, Theorem \ref{thm:immersions}, will follow from simplified versions of the same arguments. In Subsection \ref{ssec:hPrincipleHorizontalImmersions} we explain how this is done.

\subsection{Relative version of Theorem \ref{thm:embeddings}} \label{ssec:relativeHorizontalEmbeddings}

We explained in Section \ref{sec:epsilon} that our $h$-principle arguments are relative in nature. This allows us to state an analogue of Theorem \ref{thm:embeddings} that deals with embedded regular horizontal paths and is relative in parameter and domain.
\begin{proposition} \label{prop:relativeHorizontalEmbeddings}
Let $K$ be a compact manifold. Let $I = [0,1]$. Let $(M,\SD)$ be a manifold of dimension $\dim(M)>3$, endowed with a bracket--generating distribution. Suppose we are given a map $\gamma: K \to \Emb^\varepsilon(I;M,\SD)$ satisfying:
\begin{itemize}
\item $\gamma(k) \in \Emb^\regu(I;M,\SD)$ for $k \in \Op(\partial K)$.
\item $\gamma(k)(t)$ is horizontal if $t \in \Op(\partial I)$.
\end{itemize}

Then, there exists a homotopy $\widetilde\gamma: K \times [0,1] \to \Emb^\varepsilon(I;M,\SD)$ satisfying:
\begin{itemize}
\item $\widetilde\gamma(k,0) = \gamma(k)$.
\item $\widetilde\gamma(k,1)$ takes values in $\Emb^\regu(I;M,\SD)$.
\item this homotopy is relative to $k \in \Op(\partial K)$ and to $t \in \Op(\partial I)$.
\item $\widetilde\gamma(k,s)$ is $C^0$-close to $\gamma(k)$ for all $s \in [0,1]$.
\end{itemize}
\end{proposition}
Do note that the analogous statement where we consider formal horizontal embeddings instead of $\varepsilon$-horizontal ones follows from this one, thanks to the results in Subsection \ref{ssec:epsilonHorizontalCurves}.

Now, as usual, the absolute statement follows from the relative one:
\begin{proof}[Proof of Theorem \ref{thm:embeddings} from Proposition \ref{prop:relativeHorizontalEmbeddings}]
The statement follows, according to Subsection \ref{ssec:epsilonHorizontalCurves}, from the vanishing of the relative homotopy groups of the pair 
\[ (\Emb^\varepsilon(M,\SD),\, \Emb^\regu(M,\SD)). \]
Consider a family $\gamma: \D^a \to \Emb^\varepsilon(M,\SD)$ that takes values in $\Emb^\regu(M,\SD)$ along $\NS^{a-1}$. This represents a class in the $a$th relative homotopy group. We must deform this family to lie entirely in $\Emb^\regu(M,\SD)$.

We may assume, by suitable reparametrisation in the parameter, that in a collar of $\NS^{a-1}$, the family $\gamma$ is radially constant. This provides for us an open along the boundary of $\D^a$ in which all curves are regular horizontal. We then consider the product space $\D^a \times \NS^1$ and we make the curves $\gamma$ horizontal in a neighbourhood of the slice $\D^a \times \{1\}$, using Lemma \ref{lem:horizontalisationSkeleton}. This is a $C^0$-small process.

Regarding $\D^a \times \NS^1$ as $\D^a \times I$ yields a family of $\varepsilon$-horizontal paths, horizontal at the endpoints, with all curves regular close to $\partial\D^a$. We apply Proposition \ref{prop:relativeHorizontalEmbeddings} to it, relatively to $\Op(\NS^{a-1})$ in the parameter, and relatively to $\Op(\partial I)$ in the domain. The claim follows.
\end{proof}
The remainder of this section is mostly dedicated to the proof of Proposition \ref{prop:relativeHorizontalEmbeddings}.

\begin{figure}[h] 
	\includegraphics[scale=0.8]{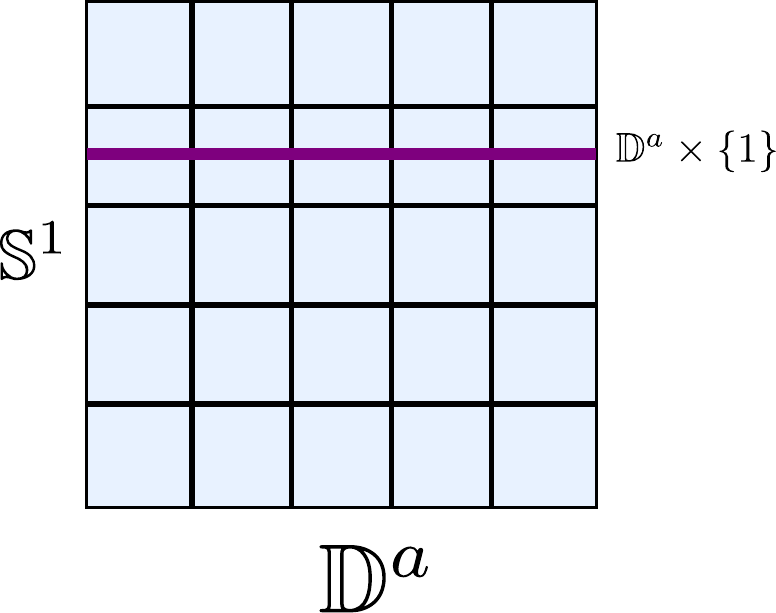}
	\centering
	\caption{Choosing the time slice $\D^a \times \{1\}$ identifies the product $\D^a \times\NS^1$ with $\D^a \times [0,1]$, allowing us to apply Proposition \ref{prop:relativeHorizontalEmbeddings}.}\label{fig:timeSlice}
\end{figure}

\subsection{Setup for the proof of Proposition \ref{prop:relativeHorizontalEmbeddings}} \label{ssec:Step1Proof}

Given the family of curves $\gamma$, we will produce a series of homotopies in order to obtain $\widetilde\gamma(-,1)$. These elementary homotopies are relative to the boundary and through $\varepsilon$-horizontal curves. Their concatenation will be the homotopy $\widetilde\gamma$.

We denote $\dim(M) = n$, $\rank(\SD) = q$, and $\dim(K) = K$. We write $\pi_K$ and $\pi_I$ for the projections of $K \times I$ to its factors. We fix a metric on $K$, endow $I$ with the euclidean metric, and endow the product $K \times I$ with the product metric. We will write $d(-,-)$ to denote distance between subsets; the metric used should be clear from context.

\subsubsection{A fine cover of $K \times I$}

Our first goal is to subdivide $K \times I$ in a manner that is nicely adapted to $(M,\SD)$ and $\gamma$. The aim with this is to reduce our subsequent arguments to constructions happening in very small balls in which errors are controlled. No homotopy of $\gamma$ is produced in this subsection, we are just doing some preliminary work.

Introduce a size parameter $N \in \N$, to be fixed as we go along in the proof. We divide $I$ into intervals $I_j$ of length $1/N$. Furthermore, we fix a finite cover of $K$ by charts parametrised by the unit cube. Such cubes can themselves be divided into cubes of side $3/N$, spaced along the coordinate axes as $1/N$. We write 
\[ \{ \phi_i: [0,3/N]^k \,\longrightarrow\, U_i \subset K\} \]
for the resulting collection of cubical charts.

\subsubsection{Bounding the number of intersections between charts}

By construction, there is a constant $C_1$ such that any intersection $U_{i_1} \cap \cdots \cap U_{i_{C_1}}$, involving distinct charts, is empty. This follows from the properties of cubical subdivision; a detailed argument can be found in \cite[p. 25]{MAP}.

\subsubsection{Covering the image by graphical models}

Given the origin $k \in K$ of the chart $U_i$ and the initial time $t_j := j/N \in I_j$, we fix an adapted chart $(V_{i,j},\Psi_{i,j})$ centered at $\gamma(k)(t_j)$. These adapted charts are given by Lemma \ref{lem:adaptedCharts}, meaning that they all have the same radius $r_0 > 0$ and the difference between their framing and the coordinate axes is controlled by some constant $C_2>0$.

\subsubsection{Discussion about parameters}

The constants $C_1$, $C_2$, and $r_0$ are given to us and depend on $(M,\SD)$ and the family $\gamma$. For convenience, we introduce new parameters $0 < r_1 < r_0$ and $0 < l$, to be fixed later in the proof. We impose $1/N << r_1, l$ in order to ensure that:
\begin{equation} \label{eq:parameterR1}
\gamma(k)(t) \in \Psi_{i,j}(\D_{r_1}) \text{ for all } (k,t) \in U_i \times I_j
\end{equation}
\begin{equation} \label{eq:parameterL}
\Psi_{i,j}^{-1} \circ \gamma(k)|_{I_j}  \text{ has euclidean length bounded above by $l$}
\end{equation}
I.e. the curves in each cube $U_i \times I_j$ are very short and are located very close to the origin of the corresponding graphical model.

\subsection{Introducing controllers} \label{ssec:Step2Proof}

Our next goal is to add controllers to $\gamma$ over each cube $U_i \times I_j$. We continue using the notation introduced in the previous subsection. We write $\nu_r(A)$ for the $r$-neighbourhood of a subset $A$.

\subsubsection{Boundary neighbourhoods}

We fix a small enough constant $\tau > 0$ so that all curves $\gamma(k)$ with $k \in \nu_\tau(\partial K)$ are horizontal and regular. We then consider a constant $\tau/2 < \tau' < \tau$ and a subset $\SU$ of the cover $\{U_j \times I_i\}$ so that:
\begin{itemize}
\item the elements of $\SU$, together with $\nu_{\tau'}(\partial (K \times I))$, cover $K \times I$,
\item all elements in $\SU$ are disjoint from $\nu_{\tau/2}(\partial (K \times I))$.
\end{itemize}

\begin{figure}[h] 
	\includegraphics[scale=0.8]{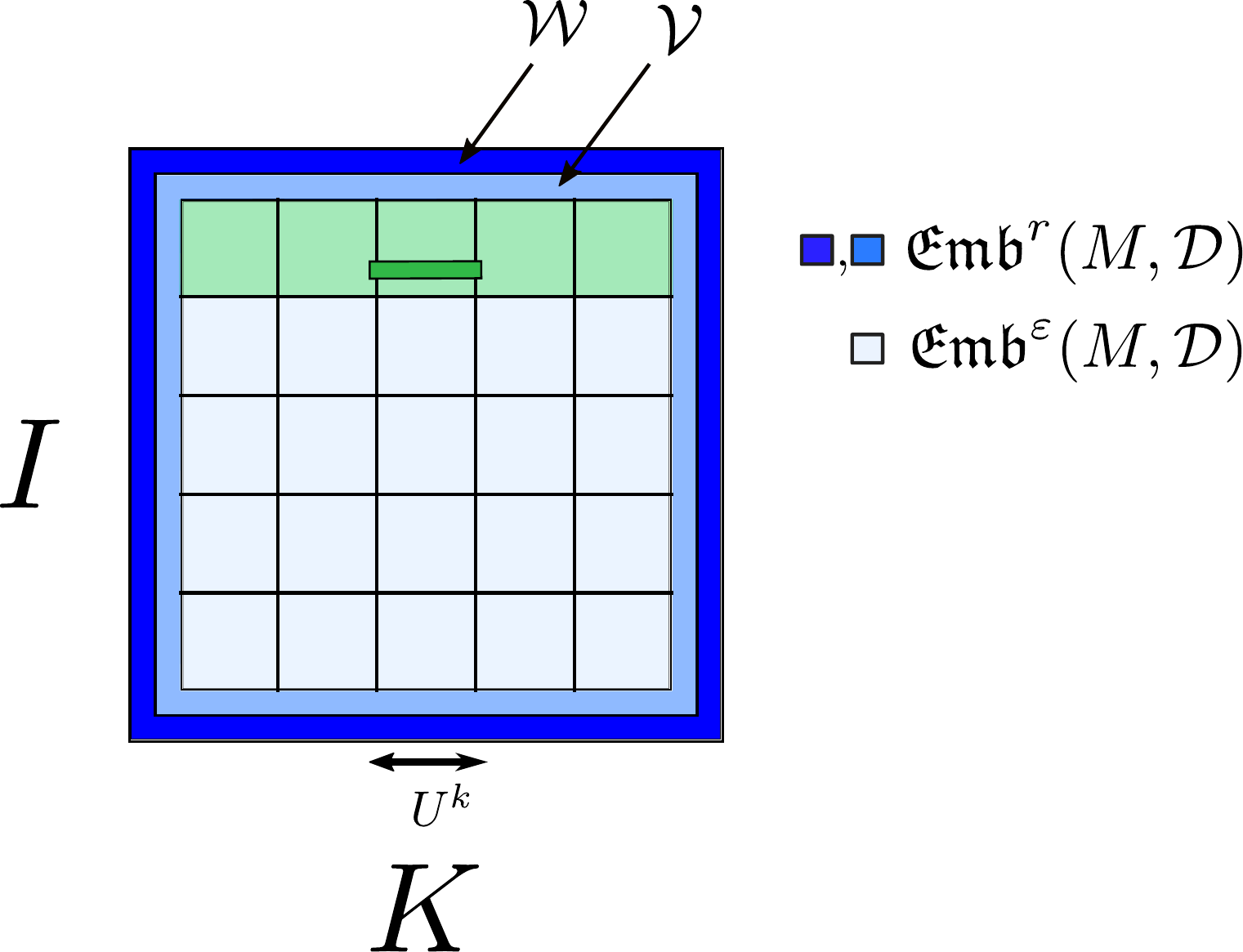}
	\centering
	\caption{The cover $\SU$, together with the boundary neighbourhoods of radii $\tau$ and $\tau/2$. A controller is shown in one of the $U_i \times I_j$.}\label{cuadricula}
\end{figure}

\subsubsection{Introducing controllers}

Given $U_i \times I_j \in \SU$, we choose a time $t_{i,j}$ in the interior of $I_j$. We require that these are all distinct. We then introduce a controller $\SC_{i,j}$ (Lemma \ref{lem:insertionControllers}) along $U_i \times \{t_{i,j}\}$. We write:
\begin{itemize}
\item $S$ for the size-at-rest of all the controllers.
\item $\eta>0$ for the size of the neighbourhood of $U_i \times \{t_{i,j}\}$ in which the controllers are contained. If $\eta$ is sufficiently small, the controllers do not interact with one another. 
\end{itemize}
Later on in the proof we will use the estimated-displacement of the $\SC_{i,j}$, one pair $(i,j)$ at a time. For now we write $\gamma'$ for the $K$-family of $\varepsilon$-horizontal curves that has the estimated-displacement of each controller at $0$. Do note that $S$ must be small enough to guarantee $\varepsilon$-horizontality. Furthermore, $\gamma'$ is homotopic to $\gamma$ through $\varepsilon$-horizontal curves, relative to the complement of all $\nu_\tau(U_i \times \{t_{i,j}\})$.

There is now a subtlety that we need to take care of: some of the $\SC_{i,j}$ enter the region $\nu_\tau(\partial (K \times I))$, destroying the horizontality condition there. This cannot be addressed using the controllers themselves, since the collection $\{U_i\}$ does not cover $\nu_\tau(\partial K)$ completely. We address it instead using regularity. Note that this is not a technical point: due to the phenomenon of rigidity (Subsection \ref{ssec:regularity}), the usage of regularity at this stage cannot be avoided.

\subsubsection{Horizontalisation} \label{sssec:auxiliaryHorizontalisation}

The family $\gamma'$ is horizontal over $\Op(\partial K)$, but not necessarily over the whole band $\nu_\tau(\partial K)$, due to our insertion of controllers. To address this, we reintroduce horizontality, at the cost of losing control of the endpoints. Namely, given $\mu > 0$, any sufficiently small size-at-rest $S$ will guarantee that there is a family of horizontal curves $(\alpha(k))_{k \in \nu_\tau(\partial K)}$ that satisfies:
\begin{itemize}
\item[i. ] $\left| \alpha - \gamma|_{\nu_\tau(\partial K)} \right|_{C^0} < \mu$.
\item[i'. ] $\len(\alpha(k)) < \len(\gamma(k)) + \mu$.
\item[ii. ] $\alpha$ is homotopic to $\gamma|_{\nu_\tau(\partial K)}$ through horizontal curves.
\item[iii. ] This homotopy is relative in the parameter to $\nu_{\tau/2}(\partial K)$.
\item[iii'. ] The homotopy is relative to $\{t=0\}$ in the domain.
\end{itemize}
The family $\alpha$ is constructed inductively, one chart $U_i \times I_j$ at a time, increasingly in $j$, and arbitrarily in $i$. The inductive step consists of using each adapted chart $(V_{i,j},\Psi_{i,j})$ to see $\gamma'|_{U_i \times I_j}$ as a family of curves in the graphical model $V_{i,j}$. We can then apply the lifting Lemma \ref{lem:horizontalisationGraphical} to the projection $\pi \circ \gamma'|_{U_i \times I_j}$. The lifting process can be completed over $[0,1]$ because a small $S$ means that $\pi \circ \gamma'|_{U_i \times I_j}$ is close to $\pi \circ \gamma|_{U_i \times I_j}$. This also justifies Conditions (i) and (i'). The families are homotopic two one another by the birth of the controller, proving Condition (ii). Lastly, projecting and lifting leaves horizontal curves invariant, proving Conditions (iii) and (iii').

We now choose $\mu$ small enough so that the bounds provided by Conditions (i) and (i') allows us to invoke the interpolation Lemma \ref{lem:varepsilonInterpolation}. This allows us to interpolate through $\varepsilon$-horizontal curves between $\alpha$ and $\gamma'$ in the region $\{\tau' < d(k,\partial K) < \tau \}$. The resulting family of curves $\alpha'$:
\begin{itemize}
\item agrees with $\gamma'$ in the complement of $\nu_\tau(\partial K)$,
\item agrees with $\gamma$ in $\nu_{\tau/2}(\partial K)$,
\item is horizontal in $\nu_{\tau'}(\partial K)$,
\item contains controllers along $U_i \times \{t_{i,j}\}$.
\end{itemize}
The issue is that $\alpha(k)(1)$ may be different from $\gamma'(k)(1) = \gamma(k)(1)$ in the region $\{\tau/2 < d(k,\partial K) < \tau' \}$. This is a feature of the lifting process. Nonetheless, according to Lemma \ref{lem:horizontalisationGraphical}, the endpoints are $\mu$-close.

\subsubsection{Controllability}

By hypothesis, the curves $\gamma|_{\nu_\tau(\partial K)}$ are horizontal and regular. It follows that $\gamma|_{\nu_\tau(\partial K)}$ is a controllable family\footnote{Do note that $\nu_\tau(\partial K)$ is not compact but, since $\tau$ is arbitrary, we can take a slightly smaller compact neighbourhood of $\partial K$ and carry the argument there.}, according to Lemma \ref{lem:controlling}. We deduce that there are constants $c,\delta > 0$ such that any $\delta$-displacement of their endpoints can be followed by a homotopy of the curves themselves, through horizontal curves, that is $c\delta$-small.

We claim that the family $\alpha'|_{\nu_{\tau'}(\partial K)}$ is also controllable, with constants $2c$ and $\delta/2$, as long as $S$ is sufficiently small. Indeed, consider the variations $F$ of $\gamma$ that yield controllability. Then, the homotopy lifting property, applied to $F$ and the homotopy of horizontal curves connecting $\gamma|_{\nu_{\tau'}(\partial K)}$ with $\alpha'|_{\nu_{\tau'}(\partial K)}$, yields corresponding variations for small values of the homotopy parameter. They exist for the whole homotopy if $\mu$ is assumed to be sufficiently small.

Then, assuming that $S$ is sufficiently small, we have that $\mu < \delta/2$ and we can use the controllability of $\alpha'|_{\nu_{\tau'}(\partial K)}$ to yield a family $\gamma'': K \to \Emb^\varepsilon(I;M,\SD)$ that:
\begin{itemize}
\item agrees with $\gamma$ in $\Op(\partial (K \times I))$.
\item is horizontal and regular if $k \in \nu_{\tau'}(\partial K)$.
\item has a family of controllers, still denoted by $\{\SC_{i,j}\}$, along $U_i \times \{t_{i,j}\}$.
\end{itemize}
The situation is depicted in Figure \ref{fig:controllersIntroduced}.

\begin{figure}[h] 
	\includegraphics[scale=1.2]{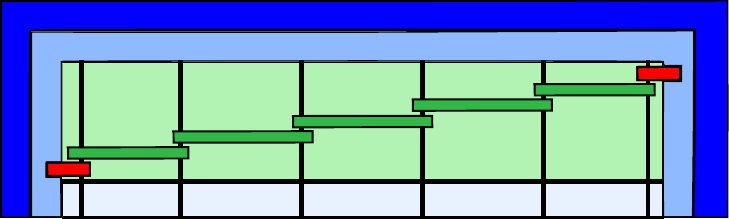}
	\centering
	\caption{For each element in the cover $\SU$, a controller has been introduced. These are shown as green thin rectangles. Close to the boundary, some controllers (in red) enter its $\tau$-neighbourhood. Horizontality is reestablished using the variations given by local controllability.}\label{fig:controllersIntroduced}
\end{figure}

\subsubsection{Discussion about parameters}

In this subsection we added controllers $\{\SC_{i,j}\}$ to the family $\gamma$. Each curve $\gamma(k)$ crosses at most $N.C_1$ controllers; here $C_1$ is the upper bound for the intersections between elements in $\{U_i\}$. Inserting the controllers produces a deformation $\alpha$ of $\gamma|_{\nu_{\tau'}(\partial K)}$ through horizontal curves. This deformation displaces the endpoints an amount $\mu$, which we can estimate. For each controller inserted, the endpoints move a magnitude $O(S)$, the size-at-rest. This implies that $\mu$ is bounded above by $O(S).N.C_1$.

Furthermore, in Subsection \ref{ssec:Step1Proof} we showed that $\gamma$ satisfies the size estimates given in Equations \ref{eq:parameterR1} and \ref{eq:parameterL}, involving $r_1$ and $l$. We want $\alpha$ and thus $\gamma'$ to satisfy these as well. To this end, the $C^0$-distance between $\gamma$, $\alpha$, and $\gamma'$ must be much smaller than $r_1$ and $l$.

These considerations force the choice $S << 1/N << l, r_1$.

\subsection{Using the controllers} \label{ssec:Step3Proof}

In this subsection we complete the proof of Proposition \ref{prop:relativeHorizontalEmbeddings}. The idea is to use projection and lifting to replace $\gamma''$ by a family of horizontal curves $\beta$, whose endpoints are incorrect. The endpoints will then be adjusted thanks to the presence of controllers.

\subsubsection{Horizontalisation}

Much like in the proof of Theorem \ref{thm:embeddings}, we first apply Lemma \ref{lem:horizontalisationSkeleton} to $\gamma''$ at each time $t_j = j/N \in I$. This can be done in a $C^0$-small way, through $\varepsilon$-horizontal curves, by making the newly created horizontal region sufficiently small. This is relative to $\nu_{\tau'}(\partial K)$ in the parameter. The resulting family is denoted by $\gamma_0$. The proof now focuses on a concrete interval $I_j$; the argument is identical for all of them.

\subsubsection{Horizontalisation again}

We have a family $\gamma_0$ with values in $\Emb^\varepsilon(I_j;M,\SD)$ that along the boundary of $K \times I_j$ is horizontal. Suppose $1/N$ is sufficiently small. We can argue as in Subsection \ref{sssec:auxiliaryHorizontalisation} to construct a $K$-family $\beta$ with values in $\Emb(I_j;M,\SD)$ such that:
\begin{itemize}
\item $\beta(k) = \gamma_0(k)$ if $d(k,\partial K) \leq \tau'$.
\item $\beta(k)(t) = \gamma_0(k)(t)$ if $t \in \Op(\{t_j\})$.
\item Write $\pi$ for the projection to the base given by each graphical model $V_{i,j}$. Then, the $C^\infty$-closeness of $\pi \circ \beta$ and $\pi \circ \gamma_0$ is controlled by $S$.
\item In particular, the two families are homotopic through $\varepsilon$-horizontal curves.
\end{itemize}
In particular, the controllers of the family $\gamma_0$ define controllers for $\beta$. We still denote them by $\{\SC_{i,j}\}_i$. The goal now is to use these to produce a homotopy of horizontal curves between $\beta$ and the claimed $\widetilde\gamma|_{K \times \{1\}}$.

\subsubsection{Adjusting the endpoint over one chart}

The difference $e = |\gamma_0(k)(t_{j+1}) - \beta(k)(t_{j+1})|$ is certainly bounded above by $r_1$. However, as explained in Subsection \ref{sssec:ODEsize}, it can also be bounded above by $C_2.r_1.l$. It follows that we should impose $l << 1/(C_2.r_1)$ to make $e$ much smaller than $r_1$. By making this choice, Equations \ref{eq:parameterR1} and \ref{eq:parameterL} hold for $\beta$.

We now perform induction on $i$ to correct this difference. We start with the base case $i=1$, so we work over the chart $U_1$. Adjusting the estimated-displacement of the controller $\SC_{1,j}$ yields a homotopy of horizontal curves 
\[ \widetilde\beta: U_1 \times A \subset \R^{n-q} \quad\longrightarrow\quad \Emb(I_j;M,\SD) \]
with $\widetilde\beta(k,0) = \beta(k)$. The variable $a \in A$ measures the endpoint displacement introduced by the controller vertically. We package this as an endpoint map
\[ \widetilde\beta(-,-)(t_{j+1}): U_1 \times A \longrightarrow \R^{n-q} \]
which satisfies the following error estimate:
\[ \widetilde\beta(k,a)(t_{j+1}) = \beta(k)(t_{j+1}) + a.(1+C_3.(r_1+|a|+\delta)). \]
According to Lemma \ref{lem:adaptedCharts}, the constant $C_3$ appearing as the coefficient of the error term $(r_1+|a|)$ is independent of the adapted chart and thus independent of the controller. This forces us to choose $\delta, r_1 << 1/C_3$. The quantity $|a|$ will be of magnitude $e$ and thus smaller than $r_1$.

This choice tells us that $\widetilde\beta(k,-)(t_{j+1})|_{\D_{2r_1}}$ is an embedding whose image contains $\D_{r_1}$, for all $k \in V_1$. In particular, it contains the desired endpoint $\gamma_0(k)(t_{j+1})$. The inverse function theorem (Lemma \ref{lem:inverseFunctionTheorem}) defines for us a unique function $a: V_1 \rightarrow \D_{2r_1}$ so that $\widetilde\beta(k,a(k))(t_{j+1}) = \gamma_0(k)(t_{j+1})$.

We now cut-off the function $a$, in order to make the construction relative to the boundary of $U_1 \times I_j$. Fix a constant $\rho > 0$ and write $W_i \subset U_i$ for a domain covering $U_i$ up to a $\rho$-neighbourhood of its boundary. We require that the family $\{W_i\}$ is an open cover of $K \setminus \nu_{\tau'}(\partial K)$. This imposes $\rho << 1/N,\tau$. This allows us to introduce a cut-off function $\chi_1: K \rightarrow [0,1]$ that is one in $W_1$ and zero along $\partial U_1$. We set $\beta_1(k) := \widetilde\beta(k,\chi(k).a(k))$. This is a family of horizontal curves such that:
\begin{itemize}
\item $\beta_1(k)$ is horizontal.
\item $\beta_1$ is homotopic to $\beta$ as maps into $\Emb(I_j;M,\SD)$.
\item The base projections of $\beta_1(k)$ and $\gamma_0(k)$ agree for all $k \in \Op(\partial U_1)$.
\item $\beta_1(k)$ and $\gamma_0(k)$ agree over $t \in \Op(\partial I_j)$, for every $k \in W_1$.
\end{itemize}
The third property allows us to homotope $\beta_1$ to a family $\gamma_1: U_1 \to \Emb^\varepsilon(I_j;M,\SD)$ that is horizontal over $W_1$ and agrees with $\gamma_0$ in $\Op(\partial U_1)$. We can then use $\gamma_0$ to extend $\gamma_1$ to a family $K \to \Emb^\varepsilon(I_j;M,\SD)$. The two are homotopic, relative to endpoints and to the complement of $U_1$, thanks to the second property and the fact that $\beta$ was the horizontal lift of $\gamma_0$.

\subsubsection{The inductive argument}

The $i_0$th inductive step follows the exact same argument. It produces a family $\gamma_{i_0}: K \to \Emb^\varepsilon(I_j;M,\SD)$ that is horizontal over the union $\cup_{i \leq i_0} W_i$. The observation to be made is the following. Suppose $U_{i_0}$ intersects non-trivially some previous $U_i$. Then, in the overlap $U_{i_0} \cap W_i$, we have that the family $\gamma_{i_0-1}$ is already horizontal over $W_i$. It follows that the associated horizontal family $\beta$ agrees with $\gamma_{i_0-1}$ over $W_i$, due to the uniqueness of horizontal lifts. In particular, when we use the controller $\SC_{i_0,j}$ to produce a fully controllable family over $W_{i_0}$, we see that no adjustments must be made over $W_i$, since the endpoint is already correct there. This is immediate from the uniqueness provided by the inverse function theorem.

Since the $W_i$ cover $K \setminus \nu_\tau(\partial K)$, the inductive argument produces the required homotopy $\widetilde\gamma$. The proof of Proposition \ref{prop:relativeHorizontalEmbeddings} is complete.

\subsubsection{Final discussion about constants} \label{sssec:finalConstants}

Two quantities had to be controlled in this section. The first was the error suffered by the controllers. This was proportional (Lemma \ref{lem:errorControllers}) to the radius of the graphical models (and could therefore be controlled by $r_1$) and by the smoothing parameter $\delta$. The other quantity was the $C^0$-distance $e$ between $\beta$ and $\gamma_{i-1}$ (particularly at their endpoints). This was controlled by setting $l << r_1$.

We note that the embedding condition enters the discussion only in the choice of $r_1$. Namely: once the error of the controllers has been bounded, it follows that the horizontal curves produced by the controller are indeed embedded. See Lemma \ref{lem:errorControllers}. 

The summary is that we require the chain of inequalities $\delta << S << 1/N << l << r_1$.  \hfill$\Box$

\begin{figure}[h] 
	\includegraphics[scale=0.8]{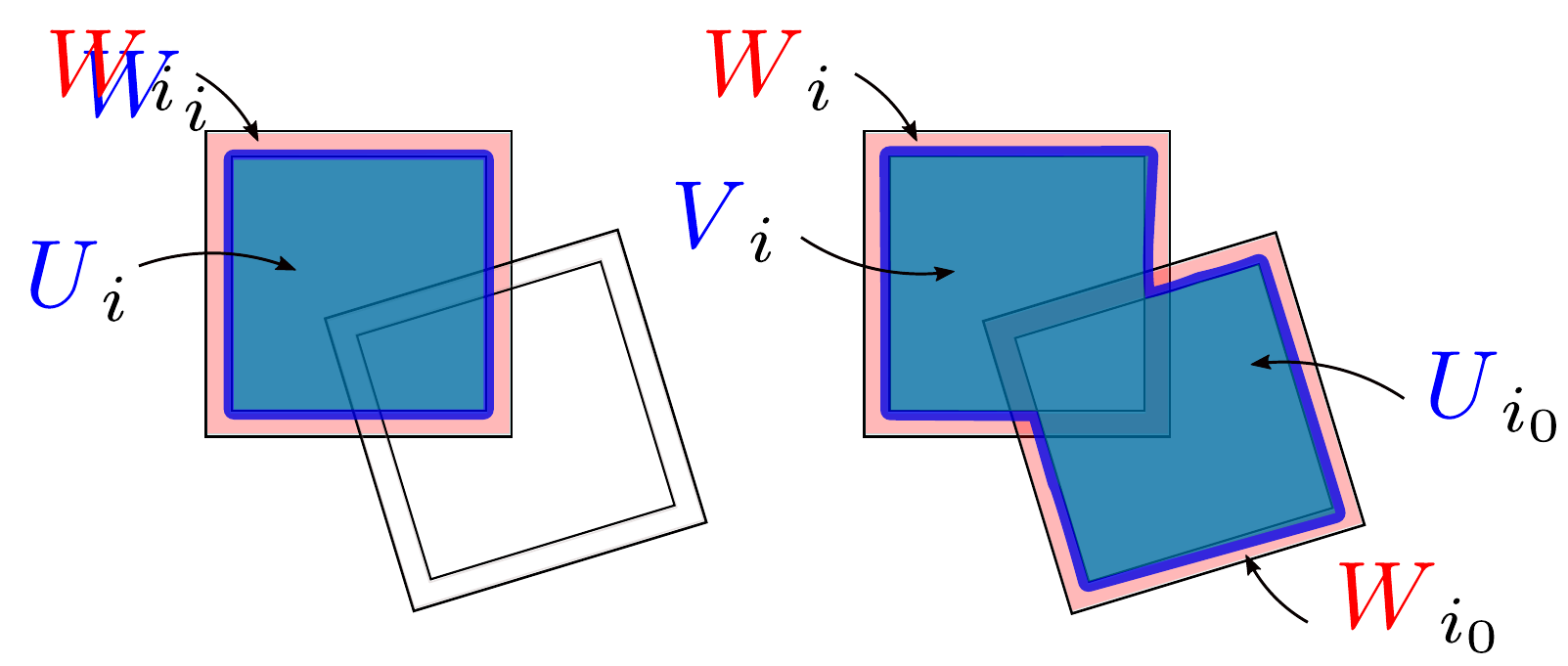}
	\centering
	\caption{The inductive process. In the $i$th step of the induction we introduce horizontality over the region $W_i$, in dark blue. Appropriate cut-off functions have been introduced in $V_i \setminus W_i$ (light red) to make this homotopy relative to the boundary of the model. At a later stage, we consider some $U_{i_0}$ overlapping with $W_i$. In the overlap, the step $i_0$ homotopy is constant, thanks to horizontality.}\label{InduccionCuadricula}
\end{figure}

\begin{remark} \label{rem:whyNotTriangulate}
The reader experienced in $h$-principles may wonder why we do not use a triangulation of $K \times I$, in general position with respect to $\pi_K$, to argue. Indeed, this would have the added advantage of localising our arguments to balls that do not interact with one another. This was not the case in the proof we presented.

The issue with the triangulation approach is that we would have to make a first homotopy that makes our curves horizontal along the codimension-1 skeleton. This is certainly possible, but we have no guarantee that the produced curves are themselves regular. This is absolutely necessary, since we need to be able to introduce controllers at the bottom of each top-cell.

In fact, this can be made to work. The local integrability of micro-regular curves (i.e. curves that are in particular regular over any interval) was proven in \cite{PS,Bho}, which would allow us to produce regular curves along the skeleton. However, it seemed preferable to us to keep the proof self-contained and not invoke additional results. \hfill$\Box$
\end{remark}

\subsection{Other $h$-principles for horizontal curves} \label{ssec:hPrincipleHorizontalImmersions}

We now discuss how the proof of Theorem \ref{thm:embeddings} adapts to prove Theorems \ref{thm:TransverseImmersions} and \ref{thm:Ge}.

\begin{proof}[Proof of Theorem \ref{thm:TransverseImmersions}]
The absolute statement can be reduced to proving the relative $h$-principle over the interval (i.e. the analogue of Proposition \ref{prop:relativeHorizontalEmbeddings}). The proof of the relative statement is identical to the one we presented for embeddings. The reader can check that the proof goes through line by line. Instead, it is more interesting to point out how the proof simplifies for immersions. 

First note that the construction of tangles (Section \ref{sec:tangles}) and controllers (Section \ref{sec:controllers}) is less involved if we do not need to take care of self-intersections. In particular, we do not need to develop all the explicit models shown in Figures \ref{EmbeddedCrossings} and \ref{fig:EmbeddedCrossings2}. Similarly, in Subsection \ref{sssec:finalConstants}, we do not need to control errors in order to ensure the controller produces embedded curves.

This crucial difference explains why the statement for immersions goes through in $\dim(M) = 3$. Achieving embeddedness parametrically was not possible in dimension $3$, but one can certainly produce tangles and controllers that are immersed.
\end{proof}

Similarly:
\begin{proof}[Proof of Theorem \ref{thm:Ge}]
The statement reduces once again to the $h$-principle for horizontal paths, relative both in parameter and domain. We now indicate the differences with respect to the proof of Proposition \ref{prop:relativeHorizontalEmbeddings}.

First: since there is no first order formal data, we do not need to set up a convex integration argument to achieve $\varepsilon$-horizontality. Second: we use the covering arguments as they appear in Subsection \ref{ssec:Step1Proof} but we run into issues in Subsection \ref{ssec:Step2Proof}, when we try to introduce controllers. In Proposition \ref{prop:insertionControllers} we explained how to introduce them when our curves are embedded/immersed, which may not be the case here. To address this, we use the  ``stopping trick''; see Figure \ref{StoppingTrick}. We explain it next.

Given a smooth horizontal arc $\gamma: [a,b] \rightarrow (M,\SD)$, we can precompose it with a non-decreasing map $\phi: [a,b] \rightarrow [a,b]$ that is the identity at the endpoints and is constant in $\Op(\{(a+b)/2\})$. Since $\phi$ is homotopic to the identity rel boundary, it defines a homotopy between $\gamma$ and a horizontal curve $\gamma \circ \psi$ whose parametrisation is stationary at the middle point. The curve $\gamma \circ \psi$ is thus regular. Even more: suppose $\nu: [0,1] \rightarrow (M,\SD)$ is some other horizontal curve with $\nu(0) = \gamma \circ \psi((a+b)/2)$. Then $\gamma \circ \psi$ is homotopic to a smooth horizontal curve $\gamma'$ whose image is the concatenation
\[ (\gamma \circ \psi|_{[(a+b)/2,b]}) \bullet \bar\nu \bullet \nu \bullet (\gamma \circ \psi|_{[a,(a+b)/2]}). \]
Since $\nu$ is arbitrary, it may be chosen to be a curve contained in a graphical model around $\gamma \circ \psi((a+b)/2)$ and projecting to the base as whatever direction we require.

The conclusion is that any family $\gamma: K \longrightarrow \SL(M)$, horizontal at the boundary, can also be assumed to be regular along $\partial K$, up to a homotopy through horizontal curves. Furthermore, it can be assumed to be immersed whenever controllers need to be introduced. This concludes the proof.
\end{proof}

\begin{figure}[h] 
	\includegraphics[scale=0.125]{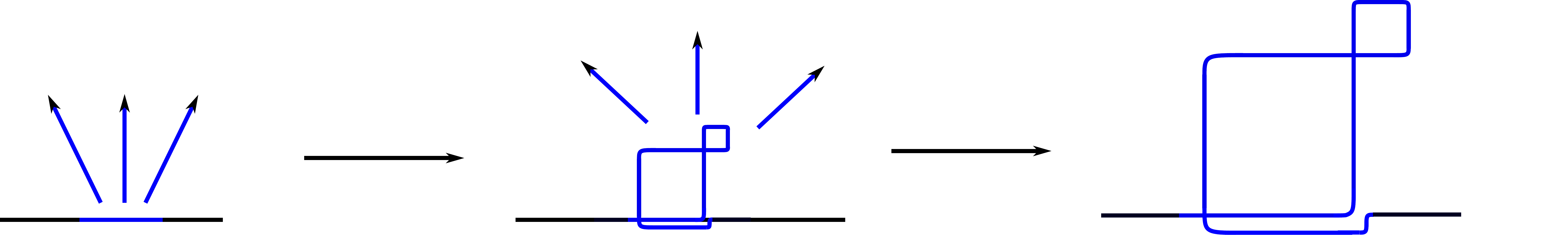}
	\centering
	\caption{The stopping trick. The space of horizontal loops is extremely flexible. Any family can be homotoped to a regular family by introducing stationary points in the parametrisation. This allows us to introduce controllers.}\label{StoppingTrick}
\end{figure}

\section{h-Principle for transverse embeddings} \label{sec:hPrincipioTransverse}

In this section we prove Theorem \ref{thm:TransverseEmbeddings}, the classification of transverse embeddings.

\subsection{The relative $h$-principle} \label{ssec:relativeTransverseEmbeddings}

The $h$-principle for transverse paths, relative in parameter and domain, reads:
\begin{proposition} \label{prop:relativeTransverseEmbeddings}
Let $K$ be a compact manifold. Let $I = [0,1]$. Let $(M,\SD)$ be a manifold of dimension $\dim(M)>3$, endowed with a bracket--generating distribution of corank $1$. Suppose that we are given a map $\gamma: K \to \Embt^f(I;M,\SD)$ satisfying:
\begin{itemize}
\item $\gamma(k) \in \Embt(I;M,\SD)$ for $k \in \Op(\partial K)$.
\item $\gamma(k)(t)$ is positively transverse if $t \in \Op(\partial I)$.
\end{itemize}

Then, there exists a homotopy $\widetilde\gamma: K \times [0,1] \to \Embt^f(I;M,\SD)$ satisfying:
\begin{itemize}
\item $\widetilde\gamma(k,0) = \gamma(k)$.
\item $\widetilde\gamma(k,1)$ takes values in $\Embt(I;M,\SD)$.
\item this homotopy is relative to $k \in \Op(\partial K)$ and to $t \in \Op(\partial I)$.
\item $\widetilde\gamma(k,s)$ is $C^0$-close to $\gamma(k)$ for all $s \in [0,1]$.
\end{itemize}
\end{proposition}
Observe that we do not assume that $\SD$ is cooriented. We will be able to pass to the cooriented case (and thus use $\varepsilon$-transverse embeddings) during the proof.

\begin{proof}[Proof of Theorem \ref{thm:TransverseEmbeddings} from Proposition \ref{prop:relativeTransverseEmbeddings}]
We must prove the vanishing of the relative homotopy groups of the pair 
\[ (\Embt^f(M,\SD),\, \Embt(M,\SD)). \]
Given a family $\gamma$ representing a class in the $a$th relative homotopy group, we must deform it to lie entirely in $\Embt(M,\SD)$. Close to $\partial\D^a$ the family is transverse because transversality is an open condition. We then fix a slice $\D^a \times \{1\} \subset \D^a \times \NS^1$ and apply the transversalisation Lemma \ref{lem:transversalisationSkeleton} to $\gamma$ there. This reduces the argument to a family of paths and thus to Proposition \ref{prop:relativeTransverseEmbeddings}.
\end{proof}
We henceforth focus on the proof of Proposition \ref{prop:relativeTransverseEmbeddings}

\subsection{The first triangulation step} \label{ssec:Step1ProofTrans}

We are given a family of formally transverse curves $\gamma$. Our first goal is to pass to the $\varepsilon$-transverse setting. To do so, we will produce a very fine subdivision of $K \times I$ so that we can work over balls in $(M,\SD)$. We can then use the local coorientability of $\SD$ to introduce $\varepsilon$-transversality.

\subsubsection{Triangulating}

We subdivide $K \times I$ using a triangulation $\ST$. This should be compared to the proof of Proposition \ref{prop:relativeHorizontalEmbeddings}, which used a different scheme to localise the arguments to little balls. The reason was explained in Remark \ref{rem:whyNotTriangulate}: Triangulating $K \times I$ would have led to issues due to the phenomenon of rigidity for horizontal curves. However, there is no rigidity for transverse curves because they are defined by an open condition.

In order to produce a triangulation, we proceed as follows. We pick a small constant $\tau > 0$ so that $\gamma$ is transverse over $\nu_\tau(\partial(K \times I))$. We choose a closed domain $A \subset K \times I$ such that $\{A,\nu_\tau(\partial(K \times I))\}$ covers $K \times I$. We ask that $\partial A$ is smooth. We then apply the jiggling Corollary \ref{cor:jigglingBoundary} to $(A, \ker(d\pi_K))$. We do not need to introduce further subdivisions, we simply choose $\ST$ fine enough so that each top-cell $\Delta \in \ST$ is mapped by $\gamma$ to an adapted chart of $(M,\SD)$.

According to Lemma \ref{lem:jiggling}, the following holds: a simplex $\Delta$ is either transverse to the vertical or is contained in $\partial A$. In the latter case, $\gamma$ is already transverse in $\Op(\Delta)$. It follows that we can apply Lemma \ref{lem:transversalisationFormalSkeleton} to $\gamma$, along the codimension-$1$ skeleton, in a manner relative to $\partial A$. This yields a homotopy of formal transverse embeddings, relative to $\partial A$, between $\gamma$ and a family $\gamma_1$. The family $\gamma_1$ is transverse on a neighbourhood of the codimension-$1$ skeleton.

\subsubsection{$\varepsilon$-transversality}

Let $\Delta \in \ST$ be a top-cell. Since $\gamma_1$ maps it to to an adapted chart $(V,\Phi)$, we have that $\gamma_1|_\Delta$ takes values in a manifold endowed with a coorientable distribution. Furthermore, since $\gamma_1|_\Delta$ is formally transverse, it defines a preferred coorientation for $\SD|_{\Phi(V)}$.

According to Lemma \ref{lem:jiggling}, the top-cells of $\ST$ are flowboxes, meaning that there is a fibre-preserving embedding 
\[ \Psi: \D^k \times [0,1] \longrightarrow \Delta \subset A \subset K \times I \]
whose image covers most of $\Delta$. In particular, the boundary of this embedding may be assumed to be contained in the region where $\gamma$ is transverse.

The conclusion is that $\Phi^{-1} \circ \gamma_1 \circ \Psi$ is a family of formally transverse curves that:
\begin{itemize}
\item takes values in a cooriented graphical model $(V,\SD_V)$,
\item is transverse over $\Op(\partial(\D^k \times [0,1]))$.
\end{itemize}
This implies that we can apply the $h$-principle for $\varepsilon$-transverse curves (Subsection \ref{ssec:epsilonTransverseCurves}) to $\Phi^{-1} \circ \gamma_1 \circ \Psi$, yielding a family $\gamma_2: \D^k  \longrightarrow \Embt^\varepsilon([0,1];V,\SD_V)$ that is (positively) transverse over $\Op(\partial(\D^k \times [0,1]))$. 

We have one such family per top-cell $\Delta$. Furthermore, their domains are disjoint. This implies that it is sufficient for us to work with each $\gamma_2$ individually, homotoping them through $\varepsilon$-transverse curves (and relative to $\partial(\D^k \times [0,1])$) to a family of transverse curves.

\subsection{The second triangulation step} \label{ssec:Step2ProofTrans}

Let us explain how the remainder of the proof goes, morally. Our goal is to apply the case of horizontal embeddings. The reasoning is that, whenever the curves $\gamma_2(k)$ are graphical over $\SD_V$, we can flatten them to make them almost horizontal, which will then allow us to manipulate them through the use of controllers. In order to project towards $\SD_V$, we use the projection $\pi: \R^n \rightarrow \R^{n-1}$ to the base of the graphical model.

There are two issues with this idea. The first is that the family $\gamma_2$ may have uncontrolled length, since it was produced by the $h$-principle for $\varepsilon$-transverse curves. To address this, we will triangulate again to pass to small balls. The other issue will be explained afterwards.

\subsubsection{Triangulating again...}

We proceed as above, applying the jiggling Lemma \ref{lem:jiggling} to $\gamma_2$ and $(\D^k \times [0,1], \ker(d\pi_{[0,1]}))$. We obtain a sequence of triangulations $\ST_b$ inducing a triangulation of the boundary (do note that it has corners, but this is not an issue). We fix $b$ as in the horizontal case (Subsection \ref{ssec:Step1Proof}): each simplex $\Delta$ should have diameter bounded above by $1/N << l, r_1$. In this manner, $\gamma_2(\Delta)$ is contained in an adapted chart $(U,\Psi)$ of radius $r_1$ and each curve in $\gamma_2|_{\Delta}$ has length at most $l$ from the perspective of the corresponding graphical model $U$. We still refer to these conditions as Equations \ref{eq:parameterR1} and \ref{eq:parameterL}.

We now apply Lemma \ref{lem:transversalisationSkeleton} to $\gamma_2$, along the codimension-1 skeleton of $\ST_b$. This shows that $\gamma_2$ is homotopic, relative to the boundary, to a family $\gamma_3$ that is transverse along the codimension-1 skeleton. This is a $C^0$-small process that does not increase the length of the curves much. It follows that the conditions given by Equations \ref{eq:parameterR1} and \ref{eq:parameterL} apply to $\gamma_3$ as well.

As we did earlier, we can embed a copy of $\D^k \times [0,1]$ into each top-simplex, in a fibered manner, in such a way that its boundary lies in the region where $\gamma_3$ is transverse. These domains do not interact with one another, so we argue on each of them separately. 

\subsubsection{... and again}

We now encounter the second issue. We cannot flatten $\gamma_3$ to make it almost horizontal along the locus
\[ \Sigma := \lbrace (k,t) \in D^k \times [0,1] \,\mid\, d\pi(\gamma_3(k)'(t)) = 0 \rbrace. \]
That is, the locus where the velocity vector becomes vertical. Nonetheless, using Thom transversality \cite[p. 17]{EM}, we can assume that $\Sigma$ is a closed submanifold of $\D^k \times [0,1]$.

\begin{figure}[h] 
	\includegraphics[scale=1]{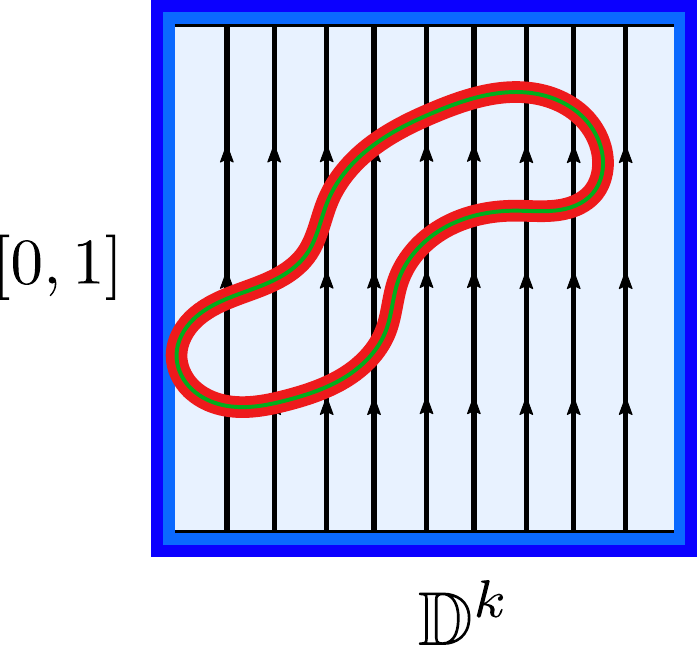}
	\centering
	\caption{Schematic representation of $\D^k \times [0,1]$. The blue band represents a neighbourhood of the boundary. The thin green curve is the vertical locus $\Sigma$. Its neighbourhood $W$ is shown in red. The region $B$ is a slight thickening of the complement.}\label{fig:locusB}
\end{figure}

The family $\gamma_3$ is vertical along $\Sigma$ and $\varepsilon$-transverse in general. It follows that $\gamma_3$ is positively transverse on a neighbourhood $W \supset \Sigma$. We can therefore find a closed subdomain $B \subset \D^k \times [0,1]$, disjoint from $W$, whose smooth boundary lies in the region where $\gamma_3$ is transverse. See Figure \ref{fig:locusB}.

We can then proceed as above, applying Lemma \ref{lem:jiggling} to $\gamma_3$ and $(B, \ker(d\pi_{[0,1]}))$. We do not need the resulting triangulation $\ST'$ to be thin, since we already achieved quantitative control in the previous subdivision. Applying Lemma \ref{lem:transversalisationSkeleton} shows that $\gamma_3$ is homotopic, relative to $\partial B$, to a family $\gamma_4$ that is transverse along the codimension-1 skeleton. This is a $C^0$-small process.

We henceforth argue on each top simplex separately, relative to the boundary. The punchline is that we have a family
\[ \gamma_4: \D^k  \longrightarrow \Embt^\varepsilon([0,1]; ,\SD_U), \]
that is transverse along the boundary. Here $(U,\SD_U)$ is the graphical model of radius $r_1$ that we fixed earlier (and that we used to discuss verticality). Each curve of $\gamma_4$ has length bounded above by $l$. Since we avoided $\Sigma$, we can assume that each curve $\gamma_4(k)$ is graphical over $\SD_U$.

\subsection{End of the argument} \label{ssec:Step3ProofTrans}

Choose $\tau>0$ small enough so that $\gamma_4$ is transverse in a $\tau$-neighbourhood of the boundary of $\D^k \times [0,1]$. We can apply Corollary \ref{cor:horizontalisationSkeleton} to $\gamma_4$ along the slice $D = \D^k_{1-\tau/2} \times \{\tau/2\}$, in order to yield a family that is almost transverse in $\D^k \times [0,\tau]$ and horizontal in $\Op(D)$. This allows us to introduce a controller $\SC$ along $D$; see Lemma \ref{lem:insertionControllers}. The resulting family is called $\gamma_5$. It consists of embedded curves as long as $r_1$ was sufficiently small (Lemma \ref{lem:errorControllers}).

We now adjust the estimated-displacement of $\SC$ in order to obtain an almost transverse family $\gamma_6$ such that
\[ e = \left| \gamma_6(k)(1-\tau) - \gamma_5(k)(1-\tau) \right| \]
is small. We require that
\[ e << h = \left| \gamma_5(k)(1) - \gamma_5(k)(1-\tau) \right|. \]
If that is the case, we can invoke the fact that $\gamma_5(k)$ is transverse in the interval $[1-\tau,1]$, to extend $\gamma_6(k)$ to a curve that is transverse in $[1-\tau,1]$ and agrees at the end with $\gamma_5(k)$. This preserves embeddedness.

The proof concludes invoking Lemma \ref{lem:transversalisationAT}, which adds a $C^1$-small perturbation to $\gamma_6$ to yield a family that is transverse. \hfill$\Box$

\begin{proof}[Proof of Theorem \ref{thm:TransverseImmersions}]
The proof follows from the same arguments. As we already observed for horizontal embeddings, handling the controller becomes easier in the immersion case, since self-intersections do not need to be avoided. Because of this reason, the statement also holds in dimension $3$.
\end{proof}

\section{Appendix: Technical lemmas on commutators}\label{Appendix}

Given a vector field $Z$ on $M$ we write $\phi^Z_t$ for its flow at time $t$. Given a pair of vector fields $X$ and $Y$ on $M$, we want to compare, in a quantitative manner, the flow of their Lie bracket $\phi^{[X,Y]}_t$ with the commutator of their flows $\phi_t^X$ and $\phi_t^Y$. The contents of this appendix will be an important technical ingredient in the proof of our main theorems.

Given $1$--parameter families (not necessarily subgroups) $(\varphi_t)_{t \in \R}$ and $(\psi_t)_{t \in \R}$ in the diffeomorphism group $\Diff(M)$, we can define in a given local chart the map $[\psi(t),\varphi(s)]= \varphi_s \circ \psi_t \circ \varphi_s^{-1} \circ \psi_t^{-1}(x)$.
Note that if we take $s=t$ then this map is the commutator of the families taken for each time $t$, and we denote it by $[\psi_t, \varphi_t]:=[\psi(t),\varphi(s)]$.

\begin{lemma}\label{lem:bracket1}
Write $X = \frac{\partial}{\partial t}\Big|_{t=0}\phi_t$ and $Y = \frac{\partial}{\partial t}\Big|_{t=0}\psi_t$ and assume $\phi_0=\psi_0=Id$. Then the following statements hold:
\begin{itemize}
\item[$i)$] There exists a $2-$parametric family of diffeomorphisms $\varepsilon_{ts}=o(ts)$ such that \[\psi_t\circ\phi_s(x)=\varepsilon_{ts}\circ\varphi^{X+Y}_{ts}(x)\]

\item[$ii)$] $\frac{\partial^k}{\partial t^k}\Big|_{t,s=0}[\psi(t),\phi(s)](x) = 0$ for any $k\in\mathbb{N}$,
\item[$iii)$] $\frac{\partial^k}{\partial s^k}\Big|_{t,s=0}[\psi(t),\phi(s)](x) = 0$ for any $k\in\mathbb{N}$,
 \item[$iv)$] $\frac{1}{2}\frac{\partial^2}{\partial t \partial s}\Big|_{t=0}[\psi(t),\phi(s)](x) = [X,Y]$.
\item[$v)$]  There exists a $2-$parametric family of diffeomorphisms $\varepsilon(ts)=o(ts)$ such that \[[\phi_t,\psi_s] = \varepsilon_{ts}\circ\varphi^{[X,Y]}_{ts}.\]
\end{itemize}

\end{lemma}
\begin{proof}
Part $i)$ follows from Taylor's Remainder Theorem applied to the composition map $\psi_t\circ\phi_s(x)$, $ii)$ and $iii)$ are obvious, $iv)$ follows from the definition of Lie Bracket and $v)$ follows from an application of Taylor's Remainder Theorem together with $ii), iii)$ and $iv)$.
\end{proof}

\begin{remark}\label{rem:flow1}
A trivial but rather useful observation that we will eventually make use of is the following one. If $\phi_t$, $\psi_t$, are two flows such that for some $1-$parametric family of diffeomorphisms $\varepsilon_t$
\[\phi_t=\varepsilon_t\circ\psi_t\]
then $\varepsilon_t$ is also a flow since it is the composition of two flows $\varepsilon_t=\psi^{-1}_t\circ\phi_t$. Moreover, if $\varepsilon=o(t)$, then there exists another flow $\tilde{\varepsilon}=o(t)$ such that
\[\psi_t=\tilde{\varepsilon}_t\circ\phi_t.\]
For this last statement just note that
\[\phi_t=\varepsilon_t\circ\psi_t\Longrightarrow \phi_t\circ\psi_t^{-1}=o(t)\Longrightarrow \psi_t\circ\phi_t^{-1}=o(t)\] 
and thus there exists some flow $\tilde{\varepsilon}_t=o(t)$ such that $\psi_t=\tilde{\varepsilon}_t\circ\phi_t$.
\end{remark}

The following Lemma formalizes how errors inside a commutator of $1-$parametric families of diffeomorphisms can be taken out the bracket expressions when comparing to the flow of the respective brackets. 
\begin{lemma}\label{lem:erroresbracket}

Consider a flow $\varepsilon(t)=o(t)$ and write $X = \frac{\partial}{\partial t}\Big|_{t=0}\phi_t$, $Y = \frac{\partial}{\partial t}\Big|_{t=0}\psi_t$ and assume $\phi_0=\psi_0=Id$. Then there exists a $2-$parametric family of diffeomorphisms $\tilde{\varepsilon}(ts)=o(ts)$ such that

\[
[\varphi_s^X, \varepsilon_t\circ\varphi_t^Y]=\tilde{\varepsilon}_{ts}\circ\varphi^{[A,B]}_{ts}.
\]
\end{lemma}

\begin{proof}
The result follows from a direct application of points $i)$ and $v)$ from Lemma \ref{lem:bracket1}.
\end{proof}

\begin{remark}
In particular, if we take $s=t$ in this previous lemma, we get that $\tilde{\varepsilon}_{t^2}$ is an actual flow, since it is the composition of two flows as in Remark \ref{rem:flow1}, \[\tilde{\varepsilon}_{t^2}=\left(\varphi^{[A,B]}\right)^{-1}_{t^2}\circ[\varphi_t^X, \varepsilon_t\circ\varphi_t^Y]\]
The same remark is true for the family $\varepsilon_{ts}$ in $v)$ from Lemma \ref{lem:bracket1}.
\end{remark}

\begin{lemma}\label{lem:erroresSaltan}
Consider $X = \frac{\partial}{\partial t}\Big|_{t=0}\phi_t$, $Y = \frac{\partial}{\partial t}\Big|_{t=0}\psi_t$ and assume $\phi_0=\psi_0=Id$. Then, if  $\phi_t=\varepsilon_t\circ\psi_t$ for certain $\varepsilon_t=o(t)$, then there exists some other $\tilde{\varepsilon}_t=o(t)$ such that 
\[\phi_t=\psi_t\circ\tilde{\varepsilon}_t\] 
\end{lemma}

\begin{proof}
By Lemma \ref{lem:erroresbracket} there exists some $h_t=o(t)$ such that $h_{t^2}\circ\varepsilon_t^{-1}\circ\psi^{-1}_t\circ\varepsilon_t\circ\psi_t=Id$. So, we have that
$h_{t^2}\circ\varepsilon_t^{-1}\circ\psi^{-1}_t=\phi^{-1}_t$. The result follows from taking inverse flows at both side of the equation.
\end{proof}

\begin{lemma}
Write $X = \frac{\partial}{\partial t}\Big|_{t=0}\phi_t$ and $Y = \frac{\partial}{\partial t}\Big|_{t=0}\psi_t$ and assume $\phi_0=\psi_0=Id$. Then we have the following estimation:
\[
\phi_t^{-1}\circ\psi_s\circ\phi_t=\varepsilon_{ts}\circ\varphi^{X+t[X,Y]}_s.
\]
\end{lemma}

\begin{proof}
First, note that
\begin{align*}
\phi_t^{-1}\circ\psi_s\circ\phi_t=\psi_s\circ\left[\phi_t,\psi_s\right]=\psi_s\circ\tilde{\varepsilon}_{ts}\circ\varphi^{[X,Y]}_{ts}=\psi_s\circ\tilde{\varepsilon}_{ts}\circ\varphi^{t[X,Y]}_{s},
\end{align*}
where the second equality follows by Lemma \ref{lem:bracket1}. Nevertheless, this implies the existence of a $2-$parametric family of diffeomorphisms $\epsilon_{ts}=o(ts)$ such that $\psi_s\circ\varphi^{t[X,Y]}_{s}\circ\epsilon_{ts}$. But, from point $i)$ in Lemma \ref{lem:bracket1} applied to the composition $\psi_s\circ\varphi^{t[X,Y]}_{s}$, there exist some other $\tilde{\epsilon}_{ts}=o(ts), \varepsilon_{ts}=o(ts)$ such that 
\[\psi_s\circ\varphi^{t[X,Y]}_{s}\circ\epsilon_{ts}=\varphi_s^{Y+t[X,Y]}\circ\tilde{\epsilon}_{ts}=\varepsilon_{ts}\circ\varphi_s^{Y+t[X,Y]}\]
thus yielding the claim.

\end{proof}

\begin{definition}
Given two $1-$parametric families of diffeomorphisms $\varphi_s$, $\psi_t$, we define their $k-$th iterated commutator $\left[  \psi_t, \varphi_s \right]^{\#k}$ as follows
\[
\left[  \psi_t, \varphi_s \right]^{\#k}:=\left( \left(\varphi_{\frac{s}{\sqrt{k}}}^{-1}\right)\circ\left(\psi^{-1}_{\frac{t}{\sqrt{k}}}\right)\circ\left(\varphi_{\frac{s}{\sqrt{k}}}\right)\circ\left(\psi_{\frac{t}{\sqrt{k}}}\right)\right)^k.
\]
\end{definition}

The reason for parametrizing the flows by $\frac{s}{\sqrt{k}}$ in the definition of iterated commutator is justified by the following proposition.

\begin{proposition}\label{prop:bracketsiterados}
Write $X = \frac{\partial}{\partial t}\Big|_{t=0}\phi_t$ and $Y = \frac{\partial}{\partial t}\Big|_{t=0}\psi_t$ and assume $\phi_0=\psi_0=Id$. Then there exists $\varepsilon_{t}=o(t)$ such that
\[
\left[  \psi_t, \varphi_s \right]^{\#k}=\varepsilon_{ts}\circ\varphi^{[X,Y]}_{ts}
\]

\end{proposition}

\begin{proof}
By the definition of the $k-$th iterated commutator of flows and by Lemma \ref{lem:bracket1}, there exist flows $\varepsilon^1_t=o(t),\cdots,\varepsilon^k_t=o(t)$ such that
\[
\left[  \psi_t, \varphi_t \right]^{\#k}=\left( \varepsilon^1_{ts/k}\circ\varphi_{ts/k}^{[X,Y]}\right)\circ\cdots\circ\left( \varepsilon^k_{ts/k}\circ\varphi_{ts/k}^{[X,Y]}\right)\]
But by Lemma \ref{lem:erroresSaltan} there exist flows $\tilde{\varepsilon}^1_t=o(t),\cdots,\tilde{\varepsilon}^k_t=o(t)$ such that
\[\left( \varepsilon^1_{ts/k}\circ\varphi_{ts/k}^{[X,Y]}\right)\circ\cdots\circ\left( \varepsilon^k_{ts/k}\circ\varphi_{ts/k}^{[X,Y]}\right)=\tilde{\varepsilon}^1_{ts/k}\circ\cdots\circ\tilde{\varepsilon}^k_{ts/k}\circ\left(\varphi_{ts/k}^{[X,Y]}\right)^k
\]
and so the claim follows.
\end{proof}

With this battery of technical results at our disposal, we can compare how taking a given bracket expression behaves with respect to taking flows \cite[Theorem 1]{Mau}:
\begin{proposition}\label{prop:bracketk}
Let $X_1,X_2,\cdots,X_{\lambda}$ be (possibly repeated) vector fields on a manifold $M$. Then, for any bracket expression $A(-,\cdots,-)$ of length $\lambda$ there exists a flow $\varepsilon_t=o(t)$ such that
\[ A\left(\varphi_t^{X_1},\cdots,\varphi^{X_\lambda}_t\right) = \varepsilon_{t^\lambda}\circ\phi^{A(X_1,\cdots,X_\lambda)}_{t^\lambda}. \]
\end{proposition}
\begin{proof}
We proceed by induction on the length of the formal bracket--expression:
\begin{itemize}
\item For $k=2$ the result holds by Proposition \ref{prop:bracketsiterados}.
\item 

The Induction Hypothesis (IH) says that the statement holds for all expressions of length $k' < k$. By definition, if $A(-,\cdots,-)$ is an expression of length $k$, there exists $i < k$ and an integer $m$ such that $A(X_1,\cdots,X_k) = [B(X_1,\cdots,X_i), C(X_{i+1},\cdots,X_k)]^{\#m}$, with $B()$ of length $i$ and $C()$ of length $k-i$. Computing we see that there are flows $f_t=o(t)$ and $g_t=o(t)$ such that:
\begin{align}\label{align1}
A(\phi^{X_1}_t,\cdots,\phi^{X_k}_t) =& \left[B(\phi^{X_1}_t,\cdots,\phi^{X_i}_t), C(\phi^{X_{i+1}}_t,\cdots,\phi^{X_k}_t)\right]^{\#m} \\
\overset{\text{IH}}{=} & \left[f_{t^i}\circ\phi^{B(X_1,\cdots,X_i)}_{t^i}, g_{t^{k-i}}\circ\phi^{C(X_{i+1},\cdots,X_k)}_{t^{k-i}}\right]^{\#m}\nonumber
\end{align}
By an application of Proposition \ref{prop:bracketsiterados} first, and by Lemma \ref{lem:erroresbracket}, there exists a flow $\varepsilon_t=o(t)$ such that

\begin{align}\label{align2}
 \left[f_t\circ\phi^{B(X_1,\cdots,X_i)}_{t^i}, g_t\circ\phi^{C(X_{i+1},\cdots,X_k)}_{t^{k-i}}\right]^{\#m}=\varepsilon_{t^{k}}\circ\left[\phi^{B(X_1,\cdots,X_i)}_{t^i}, \phi^{C(X_{i+1},\cdots,X_k)}_{t^{k-i}}\right] 
\end{align}

But, since $\left[\phi^{B(X_1,\cdots,X_i)}_{t^i}, \phi^{C(X_{i+1},\cdots,X_k)}_{t^{k-i}}\right]=\phi^{A(X_1,\cdots,X_k)}_{t^k}$ the result follows from combining (\ref{align1}) and (\ref{align2}).
\end{itemize}
\end{proof}

\end{document}